\documentclass[10.6pt, reqno]{amsart}
\usepackage{xcolor}
\usepackage{caption}
\usepackage{graphicx}
\usepackage{mathtools}
\usepackage{subfigure}
\usepackage{slashed}
\usepackage{amsfonts}
\usepackage{amssymb}
\usepackage{accents}
\usepackage{amsmath}
\usepackage{amsxtra}
\usepackage{epstopdf}
\usepackage{latexsym}
\usepackage{mathrsfs}
\usepackage{leftidx}
\usepackage{tensor}
\usepackage{enumitem}
\usepackage{cleveref}

\newcommand{\abs}[1]{\lvert#1\rvert}
\def\rp#1{^{\!#1}}
\def\O{\mathscr{O}}
\def\B{\mathscr{B}}
\def\T{\mathcal{T}}
\def\ck{\check}

\def\bm{\textbf{m}}
\def\Ga{\Gamma}
\def\pa{\partial}
\def\Pp{\mathbb{P}}
\def\ep{\epsilon}
\def\ve{\varepsilon}
\def\f12{{\frac{1}{2}}}
\def\p {\partial}
\def\E{\mathcal{E}}

\def\I{\mathcal{I}}

\def\lF{\hat{F}}
\def\tF{\tilde{F}}
\def\ET{\T}
\def\tE{\tilde {\ud E}}
\def\tH{\tilde{\ud H}}
\def\hE{\hat E}
\def\hH{\hat H}
\newcommand{\Real}{\mathbb{R}}
\def\ud#1{\underline{#1}}
\def\D{\mathcal{D}}
\def\div{\text{div}\,}

\newcommand{\Lb}{{\underline{L}}}
\def\sD{\slashed{D}}
\def\sl#1{\slashed{#1}}
\newcommand{\sJ}{\slashed{J}}
\newcommand{\ab}{\underline{\a}}
\newcommand{\les}{\lesssim}
\def\L{\mathcal L}

\def\M{\mathcal{M}}
\def\a{\alpha}
\def\b{\beta}
\def\H{\mathcal{H}}

\def\c{\cdot}
\def\eb{\ud e}
\newcommand{\sn}{\slashed{\nabla}}
\newcommand{\BT}{\mathbb{T}}

\newcommand{\dn}{\Delta_0}

\def\nn{\nonumber}
\def\sJ{\mathscr{J}}
\def\sI{\mathscr{I}}
\def\ga{\gamma}

\def\jb#1{\langle{#1}\rangle}

\def\Nb{{\underline{N}}}
\def\nt#1{\tensor[^\natural]{#1}{}}

\newcommand{\bfe}{\emph{\textbf{e}}}

\newcommand{\curl}{\text{curl}\,}
\newcommand{\Tb}{{\underline{T}}}
\newcommand{\mub}{{\underline{\mu}}}

\def\si{\sigma}
\def\gb{\underline{g}}
\def\ib{{\underline{i}}}
\def\Hb{\underline{\H}}
\newtheorem{prop}{Proposition}
\newtheorem{lemma}{Lemma}

\newtheorem{corollary}{Corollary}
\newtheorem{theorem}{Theorem}
\newtheorem{remark}{Remark}
\numberwithin{equation}{section}
\title{Global solution for Massive Maxwell-Klein-Gordon equations with large Maxwell field}
\author{Allen Fang \and Qian Wang \and Shiwu Yang}

\AtEndDocument{\bigskip{\footnotesize%
  \addvspace{\medskipamount}
  \textsc{Laboratoire Jacques-Louis Lions (LJLL), Sorbonne Universit\'{e}, Paris, France} \par
  \textit{E-mail address}: \texttt{allen.fang@sorbonne-universite.fr} \par
  \addvspace{\medskipamount}
  \textsc{Oxford PDE Center, Mathematical Institute, University of Oxford, UK \&}\par
   \textsc{Mathematical Sciences Institute, Australian National University, Australia} \par
  \textit{E-mail address}: \texttt{qian.wang@@maths.ox.ac.uk, \, qian.wang@anu.edu.au}\par
  \addvspace{\medskipamount}
  \textsc{Beijing International Center for Mathematical Research, Peking University, Beijing, China} \par
  \textit{E-mail address}: \texttt{shiwuyang@math.pku.edu.cn}
}}
\date{}

\begin{document}

\begin{abstract}
We derive the global dynamic properties of the mMKG system (Maxwell coupled with a massive Klein-Gordon scalar field) with a general, unrestrictive class of data, in particular, for Maxwell field of  arbitrary size, and by a gauge independent method. Due to the critical slow decay expected for the Maxwell field, the scalar field exhibits a loss of decay at the causal infinities within an outgoing null cone. To overcome the difficulty caused by such loss in the energy propagation, we uncover a hidden cancellation contributed by the Maxwell equation, which enables us to obtain the sharp control of the Maxwell field under a rather low regularity assumption on data. Our method can be applied to other physical field equations, such as the Einstein equations for which a similar cancellation structure can be observed.
\end{abstract}

 	\maketitle

	\section{Introduction}\label{intro}
	
	In this paper, we study the global dynamics of solutions for the massive Maxwell-Klein-Gordon equations in $\Real^{3+1}$. To define the equations, let $A = A_\mu\,dx^\mu$ be a 1-form. The covariant derivative associated to this 1-form is
\begin{equation*}
		D_\mu = \partial_\mu + \sqrt{-1}A_\mu,
	\end{equation*}
	which can be viewed as a $U(1)$ connection on the complex line bundle over $\Real^{3+1}$ with the standard Minkowski metric $\bm_{\mu\nu} = \text{diag}\,(-1,1,1,1)$. The associated curvature 2-form $F$  is  given by
\begin{equation}\label{curv_2}
		F_{\mu\nu} = -i[D_\mu, D_\nu] = \partial_\mu A_\nu -\partial_\nu A_\mu = (dA)_{\mu\nu}.
	\end{equation}
In particular $F$ is a closed 2-form, that is, it satisfies the Bianchi identity
\begin{equation*}
		\partial_\gamma F_{\mu\nu} + \partial_\mu F_{\nu\gamma} + \partial_\nu F_{\gamma\nu} =0.
	\end{equation*}
	The massive Maxwell-Klein-Gordon equation (mMKG) system is a system for the connection field $A$ and the complex scalar field $\phi$: \begin{equation}\label{eqn::mMKG}\tag{mMKG}
\left\{
		\begin{array}{lll}
			&\partial^\nu F_{\mu\nu} = \Im(\phi\cdot\overline{D_\mu\phi}) = J_\mu[\phi],\\
			&D^\mu D_\mu \phi -m^2 \phi = \Box_A\phi-m^2\phi=0.
		\end{array}\right.
	\end{equation}
For simplicity we will normalize the mass to be $1$, that is, $m^2=1$.
The system is gauge invariant in the sense that given a solution $(A, \phi)$, $(A-d\chi, e^{i\chi}\phi)$ solves the same equation system for any potential function $\chi$.
	
	In this paper we consider the Cauchy problem for \eqref{eqn::mMKG} with initial data \begin{footnote}
{If a covariant component of  a tensor field is contracted by $\p_t$, the component is denoted by $0$.}
\end{footnote}
 \begin{equation*}
		F_{0i}(t_0,x) = E_i(x), \quad \tensor[^\star]{F}{_{0i}}(t_0, x)=H_i(x),\quad \phi(t_0, x)=\phi_0,\quad D_0\phi(t_0,x)=\phi_1,
	\end{equation*}
	where ${}^\star F$, the Hodge dual of the 2-form $F$, is defined by ${}^\star F_{\mu\nu}=\f12 \tensor{\ep}{_{\mu\nu}^{\ga\delta}}F_{\ga\delta}$,\begin{footnote}
 {We will frequently adopt the Einstein summation convention in tensor calculus, and use the Minkowski metric to lower and lift the indices of a tensor field, unless specified otherwise. We also fix the convention for the range of indices in summation that the Greek letters $\mu, \nu, \ga,\cdots=0,1,2,3,$ and  $A, B, C=1,2$. As indices in a summation $i,j,k,l,\ell=1,2,3$, otherwise they are considered to be nonnegative integers.}
\end{footnote} with $\ep$ the volume form of the Minkowski space.
We say that the initial data is  admissible if the following compatibility condition is satisfied: \begin{equation*}
		\div E = \Im(\phi_0\cdot\overline{\phi_1}),\quad \div H=0
	\end{equation*}with the divergence taken on the initial hypersurface $\Real^3$. For solutions of \eqref{eqn::mMKG},  the total charge\begin{equation}\label{1.07.2.19}
		q_0 = \frac{1}{4\pi}\int_{\Real^3}\Im(\phi\cdot \overline{D_0\phi})\,dx = \frac{1}{4\pi}\int_{\Real^3}\div E\,dx
	\end{equation} is conserved.
	
	The global well-posedness result of  \eqref{eqn::mMKG} was established by the pioneering work  \cite{Eardley1982,Eardley1982a} of Eardley-Moncrief  for sufficiently smooth initial data in $\Real^{3+1}$. For the massless case of \eqref{eqn::mMKG} (where $m=0$), by introducing bilinear estimates on null forms,  Klainerman-Machedon \cite{Klainerman1994} extended the global regularity result to the more general data $(\phi, d A)$  with merely bounded energy.  There has been extensive attention drawn by  the asymptotic behavior of smooth solutions to the  massless MKG equations in $\Real^{3+1}$.  The global asymptotic behavior for small data solutions was given by Lindblad-Sterbenz \cite{LindbladMKG} by combining the standard commuting vector field approach with conformal methods. The improved result of this which allows large initial Maxwell data was established by Yang in \cite{Yang2015} with the help of the  multiplier method due to Dafermos-Rodnianski \cite{Dafermos2009}. It was further improved by Yang-Yu in \cite{Yang2018}, by adapting the canonical conformal compactification and rekindling the trick of using Kirchhoff formula in Eardley-Moncrief. This result gives a full depiction of the global asymptotic behavior  for large initial data in $\Real^{3+1}$.

The key techniques leading to the result of Yang-Yu \cite{Yang2018} are based on the important fact that the two equations of the massless MKG equations are both scaling invariant. This feature allows a variety  of combinations of physical methods for nonlinear wave equations. With the presence of the nontrivial mass,  the scalar field equation is no longer scaling invariant, so is the whole system. This basic fact significantly restricts the physical methods for treating the system, compared with those for the massless case. Moreover, the Maxwell field  plays the role of the background geometry in the nonlinear analysis, with the optical decay rate merely  $(r+t)^{-1}$ in the entire spacetime, similar to a free wave solution. This makes the nonlinear analysis for  the wave Klein-Gordon system much more challenging than for a pure Klein-Gordon system (c.f. \cite{Kl:KG:85}).
	
As a consequence, the decay properties for the massive MKG \eqref{eqn::mMKG} are much less understood.  By using the hyperboloidal method introduced by Klainerman in \cite{Kl:KG:85} \footnote{A similar result of \cite{Kl:KG:85}  based on Fourier methods  and renormalization   is due to   J. Shatah
 \cite{shatah}.  Note that  \cite{shatah}  does not require   restrictions on the data. }, a set of quantitative decay estimates was established in  \cite{Psarelli1999} by Psarelli in the interior of a forward light cone with essentially compactly supported small data. This is certainly  restrictive and unsatisfactory. The restrictions were rooted in the hyperboloidal method.  The method is actually a robust tool to tackle the non-scaling-invariant issue of the Klein-Gordon equations. Nevertheless, the family of hyperboloids for carrying out the fundamental analysis can not be constructed outside of the outgoing cone, which means it does not cover the full spacetime.  To complement this physical method, in \cite{KWY} Klainerman-Wang-Yang removed the restriction to the compactly supported data of scalar fields by developing a neat hierarchic multiplier regime, and thus gave decay properties of solutions of (mMKG) in the exterior of the light cone, provided that the generic data set is sufficiently small on ${\mathbb R}^3\setminus B_R$.

 In this paper, we consider a  general class of data on the full initial slice ${\mathbb R}^3$, or identically  $\{t=t_0=2R\}$, ($R=1$ without loss of generality). By merely assuming smallness for the data of scalar field in terms of the standard weighted Sobolev norm, our assumption allows the Maxwell field to have finite weighted energy initially, without any restriction on the size. We will give a full account of the asymptotic behavior of solutions in the entire spacetime, which is divided by the outgoing lightcone $C_0=\{t-t_0=r-R, t\ge t_0\}$ into  interior and exterior parts.
\subsection{The Scheme of Reduction}
Based on the principle of divide and conquer, we first reduce the large data problem to a perturbation problem, and then  separate the energy estimates in the interior of the cone $C_0$ from the energy propagation  in its exterior. We now outline the scheme of the reduction.

{\bf Step 1.} {\it Reduction from the large data problem to a perturbation problem.}

One difference that makes  our result stronger than a standard global nonlinear stability result for Maxwell-Klein-Gordon equations lies in that we do not assume smallness on the initial weighted energy of the Maxwell field. Similar to the work of \cite{Yang2015} for the massless MKG, to treat the Maxwell field,
we start with  decomposing
\begin{equation}\label{eq:decomposition4F}
F=\tF+\hat F,
\end{equation}
where the linear part $\hat F$  and the perturbation part $\tF$ verify
\begin{equation}\label{eq:eq4lFandNLF}
\p^\nu \hat F_{\mu\nu}=0,\qquad \partial^\nu \tF_{\mu\nu} = J_\mu[\phi]
	\end{equation}
with initial data
\begin{align*}
 &\lF(t_0, x)=E^{df}(x),\quad {}^{\star}\lF(t_0, x)=H(x), \\
 &\tF(t_0, x)=E^{cf}(x),\quad {}^{\star}\tF(t_0, x)=0, \quad \phi(t_0, x)=\phi_0,\quad D_0 \phi(t_0, x)=\phi_1,\quad
\end{align*}
where  $E(t_0,x)$ is decomposed into the divergence-free part and curl-free part $E=E^{df}+E^{cf}$ such that
\[
\div{E^{df}}=0,\quad \curl{E^{cf}}=0,\quad \div{E^{cf}}=\Im(\phi_0\cdot\overline{\phi_1}).
\]
Due to  the elliptic theory of the Hodge system, $E^{cf}$ is uniquely determined by $\phi_0$ and $\phi_1$. We thus can freely assign $E^{df}$, $H$ and $\phi_0$, $\phi_1$ as long as $E^{df}$ and $H$ are  divergence free.

Let $1<\ga_0<2$ be fixed.  We define the weighted energy norms
\begin{align}
\M_{k, \ga_0}&= \sum\limits_{l\leq k}\int_{\Sigma_{t_0}}(1+r)^{\ga_0+2l}(|\bar\nabla^l E^{df}|^2+|\bar\nabla^l H|^2)dx, \label{IDMKG}\\
\E_{k,\ga_0}&=\sum\limits_{l\leq k}\int_{\Sigma_{t_0}}(1+r)^{\ga_0+2l}(|\bar D\bar D^l\phi_0|^2+|\bar D^l\phi_1|^2+|\bar D^l\phi_0|^2)dx,\label{data_scf}
\end{align}
where $\bar D$ is the induced covariant derivative on $\Sigma_{t_0}$ with respect to the connection $A$, and $\bar\nabla$ denotes the spatial derivative on $\Sigma_{t_0}$.

Throughout this paper, we assume
\begin{equation}\label{asmp}
\M_{2,\ga_0}\le \M_0^2, \quad \E_{2,\ga_0}\le 1,
\end{equation}
where $\M_0>0 $ can be arbitrary, and the bound of $\E_{2,\ga_0}$ will be further assumed to be sufficiently small.   Without loss of generality, we may assume that $\M_0\ge 1$.

{\bf Goal}: For any admissible data with $\M_{2,\ga_0}<\infty$, $\hat F$ is uniquely determined  directly due to the linear equation. Hence we can easily  obtain its asymptotic behavior.  Our aim is to show  the following global nonlinear stability result of any such $\hat F$,
 \begin{theorem}\label{thm1}
Let $0<\delta\ll \frac{1}{3} $ be fixed and arbitrarily small. Suppose that the admissible data $(E, H, \phi, \phi_1)$  verify (\ref{asmp}) on the hyperplane $\Sigma_{t_0}:=\{t=t_0\}$ with $t_0=2R$.
 There exists $\ve_0>0$ sufficiently small, depending on $\M_0$, $\ga_0$ and $\delta$, such that, if $\E_{2,\ga_0}\le \ve_0^2$ then
	 there exists a unique (up to gauge transformation) global smooth solution of \eqref{eqn::mMKG} in $(\Real^{3+1}, \bm)$ for $t\ge t_0$.
Moreover, it satisfies the uniform decay estimates  in $\{t-t_0\ge r-R, t\ge t_0\}$ with $R=1$ (without loss of generality)
 \begin{align*}
|\phi(t,x)|&\les \ve_0\jb{\tau_+}^{-\frac{3}{2}}\jb{\tau}^{\delta}, \quad |\tF(t,x)|\les \ve_0\jb{\tau_+}^{-1}\jb{\tau_-}^{-1},
\end{align*}
where $\tau_+=t+r, \tau_-=t-r, \tau=(\tau_+\tau_{-})^\f12$, and $\jb{f}=1+|f|$ for real-valued scalar functions $f$. The constants \begin{footnote}{ $A\les B$ stands for $A\le C B$,  with $C>0$ a fixed constant. If $A\les B$ and $B\les A$, we denote it as $A\approx B$ for short. }\end{footnote} in the bounds depend on $\M_0$, $\ga_0$ and $\delta$. In the exterior region $\{t-t_0\le r-R, t\ge t_0\}$, the chargeless part of $\tF$ and $\phi$ have similar decay properties as in \cite{KWY}.
\end{theorem}
 The complete statement of the main result will be given in Theorem \ref{ext_stb} and Theorem \ref{them::mainThm}  in Section \ref{1.08.1.19} with the linear behavior of $\hat F$ stated in Lemma \ref{lemma:decay:lF} and Proposition \ref{linear_ex}.

{\bf Step 2.} {\it Control the energy propagation in the exterior region.}

 In \cite{KWY}, we proved the exterior stability result provided that the weighted Sobolev norm of the initial data is sufficiently small, without any restriction on the size of the charge $|q_0|$ of (mMKG). This result is completely independent of the data on $\{|x|\le R, t=t_0\}$.  In order to prove Theorem \ref{thm1}, we first extend in the exterior region  the global well-posedness result of \cite{KWY} for the admissible data verifying the assumption in Theorem \ref{thm1}, particularly incorporating large Maxwell field. The result is given in Theorem \ref{ext_stb} and proven in Section \ref{ext}. Note that by following the method in \cite{KWY}, we can immediately obtain the linear behavior in Proposition \ref{linear_ex} for $\hat F$ in the exterior region. Due to the smallness assumption on the scalar field and (\ref{1.07.2.19}),  the charge verifies
\begin{equation}\label{10.30.5.18}
|q_0|\les \|\phi_0\|_{H^1(\Real^3)}\|\phi_1\|_{L^2(\Real^3)}\le \ve_0^2.
\end{equation}
Consequently $|q_0|$ can be  made sufficiently small.

With the help of the bounds on $\hat F$ and $|q_0|$, in Section \ref{ext},  the global asymptotic behavior of the chargeless part of $\tF$ and $\phi$ in the exterior region is derived  completely  independent of the  propagation of the solution in the interior of the cone $C_0$ and independent of any particular choice of the gauge. 
The result in Theorem \ref{ext_stb} supplies us with the control of the boundary fluxes on $C_0$ (see details in Proposition \ref{11.8.1.18}), which plays a fundamental role for bounding the interior energies.

The exterior result can be adapted to allow $|q_0|$  to be arbitrary, which is not of interest for proving Theorem \ref{thm1}.

 {\bf Step 3.} {\it Construct the global solution in the interior of the cone $C_0$.}

As the main building block for proving  Theorem \ref{thm1}, we control the energy propagation  in the interior of $C_0$  in three steps, with the conclusion given in the main theorem, Theorem \ref{them::mainThm}.

(1)  We first derive the asymptotic behavior for $\hat F$ within $C_0$ for $t\ge t_0$, by  applying the multiplier introduced in \cite{LindbladMKG} for the massless MKG equations, $\tau_{-}^p\Lb+\tau_+^p L, \, 1<p<2$,  to obtain the weighted energies for $\hat F$, where
\begin{equation}\label{1.07.6.19}
\tau_{-}=t-r,\,\, \tau_+=t+r, \,\, L=\p_t+\p_r, \,\,  \Lb=\p_t-\p_r.
\end{equation}
This result is given in Lemma \ref{lemma:decay:lF}.

(2) Let $\tau_*>\tau_0=\sqrt{3} R$ be a fixed number.  To prove  the main theorem,  Theorem \ref{them::mainThm}, in the interior of $C_0$  for $(\phi, \tF)$, we rely on the bootstrap argument in the region
\begin{equation}\label{1.07.3.19}
\D^{\tau_*}=\bigcup_{\tau_0\le \tau'\le \tau_*} \H_{\tau'},
\end{equation}
 with the truncated hyperboloid  $\H_\tau:= \{(t,x):\sqrt{t^2-|x|^2}=\tau, t-t_0\ge r-R\}$ for $\tau\ge \tau_0$.
 We will improve the set of assumptions (\ref{BA}) on energies of  $(\phi, \tF)$.

 Since the initial data is given at $\{t=t_0=2R\}$, we need to  determine the initial energies  on $\H_{\tau_0}$.  Moreover  the energies of $(\phi, \tF)$ are derived by using  the divergence theorem in $\D^{\tau_*}$, which requires the corresponding  energy fluxes of $(\phi,\tF)$ on the null boundary $C_0\cap \{\tau_0\le \tau\le \tau_*\}$.  The boundary fluxes are controlled by using Theorem \ref{ext_stb} and the initial energy replies on a local extension of solution and local energy estimates in $\{t-t_0\ge r-R, t_0\ge t\ge R\}.$ The bounds of the initial energies  and boundary energy fluxes are given in Proposition \ref{11.8.1.18}.


(3) The last step, which is the main step, is to improve the bootstrap assumptions by  controlling the long time energy propagation in $\D^{\tau_*}$  via the hyperboloidal method and using the bounds of the boundary fluxes and initial energies derived in (2) (Proposition \ref{11.8.1.18}).

 The global result in Theorem \ref{them::mainThm} is then proved by the continuity argument in the entire interior region $$\D^+=\bigcup_{\tau'\ge \tau_0} \H_{\tau'}.$$
 One can refer to Figure \ref{fig1} for the division of the spacetime in the reduction scheme.
 \begin{figure}[ht]
\centering
\includegraphics[width = 0.8\textwidth]{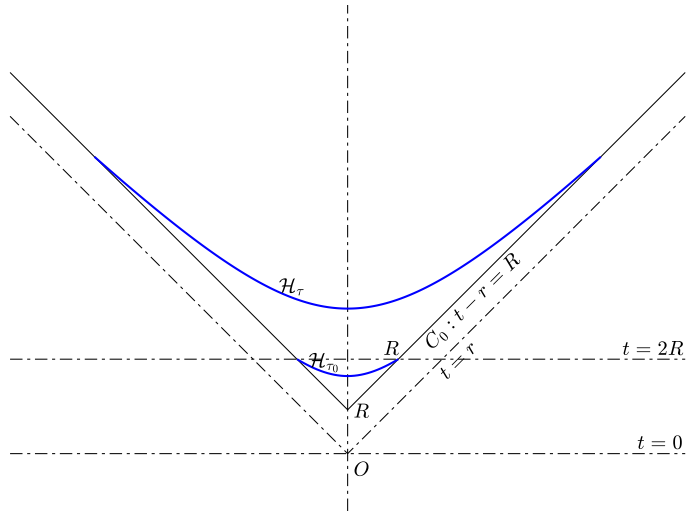}
 \vskip -0.8cm
  \caption{Illustration of the division of the spacetime}\label{fig1}
\end{figure}

   \subsection{Strategy of the proof in the interior region}
   Next, we  explain the  main ingredient of our proof for improving the bootstrap assumptions in $\D^{\tau_*}$.

In  the only work \cite{Psarelli1999} which contained the decay properties of the solutions of (mMKG) in the interior of $C_0$,  the method relied heavily on fixing the Cronst\"{o}m gauge.  The disadvantages  of choosing this gauge  lie in that
  \begin{itemize}
 \item This gauge is defined by a transport equation along any timelike geodesic initiated from the origin, and hence is only well-defined in the interior of a lightcone, i.e. $\{t> r\}$. This only works if the data of the scalar field has compact support.
      \item The control of the gauge  relies on a transport equation. This costs one more derivative than allowed by our initial data.
  \end{itemize}
In order  to cope with the general class of the initial data, we need to develop a gauge independent method to directly take advantage of the exterior result established in Step 2.

However, the reason for using the Cronstr\"{o}m gauge in  \cite{Psarelli1999}  is deeply connected to the non-scaling-invariant issue of the KG equation and the slow decay of the Maxwell field. The scaling vector field $S=x^\mu\p_\mu$ can not be used as the commuting vector field for the linear KG equation which leads to the weak decay property of  the solution along $S$  compared with the directions  tangential to hyperboloids  $\H_\tau$.  In \cite{Psarelli1999}, the use of Cronstr\"{o}m gauge,  with which the Lie derivative of the scalar field therein is introduced, was suggested by D. Christodoulou to cancel the weak decay component of the scalar field in the nonlinear analysis. This was the key to accomplish the nonlinear analysis and to give the sharp energy control on the scalar field.  (See \cite[Page 6, 39]{Psarelli1999}.) Even using the sharp control of the scalar field, the outcome of \cite{Psarelli1999} still loses the sharp decay for the Maxwell field. Surprisingly, our gauge-independent proof obtains the sharp decay of the Maxwell field without requiring the sharp bound on the scalar field. It is miraculously achieved by exploiting a hidden cancellation structure contributed  by the Maxwell equations.


As explained, to treat KG equations, one usually follows the hyperboloidal method of  Klainerman in  \cite{Kl:KG:85}, which uses the generators of the Lorentz group, $\Omega=\{\Omega_{\mu\nu}=x_\mu \p_\nu-x_\nu \p_\mu, 0\le \mu<\nu\le 3\}$,\begin{footnote}
  {We fix the convention that $x^0=t$ and $x_\mu=\textbf{m}_{\mu\nu} x^\nu$.}
\end{footnote}
 to derive the  weighted energy on hyperboloids in the interior of the cone $C_0$ and consequently the pointwise decay of the solution  by using the Sobolev inequalities on hyperboloids.

To illustrate our treatment for (mMKG) particularly with large Maxwell field, low regularity data, and without fixing a gauge, we will take the estimate of the first  order energy  as an example. For convenience, we denote the energy by  $\E^{T_0}[f](\tau)=\int_{\H_\tau}\T[ f,0](T_0, \Tb)(\H_\tau) d\H_{\tau}$, with $f=D_{\Omega_{0a}}\phi$, \, $a=1,2,3$, and with $\T$ denoting the standard stress energy tensor of (mMKG). (See the definition of $\T$ in Section 2, and  the estimation in Section \ref{scal_1} Step 1.) It requires the bound on
$$\iint_{\D^\tau} |(\Box_A-1) D_{\Omega_{0a}}\phi||D_{T_0} D_{\Omega_{0a}}\phi| dx dt$$ with $\D^\tau $
 similarly defined as in (\ref{1.07.3.19}) for $\tau_0<\tau\le \tau_*$. For the commutator (see Lemma \ref{lem::Commutator1})
$$
Q(F, \phi, \Omega_{0a}):=[\Box_A, D_{\Omega_{0a}}]\phi=(\Box_A-1) D_{\Omega_{0a}}\phi, \quad a=1,2,3
$$
we can calculate that
\begin{equation}\label{11.14.4.17}
|Q(F, \phi, Z)|\les |F_{Z\mu}D^\mu\phi|+\cdots
\end{equation}
\begin{footnote}{We often hide better terms, which have no cancelation with the leading term, by using the notation ``$\cdots$". }\end{footnote}where $Z$ is the commuting vector field $\Omega_{0a}$ with $a=1,2,3$ here and in the sequel.

    With $\Tb =\tau^{-1} S$ the future directed unit normal of $\H_\tau$ in $({\mathbb R}^{3+1}, \textbf{m})$, the component of Maxwell field  $F_{Z\Tb}$ which is paired with the weak decay component $D_\Tb \phi$, is expected to satisfy
\begin{equation*}
|F_{Z \Tb}|\les \tau^{-1},
\end{equation*}
as a result of boundedness of the weighted energy by applying $S$ as the multiplier to the Maxwell equation. Substituting this decay into (\ref{11.14.4.17}) gives
\begin{equation}\label{11.14.5.18}
|Q(F, \phi, Z)|\les |F_{Z\mu}D^\mu\phi|\les \tau^{-1}|D_\Tb \phi|+\cdots.
\end{equation}
To control this term we rely on the decomposition of vector fields in (\ref{dcp_3}), the property of the boost vector fields and $r\le t$ in $\D^{\tau^*}$ to derive
\begin{equation}\label{1.09.1.19}
|D_\Tb \phi|\les \frac{\tau}{t}| D_{T_0} \phi|+\tau^{-1}\sum_{a=1}^3|D_{\Omega_{0a}}\phi|, \mbox{ where }T_0=\p_t.
\end{equation}
This implies \begin{footnote}{We may hide the area element of the integral in $\D^\tau$ for convenience.}\end{footnote}
\begin{equation}\label{1.07.5.19}
\begin{split}
\iint_{\D^\tau}|F_{Z\Tb} D^{\Tb}\phi||D_{T_0} D_{\Omega_{0a}}\phi| &\les\int_{\tau_0}^\tau (s^{-1} \E^{T_0}[\phi]^\f12(s)\\
&+s^{-2}\sum_{b=1}^3\E^{T_0}[D_{\Omega_{0b}}\phi]^\f12(s)) \E^{T_0}[D_{\Omega_{0a}}\phi]^\f12(s) ds.
\end{split}
\end{equation}
Thus, due to the weak decay rate $s^{-1}$, $\E^{T_0}[D_{\Omega_{0a}}\phi](\tau)$ can not be bounded  uniformly in $\tau$.

Thanks to the linearity, we can adopt  the multiplier $\tau_{-}^{1+2\ep}\Lb+\tau_+^{1+2\ep} L$ with $\ep>0$ to derive the weighted energy for $\hat F$, which implies
\begin{equation*}
|\hat F_{Z\Tb}|\les \tau^{-1}\tau_{-}^{-\ep}.
\end{equation*}
The same trick does not work for $\tF$, since applying the same multiplier to the Maxwell equation of $\tF$ requires more decay from the scalar field  for the term $\L_Z\rp{\le 2} J[\phi]_S$ \begin{footnote}{$\L_Z\rp{\le k} $ stands for $\L_Z\rp{l} , 0\le l\le k$. The same interpretation applies to ``$\le k$" in the notations $D_Z\rp{\le k} f$, $\E_{\le k}^X[f, G](\Sigma)$, etc. in later sections.}\end{footnote} than the standard linear behavior of a Klein-Gordon solution.
 This is one of  the main differences between the massive MKG and the massless (or linear) MKG.

In order to transform the $\tau_{-}^{-\ep}$ into the stronger decay $\tau^{-\ep}$, we need the bound
\begin{equation*}
\|(\frac{t}{\tau})^\f12 D\phi\|_{L^2(\H_\tau)}\le \E^{T_0}[D_\p \phi]^\f12(\tau).
\end{equation*}
Thus if  $\{\p_\mu, \mu=0, \cdots 3\}$ are also employed as  commuting vector fields, we can  gain the weight of $\frac{t}{\tau}$ on the left hand in the above inequality,  which leads to
\begin{align}
\iint_{\D^\tau}| \hat F_{Z\Tb}D^\Tb \phi||D_{T_0}D_{\Omega_{0a}}\phi|\les\int_{\tau_0}^{\tau} s^{-1-\ep}\sum_{Y\in \Pp}\E^{T_0}[D_Y \phi]^\f12(s)\E^{T_0}[D_{\Omega_{0a}}\phi]^\f12(s) ds,\label{1.10.1.19}
\end{align}
 where  $ \Pp=\Omega\cup\{\p_\mu\}_{\mu=0}^3 $ consists of the generators of Poincar\'{e} group.  Using all the elements in $\Pp$ as the commuting vector fields, the above term can then be treated by Gronwall's inequality. This procedure does not require smallness of the initial energy of $\hat F$.

 Nevertheless,  with $\ve_0$  the bound of the data of scalar field,  by assuming the sharp decay
  \begin{equation}\label{1.07.4.19}
  |\tF_{\Omega_{0a}\Tb}|\les \ve_0 \tau^{-1},\quad a=1,2,3,
   \end{equation}
 the corresponding part in (\ref{1.07.5.19}) still leads to a growth of $\ln(\tau+1)$  in the first order energy for the scalar field. Such loss is unavoidable for (mMKG).

  On the other hand, by using the vector field $X$ as the multiplier for bounding the weighted energy of $\L_Z^k\tF$, with $\L_Z^k :=\L_{Z_1}\cdots \L_{Z_k}$, and $Z_1, \cdots, Z_k\in \Pp$,  the main nonlinear term is (see  (\ref{10.27.13.18}) in Section \ref{maxwell_1}),
  \begin{equation}\label{I1}
  I=\iint_{\D^\tau}\L_Z^k J[\phi]^\mu \L_Z^k \tF_{X\mu}.
  \end{equation}
   Due to the weak decay of $\L_Z^k J[\phi]_\Tb$, the best possible choice is to set $X=S$ so that the first factor can be fully evaluated by vector fields tangential to $\H_\tau$, thus having much better asymptotic behavior than the component of $\Tb$.

  A rough treatment in \cite{Psarelli1999} of   $I$ leads to the growth of energies for Maxwell field, this means (\ref{1.07.4.19}) was not achieved. Together with the loss due to (\ref{1.07.5.19}), we encounter a double-loss situation, that is, neither the energies for the scalar field are  expected to be uniformly bounded in $\tau$, nor are the energies for Maxwell field if treated as in \cite{Psarelli1999}. The principle of closing the bootstrap argument is to achieve sharp decay  for at least one quantity involved in the nonlinear analysis. We need to add two more orders of derivatives to the initial data in order to improve the decay of the scalar field by using $L^\infty-L^\infty$ estimates (see \cite{Sungjin}), which exceeds the regularity  of the result in \cite{Psarelli1999} by one order.  Thus under our assumption, the gauge independent method has a fundamental difficulty in closing the bootstrap argument for the energy propagation of $\phi, \tF$.

The way in  \cite{Psarelli1999}  to solve the issue of double-loss  is to cancel the bad term in the energy estimate by fixing the Cronstr\"{o}m gauge.  By considering the energy of $\L_{\Omega_{0a}}\phi$ with $\L$ the Lie derivative in \cite{Psarelli1999}, the counterpart of (\ref{1.07.5.19}) in the first order energy estimate for $\phi$ takes the form
\begin{equation*}
\iint_{\D^\tau}|\L_{\Omega_{0a}} A^\mu D_\mu \L_{\Omega_{0a}} \phi || D_{T_0} \L_{\Omega_{0a}} \phi|+\cdots.
\end{equation*}
The Cronstr\"{o}m gauge is chosen so that $A_S=0$ and consequently $\L_{\Omega_{0a}}A_S=0$, which eliminates the bad component  $D_\Tb \L_{\Omega_{0a}} \phi$ immediately. This allows the energy  of $\L_{\Omega_{0a}}\phi$ to be bounded without loss in \cite{Psarelli1999}.

Our method  is based on a new observation on the trilinear structure in $I$ in (\ref{I1}).  For bounding the top (second) order weighted energy of $\L_Z^2 \tF$, with $Z_1, Z_2\in \Pp$, we uncover  a structure on the following borderline term in $I$
\begin{equation*}
I=\iint_{\D^\tau} S^\nu\Im(\phi\cdot\overline{D^\mu D_Z^2 \phi}) (\L_Z^2 \tF)_{\mu\nu}+\cdots,
\end{equation*}
where for the better terms of $\L_Z^2 J^\mu[\phi]$ we refer to Lemma \ref{lem::Commutator1}. The integrand is bounded by
\begin{equation*}
|S^\nu\Im(\phi\cdot\overline{D^\mu D_Z^2 \phi}) (\L_Z^2 \tF)_{\mu\nu}|\les |\phi||D_\Omega D_Z^2 \phi||\L_Z^2 \tF|_h
\end{equation*}
where the Riemannian metric $h_{\mu\nu}:=\textbf{m}_{\mu\nu}+2\Tb_\mu\Tb_\nu$, and $\Omega$ represents all the elements in $\{\Omega_{\mu\nu},0\le  \mu<\nu\le 3\}$. (See (\ref{eq:keyObservation})).
Even assuming the sharp bound $|\phi|\les \ve_0\tau_+^{-\frac{3}{2}}$, the term on the right hand side still leads to a growth of $\ln (\tau+1)$ in the energy of the Maxwell field $\L_Z^2 F$.

We observe that the term  $\Im(\phi\cdot\overline{D^\mu D_Z^2 \phi})$ exhibits a typical low-high interaction in terms of the order of derivatives.
Our approach is to carry out the trilinear analysis by virtue of integration by parts, which passes the third derivative from the highest order derivative term  to the lower order ones  in $I$.   Thus  bounding the top order energy for the Maxwell equations is reduced to controlling
\begin{align*}
|I|&\les \iint_{\D^\tau} |S^\nu \Im(\phi\cdot \overline{D_Z^2 \phi})\p^\mu (\L_Z^2 \tF)_{\mu\nu}|+|S^\nu \Im(D^\mu \phi\cdot \overline{D^2_Z \phi})(\L_Z^2 \tF)_{\mu\nu}|\\
&\qquad+|\mbox{ boundary terms }|+\cdots.
\end{align*}
In the first term on the right, $\p^\mu (\L_Z^2 \tF)_{\mu\nu}$ seems to have one more derivative than allowed by the data. Fortunately the divergence form  allows us to take advantage of  the Maxwell equation of $\tF$ to reduce it to $-\L^2_Z J[\phi]_\nu$. Thus the  integrand becomes quartic  in terms of the scalar field and its derivatives, and  exhibits sufficient decay. The second term involves only the tangential components of $D^\mu \phi$ to the hyperboloids,  and hence can be controlled with the help of the weighted Sobolev inequality.

This treatment gives  sharp bound of energies and the decay (\ref{1.07.4.19}) for the Maxwell field $\tF$,  and allows the energies for the scalar field to have  a small controllable growth. \begin{footnote}{This treatment was mentioned in \cite[Section 1]{Wang2016} and plays a crucial role in \cite{Wang2019p} for Einstein equations with massive scalar field under the maximal foliation gauge.}\end{footnote}	

\subsection{Remarks on the technicalities}

{\bf (i)}  In the region $\D^{\tau_*}$, the energies of the Maxwell field induced by the multiplier of $S$  take a homogeneous form by using   the hyperboloidal orthonormal frames. (See the definition of the frames in Section 2 and the form  of energy density in (\ref{11.13.3.18}).) Thus it is much simpler to use the hyperboloidal orthonormal frames  to determine the asymptotic behavior for the Maxwell field than using the null tetrad. Moreover, by using the  hyperboloidal frames and taking advantage of the simple fact $\L_\Omega \Tb=0$, the comparison argument, which is used to obtain the bounds in the  $H^2\hookrightarrow L^\infty $ weighted Sobolev inequality from the weighted energies, is much simpler than \cite{Psarelli1999}. The latter  relied only on the null tetrad, which followed the same calculations as in \cite[Chapter 7]{CK}.
  Since $\L_{\Omega_{0a}} X$ with $X=L, \Lb$,  exhibits various nondegenerate asymptotic behaviors, it requires  involved analysis even in the Minkowski space. (See \cite[Section 4.3]{Psarelli1999}, which contains incompleteness and minor error. One can actually sample the flavor of such calculation from our Lemma \ref{12.13.1.17}.)

  Moreover, since we use $T_0$ and $S$ as multipliers for the energies of the scalar field and Maxwell field respectively, and the set $\Pp$ as commuting vector fields,  the Maxwell field  involved in the error integral is constantly evaluated by both  the Cartesian frames and hyperboloidal frames. In view of (\ref{11.16.1.18}) and (\ref{dcp_3}), each transformation between these two frames leads to a loss of the weight of $\frac{t}{\tau},$ which is $\les t^\f12$ in $\D^{\tau_*}$. As seen in (\ref{1.10.1.19}), such weight is crucial for transforming the additional weak decay of  $\hat F$ in $\tau_{-}^{-\ep}$ to the stronger decay $\tau^{-\ep}$.   Hence it is important not to lose the weight in the analysis. This is mainly achieved by taking advantages of the null conditions, the anti-symmetric structure of $2$-forms, and the decomposition of vector fields such as (\ref{1.09.1.19}). We systematize these calculations in Lemma \ref{11.13.5.18}. It vastly simplifies the nonlinear analysis.

  {\bf (ii)}  We briefly explain the strategy of the analysis in the region $\{t-t_0\le r-R, t\ge t_0\}$.

	Recall from (\ref{eq:decomposition4F}) that the full Maxwell field $F$  is decomposed into the linear part $\hat{F}$ and the perturbation part $\tF$. The latter will be further decomposed into the small charge and chargeless part in (\ref{12.31.2.18}). The main result in \cite{KWY} addressed the case when $\hat{F}=0$ but with large charge $q_0$.
  The charge part is the simplest nontrivial linear solution of Maxwell equations in the exterior region. In this sense, both Theorem \ref{ext_stb} and the main result in \cite{KWY} are dealing with the global exterior stability of large linear Maxwell fields. However, the charge solution has the simplest  special form, while, in Theorem \ref{ext_stb}, the large linear field $\hat F$ takes the general form. It is easily seen  as a direct consequence of \cite{KWY} that  some component of $\hat F$  still exhibits the typical critical decay similar to the free wave solution.  Different from the case in $\D^+$, the generic data assumption in Theorem \ref{thm1} does not give better  decay properties for $\hat F$ than for $\tF$ in the exterior region. Since $\hat F$  has no smallness compared with $\tF$,  $\hat F$ becomes the dominant part of $F$,  which makes it hard to directly treat the error integrals involved with $\L_Z\rp{\le 1}F$ as a small perturbation.

    The influence  of the general large linear field $\hat{F}$ can be more obviously seen in  bounding the higher order energy fluxes for the scalar field. Our observation is that if the signature\begin{footnote}{For the definition of the signature function, we refer to Theorem \ref{ext_stb}.}\end{footnote} for $D_Z^l\phi$, i.e. $\zeta(Z^l)$, $1\le l\le 2$  has a smaller value, the error integral in the corresponding energy estimates is simpler and the fluxes are expected to decay better in $|u|$ where $2u=\tau_{-}-t_0$.  Thanks to these facts,  we  bound the energy fluxes for $D_Z\phi$ inductively on the value of $\zeta(Z)$. To control the top order energy, for the same purpose of gaining sufficient $|u|$-decay, we first bound the flux for $|u|^{-\zeta(Z^2)} D_Z^2 \phi$ on the outgoing null cones $\H_u$ (see the definition in Section \ref{1.08.1.19}), which gives  the necessary improvement for deriving the full set of energy fluxes for the scalar fields in Theorem \ref{ext_stb}.
	\medskip

	\section{Set-up and Main results}\label{sec_2}
	
	This section will present the main results of this paper, while also introducing notations that will be used throughout the paper.
	\subsection{Decomposition of tensor field}
	Our analysis in the exterior region  relies on the peeling properties of the solution in terms of the null frame
$\{L, \Lb, e_1, e_2\}$ with  $\{e_1, e_2\}$ the orthonormal basis of the sphere $S_{t, r}=\Sigma_t\cap \{|x|=r\}$.\begin{footnote}{We will also employ the notations of $S_{\tau, r}$, $S_{u,v}$  later, which are intersections of the level sets of the two functions in the subscripts. }\end{footnote} We denote $T_0=\p_t$ and $N=\p_r$. We  use $\sD$ to denote the covariant derivative associated to the connection field $A$ on the sphere $S_{t, r}$, which is defined by $\sD_\nu = \tensor*[]{\Pi}{_\nu^\mu}D_\mu$, with $\tensor*[]{\Pi}{_\nu^\mu} = \delta_\nu^\mu +\frac{1}{2}(L^\mu\Lb_\nu + L_\nu\Lb^\mu)$. For any 2-form $G$, denote the null decomposition under the above null frame by
 \begin{equation}\label{12.9.2.17}
		\a_B[G] = G_{L e_B},\quad\ab_B[G] = G_{\Lb e_B},\quad \rho[G]=\frac{1}{2}G_{\Lb L},\quad \sigma[G]=G_{e_1e_2},\quad B=1,2.
	\end{equation}
Here  we interpret $\a, \ab$ as 1-forms tangent to the spheres $S_{t,r}$ (referred to as $S_{t,r}$-tangent tensors) whereas $\rho, \sigma$ are scalar functions. For convenience, $\hat \a=\a[\hat F]$ and $\tilde\a=\a[\tilde F]$ and other null components of $\hat F$ and $\tilde F$ are defined similarly.

For convenience,  slightly different from $\tau_{-}$ and $\tau_+$ in (\ref{1.07.6.19}), we set a second pair of optical functions, which is mainly used in the exterior region,
\begin{equation}\label{12.31.1.18}
u=\frac{t-t_0-r}{2}, \qquad v=\frac{t-t_0+r}{2}.
\end{equation}	

 In  the interior region $\D^+$, we will frequently decompose tensor fields by the  hyperboloidal frames.
 To define them,
 recall that, for $\tau\ge \tau_0$,  $\H_\tau$ is the truncated hyperboloid $\{(t, x)| \sqrt{t^2-|x|^2}=\tau, t-t_0\ge r-R\}$.
The future directed unit normal to the hyperboloid $\H_\tau$ is computed by  $\Tb^\mu=-\p^\mu \tau$, consequently verifies  $\Tb=\tau^{-1}S$. Let
$
\bar \Pi_{\mu\nu}=\Tb_\mu\Tb_\nu+\bm_{\mu\nu}
$
which is the projection tensor field to the tangent space of $\H_\tau$. It also gives the induced metric $\gb$ on $\H_\tau$. We denote by $\ud \nabla$ the Levi-Civita connection of $\ud g$  on $\H_\tau$.

 A vector field $V$ is $\H_\tau$-tangent if $ V(\Tb)=0$, and the same definition applies to  any tensor field if the contraction vanishes for every component.  For an $S_{t,r}$ or $S_{\tau, r}$-tangent tensor field, the norm is taken relative to the induced metric $\ga$ on the $2$-spheres and the norm of an $\H_\tau$-tangent tensor field is taken by $\gb$ as our default.  All such norms are denoted by $|\cdot|$ as long as no confusion is possible.

We will constantly employ the orthonormal basis $\{\eb_i\}_{i=1}^3$ on $\H_\tau$ relative to the metric  $\gb$. $\{\Tb, \eb_i, i=1, \cdots 3\}$ forms the hyperboloidal orthonormal basis in the Minkowski space.

 Let $\Nb$ be the outward unit radial normal, which is tangent to the hyperboloid, and $\{e_1, e_2\}$  be an orthonormal basis of the 2-sphere $S_{\tau, r}$.  We refer to $\{\Tb, \Nb, e_1, e_2\}$ as the hyperboloidal (radial) tetrad although it is also orthonormal, in order to distinguish it from the aforementioned set of frames. The relations between the vector fields $(\Tb, \Nb),$ $(T_0, N)$ and $(L, \Lb)$  are given below,
 \begin{eqnarray}
 \tau \Tb=t T_0+r N, &&   \tau \Nb=r T_0+t N,\label{11.16.1.18}\\
	   2\tau \Tb = \tau_+ L + \tau_{-} \Lb, &&  2\tau \Nb = \tau_+ L -\tau_{-}\Lb,\label{11.11.1.18}\\
\tau T_0= t \Tb-r \Nb,  && \tau N= t\Nb-r \Tb\label{dcp_3}.
\end{eqnarray}	
 Hence we can obtain the relation between components of a $2$-form $G$ in the null tetrad and the hyperboloidal tetrad:
	\begin{equation}
	\label{eqn::MaxwellHyperboloidalandNull}
	\begin{split}
	    G_{\Tb\Nb} = \f12 G_{\Lb L},\quad
	    2G_{\Tb e_A} = \frac{\tau_+}{\tau}\a_A+ \frac{\tau_-}{\tau}\ab_A, \quad
	    2 G_{\Nb e_A} = \frac{\tau_+}{\tau}\a_A - \frac{\tau_-}{\tau}\ab_A,
	\end{split}
	\end{equation}
where $A=1,2$, and the lefthand side of the identities are interpreted as a scalar function and two 1-forms on spheres respectively.

A tensor field $V$ decomposed by either of the two hyperboloidal frames is denoted by $\nt{V}$ for short. The norm of $V$ can be evaluated by the Riemannian metric $h_{\mu\nu}=\bm_{\mu\nu}+2 \Tb_\mu\Tb_\nu$. Clearly, for a 2-form $G$,
\begin{equation}\label{1.10.2.19}
|\nt{G}|= |G_{\Tb\Nb}|+|G(\Tb)|_\ga+|G(\Nb)|_\ga+|\sigma[G]|\approx |G(\Tb)|_{\gb}+\sum_{1\le i<j\le 3}|G(\eb_{i},\eb_j)|\approx |G|_h,
\end{equation}
where $\ga$ denotes the induced metric on the $2$-spheres $S_{\tau, r}$.

  We will also employ the electric-magnetic decomposition for a 2-form $G$ by the hyperboloidal orthonormal basis in Section 4, which in view of (\ref{1.10.2.19}) is equivalent to $\nt{G}$ in terms of the norm $\abs{\cdot}$.

Finally, if a spacetime tensor field $V$ is decomposed by the standard Cartesian frame, $\BT=\{\p_\mu\}_{\mu=0}^3 $ with $\p_0= T_0$,
  $|V|:=|V|_{\bar\delta}$ where $\bar \delta_{\mu\nu}=2{T_0}_\mu{T_0}_\nu+\bm_{\mu\nu}$.
\subsection{The basic energy identity}
 For any complex scalar field $f$ and any 2-form $G$, define the gauge invariant Maxwell-Klein-Gordon energy momentum 2-tensor as \begin{equation*}
		\ET[f, G]_{\mu\nu} = \frac{1}{2}( G_{\mu\delta}{G_{\nu}}^\delta +
		\tensor[^\star]{G}{_{\mu\delta}} \tensor[^\star]{G}{_\nu^\delta })
		+ \Re(\overline{D_\mu f}D_\nu f) - \frac{1}{2}\bm_{\mu\nu}(D^\delta f\overline{D_\delta f} +m^2 f\overline{f}),
	\end{equation*}
where  $\bm_{\mu\nu}$ is the Minkowski metric and $\tensor[^\star]{G}{}$ represents the Hodge dual of $G$. We may also use $\ET[f]$ or $\ET[G]$ to be shorthand for $\ET[f, 0]$ or $\ET[0, G]$ respectively.
	
We will rely on the following energy identity for the energy estimates in the region of $\D^+$.
	
	\begin{prop}\label{prop::EnergyEqMKGModified}
		There holds the following energy identity
 \begin{equation}
 \label{eq:EnergyID}
 \begin{split}
\int_{\H_\tau}\ET[f, G](X, \Tb)\,d\H_\tau& =\int_{\H_{\tau_0}}\ET[f, G](X, \Tb)\,d\H_{\tau_0}+\int_{C_0^\tau} \T[f,G](X, L) d\mu_\ga dt\\
&-\int_{\mathcal{D}^{\tau}}\left(\p^\mu \ET[f, G]_{\mu\nu}{X}^\nu+ \ET[f, G]_{\mu\nu}(\pi^X)^{\mu\nu}\right)\,d\H_{\tau'}d\tau'
\end{split}
		\end{equation}
for any smooth vector field $X$, scalar field $f$ and closed 2-form $G$. Here $\pi^X_{\mu\nu}=\frac{1}{2} (\p_\mu X_{\nu}+\p_{\nu}X_{\mu})$ is the deformation tensor of the vector field $X$, $\D^\tau$ has been defined in (\ref{1.07.3.19}) and  $C_0^\tau=\{\tau_0\le \tau'\le\tau\} \cap C_0$.  The area element on a sphere of radius $r$ is denoted by $d\mu_\ga=r^2 d\omega,\, \omega\in {\mathbb S}^2.$
	\end{prop}

For any complex scalar field $f$ and closed 2-form $G$, define the 1-forms $J[f]$ and $J[G]$ as follows:
\[
J[f]_{\mu}=\Im(f\cdot \overline{D_{\mu}f}),\quad J[G]_{\mu}=\p^{\nu}G_{\mu\nu}.
\]	
We can then compute that
\begin{equation}
\label{eq:div4TfG}
\p^\mu \ET[f, G]_{\mu\nu} = G_{\mu\nu}J[G]^\mu +\Re(\overline{(\Box_A-1)f}D_\nu f) +F_{\nu\mu}J^\mu[f].
\end{equation}

In terms of the null decomposition of the 2-form $G$ and the definition of $\ET[f,G]$, we derive 
\begin{align}
		4\ET[f, G](T_0, \Tb) =& \frac{\tau_{-}}{\tau}(|\ab|^2 + |D_\Lb f|^2)
		+ \frac{\tau_+}{\tau}(|\a|^2 + |D_L f|^2)\nn\\
&+\frac{2t}{\tau}(\rho^2+ \sigma^2 + |\sD f|^2 + |f|^2),\label{11.13.1.18}\\
		4 \ET[G](S, \Tb)=&
		\frac{\tau_{-}^2}{\tau}|\ab|^2 + \frac{\tau_+^2}{\tau}|\a|^2
		+2 \tau (\rho^2+\sigma^2) ,\label{11.13.2.18}
	\end{align}
in which $\ab$, $\a$, $\rho$ and $\sigma$ are null components for the 2-form $G$.

	Using the hyperboloidal frame, we can instead write
	\begin{equation}
	    \label{11.13.3.18}
	    \T[G](S, \Tb)\approx \tau |\nt{G}|^2.
	\end{equation}
Recall that $\Pp$, $\Omega$ denote the sets of generators of Poincar\'{e} group and Lorentz group. For the set of translation vector fields $\BT$, we often simply use $\{\p\}$ with subindices omitted. To distinguish the angular momentums from the boost vector fields in $\Omega$, we denote the former  by $\O=\{\Omega_{ij},1\le i< j\le 3\}$, the latter  by $\B=\{\Omega_{0i}, 1\le i\le 3\}$.

 By a slight abuse of notation, we may denote the generic  elements of a set by the name of the set. For instance, $\O(f)=\Omega_{ij}f,\, \forall\,  1\le i<j\le 3,$  where $f$ is any scalar function, while $|\O f|$ means the sum  $\sum_{Z\in \O}|Zf|$.

For any scalar field $f$ and  2-form $G$, define the $k$-th order energy on the hyperboloid $\H_{\tau}$:
\begin{equation}\label{eqn::QPDefinition}
		\mathcal{E}_k^X[f, G](\tau) = \sum\limits_{Z^k\in \Pp^k}\int_{\H_\tau}\T[D_Z^k f , \L_Z^k G](X,\Tb)\,d\H_\tau,
	\end{equation}
\begin{footnote}{Here $D_Z^l=D_{Z^l}:=D_{Z_1} \ldots D_{Z_l},\,  \L_Z^l=\L_{Z^l}:=\L_{Z_1}\ldots \L_{Z_l}$, and $D_{Z^0}, \L_{Z^0}$ are both the identity map.} \end{footnote}where in general $\mathcal{K}^k:=\{ Z_1 Z_2 \cdots Z_k, Z_i \in \mathcal{K}, i=1,\cdots, k\}$, if $\mathcal{K}\subset \Pp$, and $\mathcal{K}^0=\{\mbox{id}\}$.

In general, denote by $\E_k^X[f,G](\Sigma)$  the energy on a hypersurface $\Sigma$ with $\mathbf{n}$ the surface normal
\begin{equation*}
		\mathcal{E}_k^X[f, G](\Sigma) = \sum\limits_{Z^k\in \Pp^k}\int_\Sigma\T(D_Z^k f , \L_Z^k G)(X,\mathbf{n})\,d\mu_\Sigma.
	\end{equation*}
If $Z^k\in \BT^k$ or $Z^k\in \Omega^k$ in the above definition, we  denote them as  $\mathcal{E}_{\BT, k}^X[f, G](\Sigma)$ and $\E_{\Omega, k}^X[f, G](\Sigma)$ respectively.

We may use $\mathcal{E}_k^X[f](\Sigma) $, $\mathcal{E}_k^X[G](\Sigma)$ to be shorthand for $\mathcal{E}_k^X[f, 0](\Sigma)$, $\mathcal{E}_k^X[0, G](\Sigma)$ respectively. The same interpretation applies to the notations $\mathcal{E}_{\BT, k}^X[f](\Sigma) $ and  $\mathcal{E}_{\Omega, k}^X[G](\Sigma)$.

\subsection{Main results}\label{1.08.1.19}
Note that in the exterior region, the charge of $\tF$ influences our analysis. We denote the chargeless part of $\tF$ as $\ck F$, and the corresponding electric part as $\ck {\tilde{E}}$. At $\Sigma_{t_0}$, $\ck{\tilde{E}}=\ck E^{cf}$, which reads \begin{footnote}{We define the chargeless part of the data in $\{ r\ge\frac{R}{2}\}$ to ensure  the differentiability of $\ck F$ on $C_0$. $\frac{R}{2}$ can be replaced by any $R'$ with $0<R'<R$.}\end{footnote}
\[
\ck E^{cf}_i=E^{cf}_i-q_0 r^{-2}\chi_{\{\frac{R}{2}\leq r\}} \omega_i, \quad \mbox{ where } \omega_i=\frac{x_i}{r}.
\]
 We have
\[
\ck{F}:=\tF-q_0 r^{-2}\chi_{\{t-t_0+\frac{R}{2}\leq r\}}dt\wedge dr.
\]
By the decomposition in (\ref{eq:decomposition4F})
\begin{equation}\label{12.31.2.18}
F=\ck{F}+\hat F+q_0 r^{-2}\chi_{\{t-t_0+\frac{R}{2}\leq r\}}dt\wedge dr,
\end{equation}
where $q_0$ verifies (\ref{10.30.5.18}).

We will state the main theorems of the paper, which are  Theorem \ref{ext_stb} for the result in  the exterior region $\{0\le t-t_0\le r-R\}$, Theorem \ref{them::mainThm} for the result in the entire interior region $\D^+$, with the linear behavior of $\hat F$ given in  Lemma \ref{lemma:decay:lF} and Proposition \ref{linear_ex}.
We first state the result in the exterior region.  For this purpose, we recall a few notations from \cite{KWY}.

  Let $\mathcal{H}_{u}$  denote the outgoing null hypersurface $\{t-t_0-r=2u, 0\le t-t_0\le r-R\}$ and $\Hb_v$  the incoming null hypersurface $\{t-t_0+r=2v, 0\le t-t_0\le r-R\}$, where $u,v$ are optical functions defined in (\ref{12.31.1.18}). We also use $\H_{u}^{v}$ and $\Hb_{v}^{u}$ to denote the truncated hypersurfaces
\begin{align*}
&\H_{u}^{v}:=\{(t, x): \,t-t_0-|x|=2u, \quad -2u\leq  t-t_0+|x|\leq 2v\};\\
&\Hb_{v}^{u}:=\{(t, x):\,t-t_0+|x|=2v, \quad -2v\leq t-t_0-|x|\leq 2u\}.
\end{align*}
On the initial hypersurface $\{t=t_0\}$, define $\Sigma^e_{t_0}={\mathbb R}^3\cap \{r\ge R\}$ and
let $\mathcal{D}_{u}^{v}$ be the domain
bounded by $\H_{u}^{v}$, $\Hb_{v}^{u}$ and the initial hypersurface:
\[
\mathcal{D}_{u}^{v}:=\{(t, x):\, t-t_0-|x|\leq 2u,\quad t-t_0+|x|\leq 2v, \quad t\ge t_0\}.
\]

We denote by  $E[f, G](\Sigma)$  the  appropriate  energy-flux of the 2-form $G$ and  complex scalar field $f$ along the hypersurface $\Sigma$. For the hypersurfaces  of interest to us, \begin{footnote}{ We may hide the notation of the area element if it is the standard area element for the region of the integral.}\end{footnote}
\begin{equation}
\label{generalizedenergy-norms}
\begin{split}
 E[f, G](\Sigma^e_{t_0})&=\int_{\Sigma_{t_0}^e}(|G|^2+|Df |^2+|f|^2)dx,\quad |G|^2=\rho^2+|\si|^2+\frac{1}{2}(|\a|^2+|\underline{\a}|^2),\\
 E[f, G](\H_u^v)&=\int_{\H_u^v}(|D_L f|^2+|\sl D f|^2+|f|^2+\rho^2+\si^2+|\a|^2) ,\\
 E[f, G](\Hb_v^u)&=\int_{\Hb_v^u}(|D_{\Lb}f|^2+|\sl D f|^2+|f|^2+\rho^2+\si^2+|\ab|^2),
\end{split}
\end{equation}
where $\a, \ab, \rho,\sigma$ are the components of $G$ defined in (\ref{12.9.2.17}).

Throughout the paper, we assume $R=1$ without loss of generality and $t\ge R$.
\begin{theorem}\label{ext_stb}
 Consider the Cauchy problem for (\ref{eqn::mMKG}) with the admissible initial data set $(\phi_0, \phi_1, E, H)$. Under the assumption of (\ref{asmp}) with $\M_0$ being fixed, there exists a positive constant $\ve_0$, depending only on $1<\ga_0<2$, $\M_0$ and the arbitrary small constant $\ep>0$ such that if $\E_{2,\ga_0}\le \ve_0^2$,  the unique local solution  $(F, \phi)$ of (\ref{eqn::mMKG}) can be globally extended \begin{footnote}{We refer to the reader to \cite{Eardley1982, Eardley1982a}  for the standard global existence proof without showing the global decay properties.}\end{footnote} in time on the exterior region $\{(t, x):\, t-t_0+R\leq |x|\}$.

\begin{itemize}
\item[(1)] The global solution verifies the following pointwise estimates,
\begin{align*}
r^2|\sD\phi|^2+u_+^2|D_{\Lb}\phi|^2+r^2|D_L \phi|^2 &\leq C \E_{2,\ga_0} r^{-\frac{5}{2}+\ep}u_+^{\f12-\ga_0},\quad  |\phi|^2 \leq C \E_{2,\ga_0} r^{-3}u_+^{-\ga_0};\\
|\ck\rho|^2+|\ck\a|^2+|\ck\si|^2&\leq C \E_{2,\ga_0} r^{-2-\ga_0}u_+^{-1},\quad  |\ck\ab|^2\leq C \E_{2,\ga_0} r^{-2}u_+^{-\ga_0-1},
\end{align*}
where $u_+=|u|$. Here the null components are for $\ck F$, the chargeless part of $\tF$ defined by the decomposition (\ref{12.31.2.18}), c.f.  $\ck \rho=\rho[\ck F]$.
\item[(2)] The  following generalized energy estimates  hold true
\begin{align*}
&E[D_Z^k\phi, \L_Z^k\ck F](\H_{u_1}^{-u_2})+E[D_Z^k\phi, \L_Z^k\ck F](\Hb_{-u_2}^{u_1})\le C (u_1)_+^{-\ga_0+2\zeta(Z^k)}\E_{2,\ga_0} ,\\
&\int_{\H_{u_1}^{-u_2}}r|D_L D_Z^k\phi|^2 +\int_{\Hb_{-u_2}^{u_1} }r(|\sD D_Z^k\phi|^2+|D_Z^k\phi|^2) \le C(u_1)_+^{1-\ga_0+2\zeta(Z^k)}\E_{2,\ga_0},\\
&\int_{\H_{u_1}^{-u_2}}r^{\ga_0}|\a[\L_Z^k \ck F]|^2 +\iint_{\mathcal{D}_{u_1}^{-u_2}}r^{\ga_0-1}|(\a, \rho,\sigma)[\L_Z^k\ck F]|^2 \\
&\qquad \qquad+\int_{\Hb_{-u_2}^{u_1} }r^{\ga_0}(|\rho[\L_Z^k\ck F]|^2+|\si[\L_Z^k\ck F]|^2)  \le C (u_1)_+^{2\zeta(Z^k)}\E_{2,\ga_0}
\end{align*}
for all $u_2<u_1\leq -\frac{R}{2}$, $Z^k=Z_1 Z_2\ldots Z_k$ with $k\le 2$ and $Z_i\in \Pp$, where $\Pp$ is the set of  generators
  of the Poincar\' e group, and  the  signature function  $\zeta:\Pp^{\ell}\rightarrow {\mathbb Z}, \ell\le 2$ is defined  by
  $$ \zeta(\Omega)=0,\, \zeta(\p)=-1 ,  \zeta(Z^0)=0; \zeta(Z_1 Z_2)=\zeta(Z_1)+\zeta(Z_2).$$
\end{itemize}
 The constant $C$ in (1) and (2) depends only on $\ga_0$, $\M_0$  and $\ep$.
\end{theorem}
The proof of Theorem \ref{ext_stb} is presented in Section \ref{ext}. We emphasize that the proof is completely independent of propagation of the solution in $\{t-t_0\ge r-R, t\ge t_0\}.$

\begin{remark}
The result can be easily adapted to the case when the size of the charge $q_0$ is arbitrary by separating the terms related to $q_0$ and following the methods in \cite{KWY}. We disregard this subtlety since it is irrelevant to the result of Theorem \ref{thm1}. 
\end{remark}

Next, we state the behavior of the Maxwell field $\hat F$.
\begin{lemma}
  \label{lemma:decay:lF}
  Let $0<\ep\ll 1$ be fixed.  There hold the following  estimates for the linear Maxwell field $\hat F$ in the entire interior region $\D^+=\bigcup_{\tau'\ge \tau_0} \H_{\tau'}$, in terms of the null decomposition,
  \begin{equation}
    \label{eq:decay:lF}
    \begin{split}
   & \tau_+^{2+\ep}(|\hat{\a}|+\frac{\tau}{\tau_+}|\hat{\rho}|+\frac{\tau}{\tau_+}|\hat{\sigma}|)+\tau_+ \tau_{-}^{1+\ep}  |\hat{\underline{\a}}|\les 1 ,\\
&\hat P_k(\tau):=\int_{\H_{\tau}}\frac{\tau_{-}^{2+2\epsilon}}{\tau}|\ab(\L_{Z}^k \lF)|^2+\frac{\tau_+^{2+2\epsilon}}{\tau}|\a(\L_{Z}^k \lF)|^2+\tau \tau_+^{2\epsilon}(|\rho(\L_{Z}^k \lF)|^2+|\sigma(\L_{Z}^k \lF)|^2)\les 1;
  \end{split}
  \end{equation}
   and with $\hat P_k:=\hat P_k(\tau)$,  there hold that
\begin{align}
&\int_{\H_\tau} \tau \tau_{-}^{2\ep} |\nt\L\rp{ k} _Z {\hat F}|^2 \les \hat P_k\les 1, \label{11.13.4.18}\\
&\tau_{-}^\ep \tau_+\tau |\nt{\hat F}|\les 1, \label{11.16.2.18}
\end{align}
where $Z^k\in \Pp^k, \, k\le 2$  and the constant bounds in all the above estimates merely depend on the fixed numbers $\M_0$.
\end{lemma}
\begin{remark}
{ (\ref{11.13.4.18})  and (\ref{11.16.2.18}) are consequences of (\ref{eq:decay:lF}), (\ref{eqn::MaxwellHyperboloidalandNull}) and (\ref{1.10.2.19}). The decay properties in the exterior region are included in Proposition \ref{linear_ex}.}
\end{remark}
We  choose $\ep$ in Lemma \ref{lemma:decay:lF} to be identical as  in Theorem \ref{ext_stb} for simplicity and  state below the main result in  the entire interior region  $\D^+$.
 \begin{theorem}\label{them::mainThm}
Let  $0<\delta\le\f12(1-\ep)$  be  arbitrarily small and fixed, with  $\ep$ the number fixed in Lemma \ref{lemma:decay:lF}. There exists a constant $\ve_0>0$  such that Theorem \ref{ext_stb}  holds, which can be further refined to be sufficiently small, depending on $\delta$, $\ep$ and $\M_0$. With this bound, if $\E_{2,\ga_0}\le \ve_0$, the local solution  $(F, \phi)$ for the admissible  data can be extended uniquely for all $t\ge t_0$. In  the interior region $\D^+$, there hold  \begin{footnote}{ We regard a constant depending only on quantities among $\delta, \ep, \M_0, \ga_0$  as a universal constant and  adopt $\les$ to keep track of  the universal constants in Section 3 and 4. The dependence on $\delta $ only occurs in Section 4.4.} \end{footnote}for  $\phi$ and the perturbation part $\tF$ defined in (\ref{eq:decomposition4F}) and (\ref{eq:eq4lFandNLF})
 \begin{itemize}
 \item[(1)]  the pointwise estimates  under the null tetrad
 \begin{align*}
& \sup_{\H_\tau}(\tau_+^2|\tilde\a| + \tau_+^\frac{3}{2}\tau_-^\frac{1}{2}(|\tilde\rho| + |\tilde\sigma|) + \tau_-\tau_+|\tilde\ab|)\les  \ve_0, \\
 &\sup_{\H_{\tau}}\left(\tau_+^\frac{3}{2}(|D_L\phi|+|\slashed{D}\phi|+|\phi|) +
			\tau_+\tau_-^\frac{1}{2}|D_\Lb\phi|\right)  \les \ve_0 \jb{\tau}^{\delta};
 \end{align*}
\item[(2)]  the energy estimates
\begin{equation*}
\E^{T_0}_k[\phi](\tau)\les \ve_0^2 \jb{\tau}^{2\delta}, \quad \E^{T_0}_0[\phi](\tau)+\E^S_k[\tF](\tau)\les \ve_0^2,\, \,  0\le k\le 2, \, \tau\ge \tau_0.
\end{equation*}
		\end{itemize}
	\end{theorem}
	\section{Preliminary results}
We will need two types of preliminary results for proving the main results.  One is to provide the initial energies and boundary fluxes, which are given in Proposition \ref{11.7.1.18} and Proposition \ref{11.8.1.18}. The other is to give the weighted Sobolev inequalities on $\H_\tau$, c.f. Proposition \ref{lemma::ScalarDecayExterior} and Proposition \ref{sob}, without additionally requiring the weighted trace bound on the intersection sphere with $C_0$.
\subsection{Initial energies and boundary fluxes}
We first give the initial weighted energy for $\ck F$ in $\Sigma_{t_0}^e$ and  for $\tF$  in $\{|x|\le R\}$.
\begin{prop}\label{11.7.1.18}
\begin{equation}\label{data_1}
 \sum\limits_{l\leq k}\{\int_{\Sigma_{t_0}\cap\{r\le R\}}|\bar\nabla^l E^{cf}|^2 dx+\int_{\Sigma_{t_0}\cap \{r\ge R\}}(1+r)^{\ga_0+2l}|\bar\nabla^l{\ck E}^{cf}|^2dx\}\les  \E_{k,\ga_0},\quad k\le 2.
\end{equation}
\end{prop}
The first part of \eqref{data_1} can follow from the standard elliptic estimate on $\Sigma_{t_0}$ for the Hodge system
\begin{equation*}
 \div{E^{cf}}=\Im(\phi_0\cdot\overline{\phi_1}), \quad  \curl{E^{cf}}=0.
\end{equation*}
We refer the proof of the second part to  \cite[Lemma 11 and Corollary 10]{Yang2015}, which is  based on the result \cite[Theorem 0]{McOwen} or \cite[Theorem 5.1]{Choquet_Chris}.

To prove Theorem \ref{them::mainThm}, we need to run energy estimates on $\D^{\tau_*}$ with $\tau_*$ a fixed large number. Due to the fundamental divergence theorem, we need the initial energies on $\H_{\tau_0}$ and the boundary fluxes on $C_0$, which are given in the following result.  They are derived with the help of  Theorem \ref{ext_stb} and a set of local-in-time energy estimates.
\begin{prop}\label{11.8.1.18}
 The following properties hold for $\tF$ and $\phi$ on $C_0$ and $\H_{\tau_0}$,
\begin{align}
&\sum_{k=0}^2( \E^{T_0}_k[\phi](\tau_0)+\E^{S}_k[\tF](\tau_0))\les \ve_0^2  \label{11.7.2.18},\\
&\sum_{k=0}^2 (\E^{T_0}_k[\phi](C_0)+\E^{S}_k[\tF](C_0))\les   \ve_0^2\label{11.7.3.18}.
\end{align}
\end{prop}
The estimate (\ref{11.7.2.18}) provides the initial energies for the set of interior energy estimates. It will be proved by a bootstrap argument shortly. The estimate (\ref{11.7.3.18}) provides the control on the boundary fluxes for the interior energy estimates.

 As explained in Section \ref{intro}, the sharp boundedness of the weighted energy  for the Maxwell field $\tF$ is crucial for proving Theorem \ref{them::mainThm}. This requires the boundedness of the fluxes for $\tF$ on $C_0$, which is presented in (\ref{11.7.3.18}). It is derived by using the result of Theorem \ref{ext_stb}.  However, Theorem \ref{ext_stb} only provides the derivative control on $\ck F$, instead of $\tF$.  Since the fluxes for $\tF$ in (\ref{11.7.3.18}) are defined naturally by the null tetrad, in the following result we will carry out a careful comparison between the null components of $\L_Z^{\ell} \tF$ and of $\L_Z^{\ell} \ck F$, for $Z^\ell\in \Pp^\ell, 0\le \ell\le 2 $.

\begin{lemma}\label{12.13.1.17}
There hold on $C_0$ for all $Y\in \Pp$ that
\begin{equation}\label{1.13.2.19}
\begin{split}
&|\a[\L_Y (\tF-\ck F)]|\les \jb{\tau_{-}}r^{-3}|q_0| ,\quad  |\rho[\L_Y (\tF-\ck F)]|\les r^{-2+\zeta(Y)}|q_0|,\\
&|\ab[\L_Y (\tF-\ck F)]|\les r^{-2+\zeta(Y)}|q_0|, \quad \sigma[\L_Y \tF]= \sigma[\L_Y \ck F],
\end{split}
\end{equation}
and the symbolic rough version
\begin{equation}\label{12.15.1.17}
|\L_Y (\tF-\ck F)|\les r^{-2+\zeta(Y)}|q_0|.
\end{equation}
For the second order derivatives, there hold for $X, Y\in \Pp$ that
\begin{equation}
|(\a,  \sigma)[\L_Y \L_X (\tF-\ck F)]|\les |q_0| r^{-3} \jb{\tau_{-}},\quad |\rho[\L_Y \L_X(\tF-\ck F)]|\les |q_0|r^{-2}\label{11.8.2.18}.
\end{equation}
\end{lemma}
\begin{remark}
This result is  important to guarantee the boundedness of fluxes on $C_0$ in the  second estimate of (\ref{11.7.3.18}). Note that $\tau_{-}=R$ on $C_0$.
In particular, if  the larger quantities  $r$ or $\jb{\tau_+}$ appeared in the first inequalities in  (\ref{11.8.2.18}) and (\ref{1.13.2.19}),  instead of the smaller one $\jb{\tau_{-}}$ we obtained,  it would cause the loss of the sharp boundedness  in (\ref{11.7.3.18}).  The proof of this lemma  is based on  delicate calculations and  the fact that $r\approx \tau_+$  on $C_0$.
\end{remark}
\begin{proof}
Let $\{e_A\}_{A=1}^2$ be the orthonormal basis on the 2-sphere $S_{u,v}$. We can decompose the standard cartesian frame $\pa_i$ as $$\partial_i= \omega_i \pa_r+ \omega_i^A e_A.$$ It is easy to see that $|\omega_i^A|_\ga\le 1$. It is straightforward to derive
\begin{equation}\label{12.21.1.17}
\begin{split}
&[\Omega_{0i}, L]=\omega_i L -\frac{\tau_{-}}{r} \omega_i^A e_A,\quad  [\Omega_{0i}, \Lb]=-\omega_i \Lb+\frac{\tau_+}{r} \omega_i^A e_A, \, i=1,2,3,\\
&[L, \p_0]=[\Lb, \p_0]=0, \,  [L, \p_i]=-\frac{\omega_i^A}{r} e_A=- [\Lb, \p_i],\quad i=1,2,3\\
&[L, \Omega_{ij}]=[\Lb, \Omega_{ij}]=0,\quad 1\le i<j\le 3.
\end{split}
\end{equation}
The first derivative results in Lemma \ref{12.13.1.17} are included in the following results.
 \begin{enumerate}
 \item[(I)] If $Y=\Omega_{0i}, i=1,2,3$
 \begin{align}
& (\rho, \sigma)[\L_Y (\tF-\ck F)]=(q_0 Y(r^{-2}), 0) \label{1.11.1.19},\\
&\a[\L_Y (\tF-\ck F)]_A=-q_0 \frac{\tau_{-}\omega_{iA}}{r^3} =q_0 r^{-2} \L_{\Omega_{0i}}L^\mu \Pi_{\mu\nu} e_A^\nu,\label{1.11.2.19}\\
&\ab[\L_Y (\tF-\ck F)]_A=-q_0\frac{\tau_+ \omega_{iA}}{r^3} =-q_0 r^{-2} \L_{\Omega_{0i}} \Lb^\mu \Pi_{\mu\nu} e_A^\nu.\label{1.11.3.19}
 \end{align}
 \item[(II)]
 \begin{itemize}
\item[(a)]   If $ Y= \O, \p_t$,  $\L_Y\ck F=\L_Y \tF$; and $\L_Y G=\sl{\L}_Y G$ holds for $2$-forms $G$.
\item[(b)] If $Y=\p$, $\L_Y G-{\sl\L}_Y G=O(r^{-1}) G$ for $2$-forms $G$, \begin{footnote}{Here $U=O(f)$ is either a scalar or an $S_{u,v}$-tangent tensor field  verifying $|U/f|\les 1$.  }\end{footnote} where $G$ on the right hand side represents symbolically any component of the 2-form. There also hold
\begin{equation}\label{1.18.4.19}
|(\a, \ab)[\L_Y(\tF-\ck F)]|\les r^{-3} |q_0|,\quad    (\rho, \sigma)[\L_Y(\tF-\ck F)]=(q_0 Y(r^{-2}), 0).
\end{equation}
\end{itemize}
\end{enumerate}
The  $\sl{\L}_Y G$  in (a) and (b)  represents the projections of $\L_Y G$ on $S_{u,v}$.
With $\a, \ab, \rho, \sigma$  the null  components of the 2-form $G$, $\sl{\L}_Y G$ consists of  two scalar functions  ${\sl\L}_Y\rho:=Y(\rho)$, ${\sl\L}_Y\sigma:=Y(\sigma)$, and a pair of $S_{u,v}$-tangent $1$-tensor fields, which  are defined  \begin{footnote}{In general, for an $S_{u,v}$-tangent vector field $V^\mu$,
$
{\sl\L}_Y V_A=\L_Y V^\mu  {e_A}^\nu.
$}
\end{footnote}below
\begin{equation*}
{\sl\L}_ Y \a_A := \L_Y(L^\mu G_{\mu\nu'} \Pi_\nu^{\nu'})e_A^\nu, \quad {\sl\L}_Y\ab_A:=\L_Y(\Lb^\mu G_{\mu\nu'}\Pi_\nu^{\nu'}) e_A^\nu.
\end{equation*}
To compare these components with the corresponding null components of $\L_Y G$, we adopt the decompositions  for $Z\in \Pp$ in \cite[Page 152-154]{CK}
\begin{equation*}
[Z, \Lb]=\ud Q_A e_A+\ud M \Lb, \quad [Z, L]=Q_A e_A +M L,
\end{equation*}
to derive
\begin{equation}\label{2.07.1.19}
\begin{split}
&\a[\L_Z G]_A =\sl{\L}_Z \a_A-M \a_A +Q_A \rho+\ep_{AB} Q_B \sigma,\\
&\ab[\L_Z G]_A=\sl{\L}_Z \ab_A-\ud M \ab_A-\ud Q_A \rho+\ep_{AB}\ud Q_B \sigma,\\
&\rho[\L_Z G]=\sl{\L}_Z \rho+\f12 \ud Q_A \a_A-\f12 Q_A \ab_A,\\
&\sigma[\L_Z G]=\sl{\L}_Z \sigma-\f12 \ud Q_A{}^\star\a_A-\f12 Q_A {}^\star \ab_A,
\end{split}
\end{equation}
where $\ep_{\mu\nu}=\f12 \ep_{\mu\nu\Lb L}$,  ${}^\star\a_A=\ep_{AB} \a_B$ and ${}^\star \ab=\ep_{AB}\ab_B$ are the Hodge duals on $S_{u,v}$ of $\a$ and $\ab$ respectively.
By using (\ref{12.21.1.17}), we can obtain  for the nontrivial cases that
\begin{align}
& \mbox{if } Z= \Omega_{0i}:  M=\omega_i=-\ud M,\quad  Q_A=-\frac{\tau_{-}}{r} \omega_{iA}\quad \ud Q_A=\frac{\tau_+}{r} \omega_{iA} ;\label{2.07.2.19}\\
&\mbox{ if } Z=\p_i: M=\ud M=0, \quad Q_A =\frac{\omega_{iA}}{r}=-\ud Q_A.   \label{2.07.3.19}
\end{align}

Now we prove (I) and (II) by using  the following facts
\begin{equation}\label{1.18.1.19}
\a[\tF]=\a[\ck F], \,\, \ab[\tF]=\ab[\ck F],\,\,  \sigma[\tF]=\sigma[\ck F],\,\,  \rho[\tF-\ck F]= q_0r^{-2}.
\end{equation}

  Let $Y=\Omega_{0i}, i=1,2,3$ in (\ref{2.07.1.19}). By using  (\ref{2.07.2.19}) we can derive
\begin{align}
\a[\L_Y G]_A&={\sl\L}_Y \a_A-\omega_i \a_A- \frac{\tau_{-}\omega_{iA}}{r} \rho-\frac{\tau_{-}}{r} \ep_{AB} \omega_{iB}\sigma,\label{1.14.03.19}\\
\ab[\L_Y G]_A&={\sl\L}_Y \ab_A+\omega_i \ab_A- \frac{\tau_{+}\omega_{iA}}{r} \rho+\frac{\tau_{+}}{r} \ep_{AB} \omega_{iB}\sigma,\nn\\
\rho[\L_Y G]&=Y(\rho)+\frac{\tau_+}{2r} \omega_i^A \a_A+\frac{\tau_{-}}{2r}\omega_i^A \ab_A,\nn\\
\sigma[\L_Y G]&=Y(\sigma)-\frac{\tau_+}{2r} \omega_i^A{}^\star \a_A+\frac{\tau_{-}}{2r}\omega_i^A{}^\star \ab_A\nn.
\end{align}
 Hence, in view of (\ref{1.18.1.19}) and (\ref{12.21.1.17}), we can obtain   (\ref{1.11.2.19})
 and (\ref{1.11.3.19}).
Also there holds the symbolic formula
\begin{equation}\label{1.13.3.19}
(\rho, \sigma)[\L_Y G]-Y(\rho,\sigma)=O(\frac{\tau_{+}}{r})\a+O(\frac{\tau_{-}}{r})\ab.
\end{equation}
Applying the above identity to $G=\tF$ and $\ck F$, with the help of (\ref{1.18.1.19}), we can obtain (\ref{1.11.1.19}).

The result $\L_Y G={\sl\L}_Y G$ in (II) (a) holds for $Y=\p_t, \O$,   by using  the fact that $[Y, L]=[Y, \Lb]=0$  in (\ref{12.21.1.17}) and (\ref{2.07.1.19}). Hence, by using (\ref{1.18.1.19}), $\L_Y(\tilde F-\ck F)=0$ holds for $Y=\p_t, \O$. Consequently  (II) (a) is proved.

If $Y=\p_i$, using (\ref{2.07.1.19}) and (\ref{2.07.3.19}),   we can derive $\L_Y G-\sl{\L}_Y G=O(r^{-1}) G$ for any $2$-form $G$, and
$
(\rho,\sigma)[\L_Y(\tF-\ck F)]=Y(q_0 r^{-2}, 0).
$
For the $\a, \ab$ components, we note by using (\ref{2.07.1.19}) and (\ref{2.07.3.19}) that
\begin{equation}
\a[\L_Y(\tF-\ck F)]_A=q_0r^{-2}\L_Y L_\nu e_A^\nu=q_0 r^{-3} \omega_{iA}=\ab[\L_Y(\tF-\ck F)]_A \label{1.18.3.19}.
\end{equation}
Combining the results for all the components, we can obtain (\ref{1.18.4.19}). Hence we completed the proof of (I) and (II).
\medskip

Next we consider the second order derivatives.
By (II) (a), if  $X=\p_t, \O$, then $\L_Y \L_X \tF=\L_Y \L_X\ck F$.

  For other combinations of vector fields $X, Y$, note that Lie differentiating  (\ref{1.11.2.19}) and (\ref{1.18.3.19}) directly involves $\sl{\L}_Z^2 {\omega_i}_A$, which could be difficult to compute. Since for an $S_{u,v}$-tangent vector field $V$, $|V(\O)|\approx r|V|_\ga$,  we will instead rely on the  preliminary estimates
\begin{equation}\label{1.12.3.19}
\sum_{\ell=0}^2\sum_{Z^\ell\in \Omega^\ell}|\L_Z^{\ell}L(\O)|\les \jb{\tau_{-}}, \quad \sum_{Z\in \Omega}|\L_\p \L_Z  L(\O)|\les 1,
\end{equation}
and
\begin{equation}\label{1.18.2.19}
|\L_X \L_\p L(\O)|\les  r^{\zeta(X)}, \quad X\in \Pp,
\end{equation}
which will be confirmed in the end.

It is straightforward to calculate,
\begin{equation*}
\sl{\L}_Y(\L_X  L^\mu \Pi_{\mu\nu})e_A^\nu=\L_Y (\L_X  L^\mu \Pi_{\mu\nu})e_A^\nu= \L_Y \L_X L^\mu {e_A}_\mu+\f12 \L_X L^\mu \Lb_\mu \L_Y L_\nu e_A^\nu.
\end{equation*}
  In view of  (\ref{12.21.1.17}), $|\L_X L^\mu\Lb_\mu|\le 1$, which  vanishes if $X\in \Pp\setminus\B$. Hence the second term on the right is bounded by $r^{-1}|\L_Y L(\O)|$ and vanishes if $X\in \Pp\setminus\B$.
Consequently,
\begin{equation}\label{1.19.1.19}
|\sl{\L}_Y(\L_X  L^\mu \Pi_{\mu\nu})e_A^\nu|\les r^{-1}(|\L_Y \L_X L(\O)|+|\L_Y L(\O)|)
\end{equation}
and the last term  on the right hand side vanishes unless $X\in \B$. Next we will employ this estimate to control $\sl{\L}_Y\a[\L_X (\tF-\ck F)]$.

With the help of (\ref{1.19.1.19}), by using (\ref{1.12.3.19}) and (\ref{1.11.2.19}), we directly compute
\begin{equation}\label{1.13.1.19}
|\sl{\L}_Y\a[\L_X (\tF-\ck F)]|\les |q_0|(1+\tau_{-}^{\zeta(Y)+1} )r^{-3} \quad \mbox{ if }  X\in \Omega, Y\in \Pp;
\end{equation}
and  by using (\ref{1.18.2.19}) and (\ref{1.18.3.19})
\begin{equation}\label{1.18.5.19}
|\sl{\L}_Y\a[\L_{\p_i} (\tF-\ck F)]|\les |q_0|r^{-3+\zeta(Y)}, \quad \mbox{ if } Y\in \Pp,
\end{equation}
where for both we employed $|Y(r^{-2})|\les r^{-2+\zeta(Y)}$ on $C_0$.

\medskip
With $X=\Omega_{0i},\,  \p_i,  1\le i\le 3$, we now discuss the following three cases.

\noindent{ \bf Case 1.} If $Y=\Omega_{0j}$ where $1\le j\le 3 $, by using (\ref{1.14.03.19}), there holds
\begin{equation*}
\a[\L_Y\L_X (\tF-\ck F)]_A=\sl{\L}_Y(\a[\L_X (\tF-\ck F)]_\nu)e_A^\nu-\omega_j \a[\L_X(\tF-\ck F)]_A-\frac{\tau_{-}\omega_{iA}}{r} \rho[\L_X(\tF-\ck F)].
\end{equation*}
For $(\rho,\sigma)[\L_Y\L_X(\tF-\ck F)]$, we employ (\ref{1.13.3.19}) with $G=\L_X(\tF-\ck F)$ to derive
\begin{align*}
&(\rho,\sigma)[\L_Y\L_X(\tF-\ck F)]-Y((\rho, \sigma)[\L_X (\tF-\ck F)])=\left(O(\frac{\tau_{+}}{r})\a+O(\frac{\tau_{-}}{r})\ab\right)[\L_X (\tF-\ck F)].
\end{align*}
Note that it follows by using (\ref{1.11.1.19}), (II) and (\ref{1.18.4.19}) that
\begin{equation}\label{2.08.1.19}
(\rho, \sigma)[\L_X (\tF-\ck F)]= (X(q_0 r^{-2}), 0)\quad X\in \Pp.
\end{equation}
By using (\ref{2.08.1.19}), $Y((\rho, \sigma)[\L_X (\tF-\ck F)])=Y(X(q_0 r^{-2}), 0)$, which is bounded by $|q_0|r^{-2+\zeta(X)}$ by direct computation. The right hand side of the above two identities of the second derivative components can be bounded  by using (\ref{1.13.1.19})-(\ref{1.18.5.19}) and the estimates of $\a$, $\ab$ and $\rho$ in (\ref{1.13.2.19}).  In view of $\tau_+\approx r$ on $C_0$, this implies for $ Y\in \B$  and $ X \in \Pp$ that
 \begin{equation*}
| (\a, \sigma)[\L_Y \L_X (\tF-\ck F)]|\les |q_0|\jb{\tau_{-} }r^{-3},\,|\rho[\L_Y \L_X (\tF-\ck F)]|\les |q_0|r^{-2+\zeta(X)}.
 \end{equation*}

\noindent{\bf Case 2.} If $Y=\p$, by using (II)(b),
\begin{equation*}
\L_Y \L_X (\tF-\ck F)-{\sl\L}_Y \L_X (\tF-\ck F)=O(r^{-1}) \L_X (\tF-\ck F).
\end{equation*}
In view of (\ref{12.15.1.17}) and (\ref{1.13.1.19})-(\ref{2.08.1.19}),  this  implies
\begin{equation*}
|(\a, \rho, \sigma)[\L_Y \L_X (\tF-\ck F)]|\les |{\sl\L}_Y (\a, \rho, \sigma) [\L_X (\tF-\ck F)]|+ r^{-3}|q_0|\les r^{-3}|q_0|.
\end{equation*}

\noindent{\bf Case 3.} If $Y=\O$,  by using (II) (a),
\begin{equation*}
\L_Y \L_X (\tF-\ck F)={\sl\L}_Y \L_X (\tF-\ck F).
\end{equation*}
By using  (\ref{2.08.1.19}),  $Y((\rho, \sigma)[\L_X (\tF-\ck F)])=\big(q_0[Y, X](r^{-2}),0\big)$.
Thus noting $[Y,X]\in \Pp$ with $\zeta([Y,X])=\zeta(X)$, we derive
\begin{equation*}
|(\rho, \sigma)[\L_Y \L_X (\tF-\ck F)]|\les |q_0| r^{-2+\zeta(X)}.
\end{equation*}
For the $\a$ component, we employ (\ref{1.13.1.19}) and (\ref{1.18.5.19}) to obtain
$$
|\a[\L_Y \L_X(\tF-\ck F)]|\le |{\sl\L}_Y\a[\L_X(\tF-\ck F)]|\les |q_0|\jb{\tau_{-}}r^{-3}, \quad X\in \Omega
$$
and
\begin{equation*}
|\a[\L_Y \L_{\p_i}(\tF-\ck F)]|\le |{\sl\L}_Y\a[\L_{\p_i}(\tF-\ck F)]|\les |q_0|r^{-3}.
\end{equation*}

Hence (\ref{11.8.2.18}) holds for all the cases.
\medskip

It remains to prove (\ref{1.12.3.19}) and (\ref{1.18.2.19}).
Consider (\ref{1.12.3.19}). Note that the $\ell=0$ case is trivial. We first show that
\begin{equation}\label{1.14.02.19}
|\L_X L(Z)|\les \jb{\tau_{-}}^{1+\zeta(Z)}, \quad X\in \Omega, Z\in\Pp,
\end{equation}
which implies the $\ell=1$ case.
  It is straightforward to check that   $\L_\O L=0$. Otherwise since for $X\in \B$, $[X, \O]=\sum_{i=1}^3 C_i \Omega_{0i}$ with constants $|C_i|=1$,
\begin{equation}\label{1.14.01.19}
 \L_X L(\O)=X(L(\O))-L([X,\O])= \sum_{i=1}^3C_i(L, t \p_i+x^i \p_t)=\sum_{i=1}^3C_i\tau_{-}\frac{x^i}{r},
 \end{equation}
 where we used
 \begin{equation}\label{1.14.3.19}
 (L, t\p_i+x^i \p_t)=\tau_{-}\frac{x^i}{r}.
 \end{equation}
In view of
$
\L_X L(\Omega_{0i})=X(L(\Omega_{0i}))-L([X, \Omega_{0i}]),
$
due to $X\in \Omega$,  by using $|X(\frac{\tau_{-}}{r})|\les \frac{\tau_{-}}{r}$ on $C_0$, $[X, \Omega_{0i}]\in \Omega$  and (\ref{1.14.3.19}), we can obtain
\begin{equation*}
|\L_X L(\Omega_{0i})|\les\sum_{j=1}^3| X\rp{\le 1}(\tau_{-}\frac{x^j}{r})|\les \jb{\tau_{-}}.
\end{equation*}
If $Z=\p$, we can directly check that (\ref{1.14.02.19}) holds with the bounds $\les 1$. Thus (\ref{1.14.02.19}) is proved.

It remains to derive the second order estimates. In view of
 \begin{equation*}
 \L_Y \L_X L(\O)=Y(\L_X L(\O))-\L_X L  ([Y, \O]),
 \end{equation*}
 if $X, Y\in \Omega$, due to $[Y,\O]\in \Omega$, (\ref{1.14.02.19}), $\L_O L=0$ and (\ref{1.14.01.19}), we can derive  that
 \begin{equation*}
 |\L_Y \L_X L(\O)|\les \sum_{i=1}^3|Y(\tau_{-}\frac{x^i}{r})|+|\L_X L(\Omega)|\les \jb{\tau_{-}}.
 \end{equation*}
 Thus we have obtained the first result of (\ref{1.12.3.19}). If $Y=\p$ and $X\in \Omega$, by repeating the above calculation,   also noting that $[Y, \O]\in \{\p\}$ and (\ref{1.14.02.19}), we can obtain
\begin{align*}
& |\L_\p \L_X L(\O)|\les \sum_{i=1}^3|\p(\tau_{-}\frac{x^i}{r})|+|\L_X L(\p)|\les 1.
\end{align*}
Consequently the second inequality in (\ref{1.12.3.19}) is proved.

 Note $\L_Y \L_{\p_t}L(\O)=0$. To see (\ref{1.18.2.19}),   it suffices to consider the case where $\p$ is $\p_i$, with $i=1,2,3$ fixed. We first have
$$
\L_Y \L_{\p_i} L(\O)=-Y(L([\p_i, \O]))-\L_{\p_i}L([Y, \O]).
$$
By a straightforward calculation $L([\p_i, \O])=\sum_{j=1}^3C^i_j \omega_j$ with constants $|C^i_j|=1$.
Thus
\begin{equation*}
\L_Y \L_{\p_i} L(\O)=-\sum_{j=1}^3C^i_j Y(\omega_j)-\L_{\p_i}L([Y, \O]).
\end{equation*}
Note that $|\L_Y \O|\le r^{\zeta(Y)+1}.$ Thus in view of (\ref{12.21.1.17}),
$
|\L_{\p_i}L([Y, \O])|\les r^{\zeta(Y)}.
$
By direct calculation,
$
|Y(\omega_i)|\les r^{\zeta(Y)}.
$
Combining the two estimates yields (\ref{1.18.2.19}).

 We therefore have completed the proof of Lemma \ref{12.13.1.17}.
\end{proof}

\begin{proof}[Proof of Proposition \ref{11.8.1.18}]
We first consider  (\ref{11.7.3.18}).
The proof of the estimate for $\E^{T_0}_k[\phi](C_0)$ can follow directly by  Theorem \ref{ext_stb}. The estimate for the curvature fluxes requires a comparison of the curvature flux of $\L^k_Z \tF$ and $\L^k_Z \ck F$, for $k\le 2$.

 By using $2S=\tau_{+}L+\tau_{-}\Lb$ in (\ref{11.11.1.18}),  for any smooth $2$-form $G$,
\begin{align*}
2\ET[G](S, L)&=2(G_{S\a} \tensor{G}{_L^\a}-\frac{1}{4}\bm(S, L) G_{\a\b} G^{\a\b})\\
&= G_{(\tau_+ L+\tau_- \Lb)\a} \tensor{G}{_L^\a}-\frac{1}{4}\bm(\tau_+ L+\tau_- \Lb, L)G_{\a\b} G^{\a\b}\\
&= \tau_+ |\a[G]|^2+\tau_- (\rho[G]^2+\sigma[G]^2).
\end{align*}
Since $\tau_{-}\approx u_+\approx 1$ and $\tau_+\approx r$ on the cone $C_0:=\{t-t_0=r-R, t\ge t_0\}$, by using Lemma \ref{12.13.1.17}, for $Z^k\in \Pp^k$, we can bound
\begin{align*}
\ET[\L_Z^k \tF](S, L)&\les \ET[\L_Z^k \ck F](S,L)+\tau_+ \jb{\tau_{-}}^2 r^{-6} |q_0|^2+\jb{\tau_{-}} r^{-4}|q_0|^2\\
&\les \ET[\L_Z^k \ck F](S,L) +|q_0|^2 ( r^{-5}+r^{-4}).
\end{align*}
This implies the energy estimate of $\E^S_k[\tF][C_0]$ in (\ref{11.7.3.18}) by direct integration on $C_0$ with the help of  Theorem \ref{ext_stb} (2), the facts that $r\approx t$ on $C^0$, and $|q_0|\les \E_{0,\ga_0}$ in (\ref{10.30.5.18}).

\medskip
Next we prove (\ref{11.7.2.18}). Let $\tilde\Sigma_t=\Sigma_t\cap \{t\ge r+R\}, R\le t\le 2R$.  The existence and uniqueness of the solution in the region  $\{t-t_0\ge r-R, R\le t\le t_0\}$ follow directly from the classical theory.  We will focus on deriving the following energy bound in  this region with data on  $\tilde\Sigma_{t_0}$.
\begin{lemma}\label{2.03.1.19}
With $Z^k\in \mathbb{T}^k$ in the definition (\ref{eqn::QPDefinition}) and $0\le k\le 2$, there holds
\begin{align}
&\E^{T_0}_{\BT, k}[\phi, \tF](\tilde\Sigma_t)+\E^{T_0}_{\BT, k} [\phi, \tF](\tau_0) \les \ve_0^2, \quad R\le t\le 2R \label{11.08.1.18}.
\end{align}
\end{lemma}
The estimate in (\ref{11.7.2.18}) follows as a consequence of the second inequality in (\ref{11.08.1.18}). Indeed,
we can derive
$$\E^{T_0}_k[\phi](\tau_0)+\E^S_k[\tF](\tau_0)\les \E^{T_0}_{\BT, \le k} [\phi, \tF](\tau_0)$$
 by using the following facts:
\begin{itemize}
\item [(1)]
 For scalar functions $f$ and tensor fields $U$ there hold that
\begin{equation*}
|D_Z^l f|\les |D_\p\rp{\le l} f|, \quad |\L_Z^l U|\les |\L_\p \rp{\le l} U|, \,  \mbox{if }  \tau_+\les 1 \end{equation*}
for $Z^l\in \Pp^l$ with $l\in {\mathbb N}$;
\item [(2)] In the region of $\bigcup_{R\le t\le 2R} \widetilde{\Sigma}_t$, since $ \tau_+\les 1$ and $\tau\approx 1$, there hold for $2$-forms $G$ that
\begin{equation*}
\T[G](S, \Tb)\les \T[G](T_0, \Tb),
\end{equation*}
which can be directly seen by setting $f=0$ in (\ref{11.13.1.18}),  and comparing  (\ref{11.13.1.18}) with (\ref{11.13.2.18}).
\end{itemize}
Now we prove the first estimate in (\ref{11.08.1.18}) by  a bootstrap argument. It will immediately give the second one  by running the standard energy estimates.

With the initial bound for the energies of $(\tF, \phi)$ on $\widetilde{\Sigma}_{t_0}$  given in the first part of  Proposition \ref{data_1} and the bound of $\E_{2,\ga_0}$,   in the region $\D_{-}^{t_*}=\bigcup_{\{R< t_*\le t\le 2R \}}\widetilde{\Sigma}_t$, we assume that
\begin{equation}\label{11.8.4.18}
\E^{T_0}_{\BT,k}[\phi, \tF](\widetilde{\Sigma}_t)\le 2\dn^2, \mbox{ where } k=0, 1,2
\end{equation}
with $\dn>\ve_0$ to be chosen later.
The right hand side of (\ref{11.8.4.18}) will be improved to  $<2\dn^2$. Then (\ref{11.8.4.18}) will hold true for $R\le t\le 2R$ by the continuity argument.

We will constantly employ the result  in $\{R\le t\le 2R, t\ge r+R\}$  of the linear part of the Maxwell field $\hat F$,
\begin{equation}\label{11.14.2.18}
\E^{T_0}_{\BT, \le 2}[\hat F](\widetilde{\Sigma}_t)\les \M_0^2;\qquad |\hat F| \les \M_0,
\end{equation}
where the last estimate is a consequence of the first one and the standard Sobolev embedding on $\widetilde{\Sigma}_t.$ Again, since the bounds are  the universal constants, we will denote them as $\les 1$.

 By using the standard Sobolev inequality  and (\ref{11.8.4.18}), we can obtain the pointwise bounds
  \begin{equation}\label{11.8.5.18}
 \|\tF\|_{L^\infty(\D_{-}^{t_*})}+ \|D\rp{\le 1}\phi\|_{L^\infty(\D_{-}^{t_*})}\les \dn, \qquad |F|\les 1,
  \end{equation}
where the last estimate  is a consequence of $|\hat F|\les 1$ in (\ref{11.14.2.18}) and the first estimate of (\ref{11.8.5.18}).

  Now we give the standard energy inequalities. Let $C_0^{-}=\{t-t_0=r-R,  R\le t\le t_0\}.$ Recall the definition of $\ET[f, G]_{\mu\nu}$ from Section \ref{sec_2}. The energy density on $C_0^{-}$ verifies $\ET[f, G](T_0, L)(C_0^{-})\ge 0$.
  Hence it follows by the divergence theorem and (\ref{eq:div4TfG}) that
  \begin{align}
  \E_{\BT, k}^{T_0}[\phi,\tF](\widetilde{\Sigma}_t)&\le\sum_{Z^k\in \BT^k} \abs{\iint_{\D^t_{-}}\L_Z^k \tF_{\mu T_0} J[\L_Z^k \tF]^\mu+\Re(\overline{(\Box_A-1)D_Z^k \phi} D_{T_0} D_Z^k \phi)+F_{T_0\mu}J^\mu[D_Z^k \phi]}\nn\\
  &+ \E_{\BT, k}^{T_0}[\phi, \tF](\widetilde{\Sigma}_{t_0})\nn\\
  & \les \ve_0^2+\sum_{Z^k\in \BT^k}\iint_{\D^t_{-}} |J[D_Z^k \phi]|+|J[\L_Z^k \tF]|^2+|(\Box_A-1)D_Z^k \phi|^2\nn\\
  &\les \ve_0^2+\sum_{Z^k\in \BT^k}\iint_{\D^t_{-}}|(\Box_A-1) D_Z^k \phi|^2+|\L_Z^k J[\phi]|^2, \qquad  k\le 2,\label{1.14.04.19}
  \end{align}
where we used  $|J[D_Z^k\phi]|\les |D D_Z^k \phi||D_Z^k\phi|$, $J[\L_Z^k \tF]=\L_Z^k J[\phi]$ (due to the Maxwell equation in (\ref{eq:eq4lFandNLF})), (\ref{11.8.5.18}), Cauchy-Schwarz inequality and Gronwall's inequality.

   In particular, if $k=0$, by Gronwall's inequality, the above estimate implies that
  \begin{equation}\label{11.14.1.18}
  \E^{T_0}_{\BT,0}[\phi, \tF](\widetilde{\Sigma}_t)\les \ve_0^2.
  \end{equation}

Next we show for $Z^k\in {\mathbb T}^k$, $k\le 2$
\begin{equation}\label{11.9.5.18}
\|\L_Z^k J [\phi]\|_{L^2(\widetilde{\Sigma}_t)}\les \dn^2.
\end{equation}
Indeed, by using the last inequality in Lemma \ref{lem::Commutator1}, there holds for $ Z^k\in {\mathbb T}^k$,
\begin{align*}
|\L_Z^k J [\phi]|&\les \sum_{l_1+l_2\le k} |D_Z^{l_1}\phi||D D_Z^{l_2}\phi|+\sum_{l_1+l_2\le k-1}|D^{l_1}_Z\phi||\L_Z^{l_2}F| |\phi|.
\end{align*}
Note that by  (\ref{11.8.4.18}) and (\ref{11.14.2.18})
\begin{equation}\label{11.9.4.18}
\|\L_Z \rp{\le 1} F\|_{L^2(\widetilde{\Sigma}_t)}\les \E^{T_0}_{\BT, \le 1}[F]^\f12(\widetilde{\Sigma}_t)\les 1.
\end{equation}
With $1\le k\le 2$, by using the pointwise estimate for the scalar field in  (\ref{11.8.5.18}),
\begin{align*}
&\sum_{l_1+l_2\le k} \||D_Z^{l_1}\phi||D D_Z^{l_2}\phi|\|_{L^2(\widetilde{\Sigma}_t)}\les \dn \E_{\BT, \le  k}^{T_0}[\phi]^\f12(\widetilde\Sigma_t)\les \dn^2,\\
&\sum_{l_1+l_2\le k-1}\||D_Z^{l_1}\phi||\L_Z^{l_2}F| |\phi|\|_{L^2(\widetilde{\Sigma}_t)}\les \dn^2 \|\L_Z\rp{\le 1} F\|_{L^2(\widetilde{\Sigma}_t)}\les \dn^2,
\end{align*}
where we employed (\ref{11.8.4.18}) for deriving the last step in the first inequality and (\ref{11.9.4.18}) for the second inequality.
This gives (\ref{11.9.5.18}).

It remains to consider  the second term  on the right hand side in (\ref{1.14.04.19}). By Lemma \ref{lem::Commutator1}, there hold
\begin{equation*}
(\Box_A -1)D_X \phi= Q(F, \phi, X), \quad \forall X\,\in {\mathbb T}
\end{equation*}
and
\begin{equation}\label{11.9.3.18}
\begin{split}
\sum_{X, Y\in {\mathbb T}}|(\Box_A-1)D_X D_Y\phi|&\le \sum_{X, Y\in {\mathbb T}}\{|Q(F, D_X \phi, Y)|+|Q(F, D_Y\phi, X)|+|Q(\L_X F, \phi, Y)|\\
&+|F_{X\mu} \tensor{F}{_Y^\mu} \phi|\},
\end{split}
\end{equation}
where, due to the definition of the quadratic form of $Q$  in (\ref{comm_1}) and for  $Z\in \BT$,
\begin{equation}\label{11.9.2.18}
|Q(G, f, Z)|\les |G| |Df |+|J[G]||f|, \mbox{ with } J[G]_\nu=\p^\mu G_{\nu\mu}.
\end{equation}
Therefore we can obtain for all $X\in {\mathbb T}$, with the help of the Maxwell equation in (\ref{eqn::mMKG}), that
\begin{equation*}
|(\Box_A -1)D_X\phi|\les (|F|+|\phi|^2)|D\phi|
\end{equation*}
and obtain from (\ref{11.9.3.18}) and (\ref{11.9.2.18}) by using $|F|\les 1$ in (\ref{11.8.5.18}) that
\begin{align*}
\sum_{X, Y\in {\mathbb T}}|(\Box_A-1)D_X D_Y \phi|&\les \sum_{Z\in {\mathbb T}}\left(|D D_Z\phi|+(|D\phi|^2+|J[\L_Z F]|+1) |\phi|+|\L_Z F||D\phi|\right).
\end{align*}
 With the help of  (\ref{11.8.5.18}) and (\ref{11.8.4.18}), the above two estimates can be summarized as
\begin{align*}
\|(\Box_A -1)D_Z^k\phi\|_{L^2(\widetilde{\Sigma}_t)}&\les \|D D_\p^{k-1} \phi\|_{L^2(\widetilde{\Sigma}_t)}+\||D\phi||\L_\p\rp{\le k-1} F|\|_{L^2(\widetilde{\Sigma}_t)}\\
&+\|\phi\|_{L^2(\widetilde{\Sigma}_t)}+\dn^3,
\end{align*}
where we used (\ref{eqn::mMKG}), $J[\L_Z F]=\L_Z J[\phi]$, and the estimate (\ref{11.9.5.18}) for $\L_Z J[\phi]$.

Consider the second term on the right hand side. If $k=1$, we can directly use $|F|\les 1$ to bound the curvature term; if $k=2$, we employ the Sobolev embedding and an estimate similar to (\ref{11.9.4.18}). This implies with $k=1, 2$,
\begin{align*}
\||D\phi||\L_\p\rp{\le k-1} F|\|_{L^2(\widetilde{\Sigma}_t)}&\les \| D\rp{\le k-1} D\phi\|_{L^2(\widetilde{\Sigma}_t)}(\| \L\rp{\le k}_\p F\|_{L^2(\widetilde{\Sigma}_t)}+1)\\
&\les \|D\rp{\le k-1} D\phi\|_{L^2(\widetilde{\Sigma}_t)}(\E^{\f12}_{\BT, \le k}[F](\widetilde{\Sigma}_t)+1)\les\|D\rp{\le k-1} D\phi\|_{L^2(\widetilde{\Sigma}_t)}.
\end{align*}
Substituting the above two sets of estimates and (\ref{11.9.4.18})  into (\ref{1.14.04.19}) gives
\begin{align*}
\E^{T_0}_{\BT, k}[\phi, \tF](\widetilde{\Sigma}_t)&\les \ve_0^2+\int^{t_0}_{t} \E^{T_0}_{\BT, \le k-1}[\phi](\widetilde{\Sigma}_{t'}) dt'+\dn^4+\sup_{t\le t'\le 2R}\E^{T_0}_{\BT, 0}[\phi](\widetilde{\Sigma}_{t'})\\
&\les \ve_0^2 +\dn^4
\end{align*}
by using the $0$-order energy estimate in (\ref{11.14.1.18}) and induction. Hence, we have obtained
\begin{equation*}
\E^{T_0}_{\BT, k}[\phi, \tF](\widetilde{\Sigma}_t)<C (\ve_0^2+\dn^4), \mbox{ where } k=0, 1,2,
\end{equation*}
with  the universal constant $C>1$.  The above inequality holds with the right hand side replaced by $<2\dn^2$,   as long as we choose
   $$
\dn=C\varepsilon_0, \quad C\dn^2<1.
   $$
Thus the first inequality in  (\ref{11.08.1.18}) is proved.

The remaining inequality can be derived by repeating the energy estimate in the region enclosed by $\widetilde{\Sigma}_{t_0}$ and $\H_{\tau_0}$, while the nonlinear error integral can be treated by substituting the first estimate in (\ref{11.08.1.18})  directly.
	\end{proof}
	
	\subsection{Sobolev inequalities}
	
In this subsection, we  derive in $\D^+$ various  Sobolev inequalities on hyperboloids, which will be crucial for giving decay properties of solutions. In particular,
	we  denote by $\H_\tau^i, \H_\tau^e$ the interior and exterior parts on $\H_\tau$ respectively, which read
\begin{equation*}
		\H_\tau^i=\{r\le \frac{t}{3}\}\cap \H_{\tau},\quad \H_\tau^e =\{r> \frac{t}{3}\} \bigcap \H_\tau.
	\end{equation*}
Recall that  $\H_\tau$ denotes the set $\{(t,x): \sqrt{t^2-|x|^2}=\tau,\, t-t_0\ge r-R\}$ for $\tau\ge \tau_0$  and  $S_{\tau,r}=\{|x|=r\}\cap \H_\tau$.
Let $r_{\max}(\tau)=\max\{r: S_{\tau, r}\subset \H_\tau\}$.
Since in the exterior region, the solution of (\ref{eqn::mMKG}) is highly nontrivial, we especially need to confirm that the trace of a given  field  on $S_{\tau, r_{\max}(\tau)}$ with a weight of $r$ does not appear on the right hand side of the Sobolev inequalities (which is in contrast to \cite[Sections 4.1 and 4.3]{Psarelli1999}). Otherwise we have to derive additionally a set of $r$-weighted trace estimates whenever using the Sobolev inqualities on $\H_\tau$, which is cumbersome and unnecessary. This is achieved mainly by  Lemma \ref{lemma::GaugeInvariantSobolev1} and its applications. As a remark, the Sobolev inequalities here are general results in $\D^+$, where the constants $C$, $c$ and those in ``$\les $" in this subsection are general constants, independent of any of the quantities $\M_0, \delta, \ep, \ga_0$ introduced in Section \ref{1.08.1.19}.

We will constantly employ the following facts, which can be derived by using $\L_\Omega \Tb=0$, $\L_{\Omega_{\mu\nu} }\gb=0$, and direct calculations.
\begin{lemma}\label{11.11.2.18}
\begin{itemize}
\item[(1)]  If $G$ is $\H_\tau$-tangent, for the generator of the Lorentz group  $\{\Omega_{\mu\nu}, 0\le \mu<\nu\le 3\}$  in the Minkowski space, $\L_{\Omega_{\mu\nu}}G$ is $\H_\tau$-tangent and
\begin{equation*}
|{\Omega_{\mu \nu}}|G||\le |\L_{\Omega_{\mu\nu}} G|,
\end{equation*}
where the norms for the tensor fields are taken by $\gb$.
\item[(2)] There holds the identity
\begin{equation}\label{11.10.7.18}
 \sum_{a=1}^3\Omega_{0a}^{\mu}\Omega_{0a}^{\nu}=t^2 \Pi^{\mu\nu}+\tau^2 \Nb^\mu\Nb^\nu,
 \end{equation}
 where $\Pi^{\mu\nu}$ is the projection tensor to the sphere $S_{\tau, r}$ defined in Section 2.1, and it gives the induced metric $\ga$ on the sphere.  As a straightforward consequence, due to $\tau\le t$,
 with $\ud D$ the  covariant derivative associated to $A$  on  $\H_\tau$, i.e. $\ud D_\nu=\bar \Pi^\mu_\nu D_\mu$, for a complex scalar field $f$,
\begin{equation}\label{10.27.9.18}
|{\ud D} f|^2\le \tau^{-2} \sum_{a=1}^3|D_{\Omega_{0a}} f|^2.
\end{equation}
The same estimate holds for real-valued functions $f$ with $\ud D$ replaced by  $\ud\nabla$ and $D_{\Omega_{0a}} $ replaced by $\p_{\Omega_{0a}}$.
	\end{itemize}
\end{lemma}
We first recall the standard Sobolev inequality on $\H_\tau$.
\begin{lemma}\label{sob_1}
 For any  $\H_\tau$-tangent, real-valued tensor field $U$, there hold with a Sobolev constant $C>0$ that
\begin{equation}\label{9.16.1}
\|U\|_{L^p(\H_\tau)}\le C( \|\ud \nabla U\|^{\frac{3}{2}-\frac{3}{p}}_{L^2(\H_\tau)}\|U\|_{L^2(\H_\tau)}^{\frac{3}{p}-\frac{1}{2}}+\tau^{-(\frac{3}{2}-\frac{3}{p})}\|U\|_{L^2(\H_\tau)}),\, 2 \le  p\le 6;
\end{equation}
 and the trace inequality
\begin{equation}\label{trace}
\|U\|_{L^4(S_{\tau, r})}\le C( \|\ud{\nabla} U\|_{L^2(\H_\tau)}+\tau^{-1}\|U\|_{L^2(\H_\tau)}).
\end{equation}
They also hold for complex scalar field $f$ if $\ud \nabla$ is replaced by $\ud D$.
\end{lemma}

We will employ the following scaling argument to derive Sobolev inequalities on $\H_\tau$.
		Let $(x^\mu)_{\mu=0}^3 $ be the original coordinates on $\H_\tau$, $\tau\ge\tau_0$. Define the rescaled variables $(y^\mu)$ on $\H_1$ as $x^\mu = \tau y^\mu$ where $\mu=0,1,2,3$. Then introduce a new function $\tilde{\phi}((y^\mu))$ and a new connection $\tilde{A}((y^\mu))$ on the truncated unit hyperboloid $\H_1$ by:\begin{equation*}
			\begin{split}
				\tilde{\phi}((y^\mu)) &= \phi((x^\mu)),\quad
				\tilde{A}((y^\mu)) = \tau A((x^\mu)).
			\end{split}
		\end{equation*}
		Then, partial derivatives scale as
\begin{equation*}
			\partial_{y^\mu}\tilde{\phi} = \frac{\partial\phi}{\partial x^\mu}\tau,
\end{equation*}
and the covariant derivatives relative to the connection $\tilde{A}$ scale as
\begin{equation*}
\tilde{D}_\mu \tilde{\phi} = \frac{\partial \tilde{\phi}}{\partial y^\mu} + i\tilde{A}_\mu \tilde{\phi} = \tau D_\mu\phi.
\end{equation*}
For simplicity, we will not use $\,\widetilde{}\,$ to distinguish the quantities under the rescaled coordinates from the original ones whenever there is no confusion. The volume elements of $\H_\tau$, $\H_1$ and $dx$ on $\H_\tau$ are related in the following way: $d\H_\tau = \tau^2d\H_1,$ and  $ dx = \frac{t}{\tau}d\H_\tau$ on $\H_\tau$.

\begin{lemma} \label{lemma::GaugeInvariantSobolev1}
(1)
		Let $\phi$ be a complex scalar field defined on the future directed truncated unit hyperboloid $\H_1$ of ${\mathbb R}^{3+1}$. Then there is a constant $C> 0$ such that \begin{equation*}
			\begin{split}
				\left(\int_{\H_1}(1+r^2)^3|\phi|^6\,dx \right)^\frac{1}{6} &+ \sup_{S_{1, r}\subset \H_1}\left(\int_{S_{1,r}}(1+r^2)^2|\phi|^4\,dS_{1,r} \right)^\frac{1}{4}
				\\&\le  C\left(\int_{\H_1}(|\phi|^2 + \sum_{a=1}^3|D_{\Omega_{0a}}\phi|^2)\,dx\right)^\frac{1}{2},
			\end{split}
		\end{equation*}
where $dx = r^2drd\omega, \omega\in {\mathbb S}^2$.

(2) The above estimate holds for real-valued  scalar functions with $D_{\Omega_{0a}}$ replaced by $\Omega_{0a}^\mu\p_\mu$.

(3) There holds for any smooth $\H_1$-tangent tensor field $G$ that
\begin{equation*}
			\begin{split}
				\left(\int_{\H_1}(1+r^2)^3|G|^6\,dx \right)^\frac{1}{6}& + \sup_{S_{1, r}\subset \H_1}\left(\int_{S_{1,r}}(1+r^2)^2|G|^4\,dS_{1,r} \right)^\frac{1}{4}\\
				&\le  C\left(\int_{\H_1}(|G|^2 + \sum_{a=1}^3|\L_{\Omega_{0a}}G|^2)\,dx\right)^\frac{1}{2}.
			\end{split}
		\end{equation*}

	\end{lemma}
	\begin{proof}
Noting that in the region $\H_1^i$, $r\les 1$, $t\approx 1$, the area element $dx \approx d\H_1$.  We can then apply Lemma \ref{sob_1}, Lemma \ref{11.11.2.18}  (2)  and the fact that $t/\tau\ge 1$ to obtain
\begin{equation*}
			\begin{split}
				\left(\int_{\H_1^i}(1+r^2)^3|\phi|^6\,dx \right)^\frac{1}{6} &+ \sup_{S_{1, r}\subset \H^i_1}\left(\int_{S_{1,r}}(1+r^2)^2|\phi|^4\,dS_{1,r} \right)^\frac{1}{4}\\
&		\les  \left(\int_{\H_1}(|\phi|^2 + \sum_{a=1}^{3}|D_{\Omega_{0a}}\phi|^2)\,dx\right)^\frac{1}{2}.
			\end{split}
		\end{equation*}
It now suffices to consider the estimate on $\H_1^e$.
Recall the standard isoperimetric inequality for a real-valued function $f$ on $2$-sphere $S$,
\begin{equation*}
			\int_S|f-\overline{f}|^2 \le c^2\left(\int_S|\sn f|\right)^2,
		\end{equation*}
		where $\sn$ denotes the Levi-Civita connection of the induced metric $\ga$ on $S$, $c$ is a constant independent of the radius of $S$, and $\overline{f}$ denotes the average of $f$ over $S$.
		We first apply the above isoperimetric inequality to $f = (1+r^2)^\frac{3}{2}|\phi(t(r),r\omega)|^3$, where  $r=|x|$ and $\omega\in {\mathbb S}^2$. Note that in $\H_1^e$, $r\approx t$. Let $r_0=\min\{|x|: x\in \H_1^e\}$. By definition, we can calculate $r_0\approx 1$.

 Then integrating with respect to $r$, we obtain that
 \begin{align*}
			\int_{\H_1^e}(1+r^2)^3|\phi|^6 dx&\les \int_{\H_1^e}dr\left[\frac{1}{r^2}\left(\int_{S_{1,r}}(1+r^2)^\frac{3}{2}|\phi|^3\right)^2 + \left(\int_{S_{1,r}}(1+r^2)^\frac{3}{2}|\phi|^2\sn|\phi|\right)^2\right]\\
			&\les \sup_{S_{1,r}\subset \H_1^e}\left(\int_{S_{1,r}}(1+r^2)^2|\phi|^4\right)\left(\int_{\H_1^e}(\frac{1+r^2}{r^2}|\phi|^2 + \sum_{a=1}^3|D_{\Omega_{0a}}\phi|^2) dx\right)\\
			&\les \sup_{S_{1,r}\subset \H_1^e} \left(\int_{S_{1,r}}(1+r^2)^2|\phi|^4 \right)\left(\int_{\H_1^e}(|\phi|^2 + \sum_{a=1}^3|D_{\Omega_{0a}}\phi|^2) dx\right).
		\end{align*}
 The second line follows from H\"{o}lder's inequality, $\p_i-\omega_i \p_r= -t^{-1}\sum_{j=1}^3(\Omega_{0i}-\omega_i \omega_j \Omega_{0j})$, and $r\approx t$ in $\H_1^e$.

		Now, we estimate $\int_{S_{1,r}}(1+r^2)^2|\phi|^4$ on $\H_1^e$ in the following manner:
\begin{align*}
\abs{\int_{S_{1,r}}(1+r^2)^2|\phi|^4 -\int_{S_{1,r_0}}(1+r^2)^2|\phi|^4 }&\les \int^r_{r_0}\int_{S_{1,\lambda}} |\phi|^3 (\abs{\p_\lambda|\phi|}+\jb{\lambda}^{-1}|\phi|)(1+\lambda^2)^2\, dS_{1,\lambda}d\lambda\\
			&\les \int_{r_0}^r\int_{S_{1,\lambda} }|\phi|^3(1+\lambda^2)^\frac{3}{2}|(\lambda\p_\lambda)\rp{\le 1}|\phi||\, dS_{1,\lambda} d\lambda\\
&\les \left(\int_{\H_1}(1+r^2)^3|\phi|^6 dx\right)^\frac{1}{2}\left(\int_{\H_1}\sum_{a=1}^3|D\rp{\le 1}_{\Omega_{0a}} \phi|^2 dx\right)^\frac{1}{2}.
		\end{align*}
To derive the last line, we used the fact that $\lambda\p_{\lambda}=\tau \Nb= \sum_{a=1}^3\frac{x^a}{r}(t\partial_a+x_a\partial_0)$   on $\H_1$.
		
It follows by the standard trace inequality  (\ref{trace}) and (\ref{10.27.9.18}) that,
\begin{equation*}
(\int_{S_{1,r_0}}|\phi|^4)^\frac{1}{4}\les \|{\ud D}\rp{\le 1}\phi\|_{L^2(\H_1)}\les \sum_{a=1}^3\|D\rp{\le 1}_{\Omega_{0a}}\phi\|_{L^2(\H_1)}\les\left(\int_{\H_1}\sum_{a=1}^3|D\rp{\le 1}_{\Omega_{0a}} \phi|^2 dx\right)^\frac{1}{2}.
\end{equation*}
		Combining the above three estimates for $\int_{S_{1,r}}(1+r^2)^2|\phi|^4$ and $\int_{\H_1^e}(1+r^2)^3|\phi|^6$ yields (1) in the lemma. The other two results can be derived in the same way by using Lemma \ref{11.11.2.18} and Lemma \ref{sob_1}.
	\end{proof}
Next we give the Sobolev embedding for the decay estimates.
\begin{prop}\label{lemma::ScalarDecayExterior}
Let  $\B$ represent all the elements of $\{\Omega_{0a}, a=1,2,3\}$. With a constant $C>0$,

\noindent(1)
		There holds for  any $\H_\tau$-tangent tensor field $G$ that
\begin{equation}\label{11.10.5.18}
			\sup_{\H_\tau}(\tau_+^{\frac{3}{2}}|G|) \le C\left(\int_{\H_\tau} \frac{t}{\tau} |\L\rp{\le 2}_\B G|^2 d\H_\tau\right)^{\frac{1}{2}};
		\end{equation}
(2) There holds for any complex scalar field $\phi$ that
 \begin{equation}\label{11.10.6.18}
			\sup_{\H_\tau}(\tau_+^{\frac{3}{2}}|\phi|) \le C\left(\int_{\H_\tau}\frac{t}{\tau}|D\rp{\le 2}_\B \phi|^2d\H_\tau \right)^{\frac{1}{2}}.
		\end{equation}
	\end{prop}
	\begin{proof}
In the region $\H_1^i$,  by using the standard Sobolev embedding in Lemma \ref{sob_1} for an $\H_1$-tangent tensor $U$ and  Lemma \ref{11.11.2.18}
\begin{equation}\label{11.10.4.18}
\|U\|_{L^6(\H_1^i)}\les \left(\int_{\H_1}(|U|^2 + |\L_\B U|^2)\,d{\H_1}\right)^\frac{1}{2}.
\end{equation}
Note that,  due to the Sobolev embedding on $\H_1$ and  (2) in Lemma \ref{11.11.2.18},
\begin{equation*}
\|G\|_{L^\infty(\H_1^i)}\les \|\ud\nabla |G|\|_{L^6(\H_1)}+\|G\|_{L^6(\H_1)}\les \|\B\rp{\le 1}|G|\|_{L^6(\H_1)}.
\end{equation*}
Hence by using (1) in Lemma \ref{11.11.2.18} and the above two inequalities, we obtain that
\begin{equation}\label{11.10.3.18}
\|G\|_{L^\infty(\H_1^i)}\les \|\L\rp{\le 1}_\B G\|_{L^6(\H_1)}.
\end{equation}
Apply (\ref{11.10.4.18}) to $G$ and $\L_\B G$ , where the latter is still $\H_1$-tangent due to $\L_{\Omega_{0a}}\Tb=0$. We conclude that
\begin{equation*}
\|G\|_{L^\infty(\H_1^i)}\les \|\L\rp{\le 2}_{\B} G\|_{L^2(\H_1)}.
\end{equation*}
Rescaling the above inequality back  to $\H_\tau$ and  using the fact that $\tau_+\approx \tau$ in $\H_\tau^i$ lead to
\begin{equation*}
\|\tau_+^{\frac{3}{2}} G\|_{L^\infty(\H_\tau^i)}\les \|\L\rp{\le 2}_\B G\|_{L^2(\H_\tau)}.
\end{equation*}

  We can similarly prove by  using the classical Sobolev inequality $||\phi||_{L^\infty(\H_1^i)}\les ||\phi||_{L^6(\H_1)} + ||{\ud\nabla}|\phi|||_{L^6(\H_1)}$  that
\begin{equation}\label{11.11.3.18}
	\|\tau_+^\frac{3}{2}\phi \|_{L^\infty(\H_\tau^i)} \les
			\|D\rp{\le 2}_{\B} \phi\|_{L^2(\H_\tau)}.
\end{equation}
Now it remains to derive the estimate on $\H_\tau^e$. For brevity, we will only consider the estimate for complex scalar fields $\phi$.

We recall the classical Sobolev inequality on $S_{1,r}$ given by
\begin{equation*}
\sup_{S_{1,r}}r^{\frac{1}{2}}|\phi|\les \left(\int_{S_{1,r}}|\phi|^4 + r^4|\sn|\phi||^4\right)^\frac{1}{4}.
\end{equation*}
Note that  due to (2) in Lemma \ref{11.11.2.18} and $r<t$ in $\D^{\tau_*}$,  $r|\sn|\phi||\les \sum_{a=1}^3 |D_{\Omega_{0a}}\phi|$. 	We then apply \Cref{lemma::GaugeInvariantSobolev1} for $D\rp{\le 1}_{\Omega_{0a}} \phi$
to obtain the following inequality for $S_{1,r}\subset \H_1^e$
\begin{equation*}
		\sup_{S_{1,r}\subset \H_1^e}r^\f12 (1+r^2)^\frac{1}{2} |\phi|\les\left(\int_{\H_1^e}|D\rp{\le 2}_\B\phi|^2 dx \right)^\frac{1}{2}.
		\end{equation*}
Note $r\approx t$ in $\H_1^e$. (\ref{11.10.6.18}) follows by rescaling the above inequality and combining (\ref{11.11.3.18}).  The estimate for an $\H_\tau$-tangent tensor field $G$ in $\H_\tau^e$ can be similarly derived with the help of (3) in Lemma \ref{lemma::GaugeInvariantSobolev1} and  Lemma \ref{11.11.2.18}.  Thus the proof is completed.
\end{proof}

In view of Lemma \ref{lemma::GaugeInvariantSobolev1}, we also have derived the following $L^p$ estimates by interpolation and the facts $t\approx \tau_+$ and $|\B(t^p)|\les t^p, p\in {\mathbb R}$ in $\{t-t_0\ge r-R, t\ge R\}$.
\begin{prop}\label{sob}
Let $2\le p\le 6$, $U$ be an $\H_\tau$-tangent real tensor field and $V$ be  an $\H_\tau$-tangent complex tensor field. There hold with a constant $C>0$ that
\begin{align}
&\|(\tau_+^2\tau)^{\frac{1}{2}-\frac{1}{p}}  U\|_{L^p(\H_\tau)}\le C\|\L\rp{\le 1}_\B U\|^{\frac{3}{2}-\frac{3}{p}}_{L^2 (\H_\tau)}\|U\|^{\frac{3}{p}-\f12}
_{L^2(\H_\tau)}\label{soblp},\\
&\|(\tau_+^2\tau)^{\f12-\frac{1}{p}} V\|_{L^p(\H_\tau)}\le C \| D\rp{\le 1}_{\B} V\|^{\frac{3}{2}-\frac{3}{p}}_{L^2(\H_\tau)}\|V\|^{\frac{3}{p}-\f12}_{L^2(\H_\tau)}\label{sobsclp}.
\end{align}
\end{prop}

\section{Proof of Theorem \ref{them::mainThm}}
	\subsection{Bootstrap Assumptions}
	
	Let $\tau_*>\tau_0$ be fixed, and $\D^{\tau_*} = \bigcup_{\tau\in[\tau_0,\tau_*]}\H_\tau$. For pproving Theorem \ref{them::mainThm}, we make the following bootstrap assumptions for $\tau_0< \tau\le \tau_*$,
	\begin{equation}\label{BA}
		\begin{split}
	&\E^{T_0}_0[\phi](\tau)\le \dn^2, \quad 	\mathcal{E}_{\ell}^{T_0}[\phi](\tau) \le \dn^2\jb{\tau}^{2\delta},\, \ell=1,2\\
	&		 \mathcal{E}_{k}^{S}[\tF](\tau) \le \dn^2, \quad k=0,1, 2,
		\end{split}
	\end{equation}
where  $0<\delta \le \frac{1-\ep}{2}$ is fixed,  for  the small constant $\ep>0$ in Lemma \ref{lemma:decay:lF}, and $\dn^2=C \ve_0^2$ with the constant $C\ge 4$, greater than the implicit constant in (\ref{11.7.2.18}) and to be further determined.

In $\D^{\tau_*}$,  $\tau_0\le \tau\le t\approx \tau_+$, $\tau_{-}\ge R$ and $\tau^2=\tau_+\tau_{-}$. These facts will be frequently employed.
Since in $\D^{\tau_*}$,  $\tau\ge \tau_0$, we  have $\tau\approx \jb{\tau}$.

	\subsection{Pointwise Decay for the Maxwell Field}\label{2.05.1.19}	
	
	In this subsection, we prove the following pointwise decay for the Maxwell field under the above bootstrap assumptions.
	\begin{prop}
\label{Ptw:Max}
		Let $\tF$ be defined in (\ref{eq:eq4lFandNLF}). In  $\D^{\tau_*}$,  under the bootstrap assumptions  (\ref{BA}), there hold the following decay estimates for $\tilde{F}$:
\begin{equation}\label{dc:Max_ex}
\sup_{\H_\tau}\left(\tau_+^2|\tilde\a| +  \tau_-\tau_+(|\tilde\ab|+|\tF|)+\tau_+\tau (|\tilde \rho|+|\tilde \sigma|+|\nt{\tF}|)\right)\les \dn,
\end{equation}
where $\{\tilde\ab, \tilde\a,\tilde\rho, \tilde\sigma \}$ are the components of  $\tF$ under the null tetrad $\{L, \Lb, e_1, e_2\}$ and  $\nt{\tF}$ represents  the components of $\tF$ decomposed by the hyperboloidal frames  (see Section \ref{sec_2}).
In particular together with the decay estimates in Lemma \ref{lemma:decay:lF} for the  Maxwell field $\hat{F}$ defined in (\ref{eq:eq4lFandNLF}), the full Maxwell field $F$ verifies
\begin{align}
&\sup_{\H_\tau} \tau_+ \tau |\nt{F}|+ \sup_{\H_\tau} \tau^2  |F|\les 1\label{10.27.4.18}.
	\end{align}
	\end{prop}
\begin{remark}
Usually to derive the pointwise decay for $\tF$, one may apply the Sobolev inequality to components of $\tF$ relative to radial tetrads such as the null tetrad $\{L, \Lb, e_1, e_2\}$ (see the Comparison theorems in \cite[Chapter 7]{CK} and \cite[Section 4.3]{Psarelli1999})  or, alternatively, relative to $\{\Tb, \Nb, e_1, e_2\}$, with $\{e_A\}_{A=1}^2$ the orthonormal basis on $2$-spheres.  These treatments  involve  $\L_\B L$ and $\L_\B\Lb$, or  $\L_\B \Nb$, etc. One can refer to  \cite{Psarelli1999} for a rough treatment with the null decomposition of curvature, which is actually incomplete. Note that the energy of $\E^S_k[\tF](\tau)$ takes a rather simple form in view of the hyperboloidal orthonormal frame. This indicates that it suffices to consider the decay property of $|\nt{\tF}|$ by using the Sobolev embedding on $\H_\tau$. For this purpose, we employ the electric-magnetic decomposition  under the hyperboloidal orthonormal frame, which decomposes $\nt \tF$ into the electric and magnetic parts $\tE$ and $\tH$, both of which  are $\H_\tau$-tangent. The decay estimate then follows by applying  (\ref{11.10.5.18}) to both parts. The only  derivative on the frame involved in our proof is $\L_\Omega \Tb=0$, which leads to the vast simplification.
\end{remark}
	\begin{proof}
In view of \eqref{eqn::MaxwellHyperboloidalandNull} and (\ref{1.10.2.19}), we have
\begin{align*}
 |(\rho,\sigma)[\tF]|+\frac{\tau_+}{\tau}|\a[\tilde{F}]|+\frac{\tau_{-}}{\tau}|\ab[\tilde{F}]|\approx |\nt{\tF}|.
\end{align*}
Hence to prove estimate \eqref{dc:Max_ex} it suffices to derive the pointwise decay estimate for $|\nt{\tF}|$.
From the comparison relation \eqref{11.13.3.18} and the definition of $\E^S_{\Omega,k}[\tF](\tau) $, we derive that
\begin{equation*}
\E^S_{\Omega, k}[\tF](\tau)\approx\int_{\H_\tau}\tau |\nt{\L_{\Omega}^k\tF}|^2.
\end{equation*}

We define the electric and magnetic part of a 2-form $G$ relative to the hyperboloidal orthonormal frame $\{\Tb, \eb_i\}$  by
\begin{equation*}
{\ud E}[G]_\ib=G_{\Tb \ib}, \qquad {\ud H}[G]_\ib={}^\star G_{\Tb \ib},
\end{equation*}
where $\{\eb_i, i=1,\cdots 3\}$ is the orthonormal frame on $\H_\tau$,\begin{footnote}{To distinguish them from the standard Cartesian frame, we denote by $V_{\ib}$ or $V^{\ib}$, with $\ib=1,2,3$, if the 1-tensor field $V$ is evaluated by the orthonormal frame $\{\eb_i\}_{i=1}^3$. }\end{footnote} and $\star$ denotes the Hodge dual of the 2-form. Let us fix the convention that
\begin{align*}
&\tE_\ib=\ud E[\tF]_\ib, \quad \tH_\ib=\ud H[\tF]_\ib;\quad \ud{\hE}_\ib=\ud E[\hat F]_\ib, \quad \ud{\hH}_\ib=\ud H[\hat F]_\ib.
\end{align*}
Due to $\L_\Omega \Tb=0$ and the fact that $\Omega$ are Killing vector fields in the Minkowski space,
		 there hold for any 2-form $G$ and $0\le \mu<\nu\le 3$ that
\begin{align*}
\L_{\Omega_{\mu\nu}} (\Tb^\a G_{\a\b})\eb_i^{\b}=\Tb^\a \L_{\Omega_{\mu\nu}} G_{\a\b} \eb^\b_i, \quad \L_{\Omega_{\mu\nu}} (\tensor{\ep}{_{
\Tb\b}^{\ga\delta}}G_{\ga\delta}) \eb_i^\b=\tensor{\ep}{_{\Tb \ib}^{\ga\delta}}\L_{\Omega_{\mu\nu}}G_{\ga\delta}.
\end{align*}
This implies the $k=1$ case in the following identity
 \begin{equation}\label{Com:EH}
\L_Z^k \ud E[G]=\ud E[\L_Z^k G],\quad \L_Z^k \ud H[G]=\ud H[\L_Z^k G], \,\forall\, Z^k\in \Omega^k, k\ge 0,
		\end{equation}
where the  higher order cases can be proved by induction.

Recall the standard property of the electric-magnetic decomposition for 2-forms,
\begin{equation}\label{1.15.1.19}
|\nt G|^2\approx |\ud E[G]|^2+|\ud H[G]|^2.
\end{equation}
	By using (\ref{Com:EH}), (\ref{1.15.1.19}) and (\ref{11.13.3.18}), we can write
\begin{equation*}
\E^S_{\Omega,k}[\tF](\tau)\approx\int_{\H_\tau}\tau(|\L_\Omega^k \tE|^2+|\L_\Omega^k \tH|^2).
\end{equation*}
Since $|\B^l(t^{\ga})|\les t^\ga$ for $\ga\in \mathbb R$ and $l\le 2$ in $\D^{\tau_*}$, by applying Proposition \ref{lemma::ScalarDecayExterior} (1) to $\tau t^{-\f12}(|\tE|+|\tH|)$,  we can obtain
\begin{equation*}
\tau_+^2 \tau^2  (|\tE|^2+|\tH|^2)\les \sum_{k=0}^2 \E^S_{\Omega, k}[\tF](\tau)\les \dn^2.
\end{equation*}
Thus  \eqref{dc:Max_ex} is proved.
The decay estimate \eqref{10.27.4.18} for $|\nt F|$ then follows by combining  the above estimate with (\ref{11.16.2.18}) in view of $F=\tilde{F}+\hat{F}$. Using the first estimate in (\ref{10.27.4.18}), the other estimate can be obtained by using $|F|\les \frac{\tau_+}{\tau} |\nt F|$. (The inequality  is  incorporated in (\ref{10.27.8.18}) and can be proved by using (\ref{dcp_3}) and $F_{T_0 N}=F_{\Tb\Nb}$.)
	\end{proof}
In a similar fashion, we can obtain
\begin{corollary}\label{cor_12_30}
Let $2\le p\le 6$. There holds  for a $2$-form $G$,
\begin{equation}\label{12.30.1.18}
\|(\tau_+^2\tau)^{\f12-\frac{1}{p}}\nt G\|_{L^p(\H_\tau)}\les \|\nt\L_\B\rp{\le 1} G\|_{L^2(\H_\tau)}.
\end{equation}
\end{corollary}
\begin{proof}
The proof is a straightforward combination of the property of electric-magnetic decomposition (\ref{1.15.1.19}), (\ref{Com:EH}) and (\ref{soblp}).
\end{proof}

Next we give more relations of a 2-form contracted by different coordinates or frames in the region $\D^{\tau_*}$, especially when the  2-form appears in null structures.
\begin{lemma}
\label{11.13.5.18}
Let $Z\in \Omega\cup\{S\}$ (it may be different at different places). In the region $\D^{\tau_*}$, there hold the following comparison results for any 2-form $G$
\begin{equation}\label{10.27.8.18}
 \tau(|G_{\mu \mub}|+|G|)+|G_{Z \mub}| \les\tau_+ |\nt{G}|, \qquad |G_{Z Z}|+\tau|G_{\mu Z}|\les \tau_+^2 |\nt{G}|,
\end{equation}
where $G_{\mu\mub}=G(\p_\mu, X)$ with $X$ in the hyperboloidal tetrad $\{\underline{T}, \underline{N}, e_1, e_2\}$, and $\{e_A\}_{A=1}^2$ the orthonormal frame on $S_{\tau, r}$. Moreover there hold
\begin{align}
&|G_{Z\mu} W^\mu|\les (\tau |W|+\frac{\tau_+}{\tau}\sum_{a=1}^3|W_{\Omega_{0a}}|)| \nt{G}|\label{11.13.6.18},\\
&|G_{\mu T_0} W^\mu|\les |\nt{G}| |W|+|\sl{W}||G|,\label{11.13.7.18}
\end{align}
for any real or complex valued vector field $W$. Here $\sl{W}^\mu =\Pi_\nu^\mu W^\nu$
denotes the angular part of $W$. In particular we can estimate that
\begin{align}
&|\hat F_{Z\mu} W^\mu|\les (\tau_+^{-1}|W|+\tau^{-2}|W_\Omega|)\tau_{-}^{-\ep}, \quad  |\tF_{Z\mu} W^\mu|\les (\tau_+^{-1}|W|+\tau^{-2}|W_\Omega|)\dn\label{11.13.10.18},\\
&|\hat {F}_{\mu T_0} W^\mu|\les \tau_{-}^{-\ep}\tau_+^{-1}(\tau^{-1} |W|+\tau_{-}^{-1}|\sl{W}|), \quad |\tF_{\mu T_0} W^\mu|\les \dn\tau_+^{-1}(\tau^{-1} |W|+\tau_{-}^{-1}|\sl{W}|)\label{11.13.11.18}.
\end{align}
\end{lemma}
\begin{proof}
The estimate $|G_{T_0\mub}|+|G_{N \mub}|\les \frac{\tau_+}{\tau} |\nt{G}|$
 can be derived by using  (\ref{dcp_3}).  It immediately implies  the estimates in (\ref{10.27.8.18}) for
$
\tau |G_{\mu \mub}|$, and the one for $|G|$  by using the facts that $G$ is a 2-form and that $G_{T_0N}=G_{\Tb\Nb}$.

 The estimates on $G_{Z\mub}$ and $G_{ZZ}$ follow by using $S=\tau \Tb$ and (\ref{11.10.7.18}) in Lemma \ref{11.11.2.18}. Using (\ref{dcp_3}) again, $|G_{Z \mu}|\les \frac{\tau_+}{\tau} |G_{Z\mub}|$. The estimate for $|G_{Z\mu}|$ can then be derived in view of the one for $|G_{Z\mub}|$. Thus we have completed the proof of (\ref{10.27.8.18}).

For (\ref{11.13.6.18}), decompose
$
|G_{Z\mu} W^\mu|\les |G_{Z \ib} W^{\ib}|+|G_{Z\Tb} W^{\Tb}|$ with
$\{\eb_i\}_{i=1}^3$ the orthonormal basis on  $\H_\tau$.
In view of (\ref{11.10.7.18}), we can obtain
$\tau|W^{\ib}|\les \sum_{a=1}^3|W_{\Omega_{0a}}|$.
Then it follows by using (\ref{dcp_3}) that
\begin{equation*}
|W_{\Tb}|\les \frac{\tau}{t}|W|+\tau^{-1}\sum_{a=1}^3|W_{\Omega_{0a}}|.
\end{equation*}
 Hence,
\begin{equation*}
|G_{Z\mu}W^\mu|\les |\nt{G}_Z|(\tau^{-1} \sum_{a=1}^3|W_{\Omega_{0a}}|+\frac{\tau}{\tau_+}|W|).
\end{equation*}

(\ref{11.13.6.18}) can then be proved by applying the estimate for $|G_{Z\mub}|$ in (\ref{10.27.8.18}) to the factor of $|\nt{G}_Z|$.

 Finally for (\ref{11.13.7.18}), we can carry out the radial  decomposition for $G$:
$$
G_{\mu 0} W^\mu=G_{N 0} W_N+G_{C 0} W^C.
$$
 Here $\{e_C\}_{C=1}^2$ denotes the orthonormal frame on the 2-sphere $S_{\tau,r}$. (\ref{11.13.7.18}) then follows by  using
 $G_{N0}=G_{\Nb \Tb}$ and $|G_{C0}|\le |G|$.
(\ref{11.13.10.18}) and (\ref{11.13.11.18}) can be derived by using (\ref{11.13.6.18}) and (\ref{11.13.7.18}) respectively, together with using (\ref{11.16.2.18}) and Proposition \ref{Ptw:Max}.
\end{proof}

	\subsection{Pointwise Decay of the Scalar Field}
		In this subsection, we prove the following pointwise decay estimates for the scalar field $\phi$ under the bootstrap assumptions \eqref{BA}.
	\begin{prop}
\label{thm:pointdecay:derScal}
		In  the region $\D^{\tau_*}$, under the bootstrap assumptions \eqref{BA}, there hold
		 \begin{align}\label{11.11.4.18}
			\sup_{\H_{\tau}}\{\tau_+^\frac{3}{2}(|D_L\phi|+|\slashed{D}\phi|+|D_{\bfe_a}\phi|+|\phi|)& +
			\tau_+\tau_-^\frac{1}{2}(|D_\Tb \phi|+|D_\Nb\phi|+|D\phi|+|D_\Lb \phi|)\nn\\
&+\tau_+^{\frac{1}{2}}|D_Z\phi|\}  \les \dn \jb{\tau}^{\delta},
		\end{align}
where $ \bfe_a= t^{-1}\Omega_{0a}$,  $a=1,2,3$ and $Z\in\Omega \cup\{S\}$.
	\end{prop}
 We first give a short lemma for  ease of the presentation of the proof.
	\begin{lemma}
Let $X=\Tb$ or $\p$. With $n=1,2$, the following inequality holds for a complex scalar field $\phi$
\begin{equation}\label{12.29.1.18}
|D_\Omega^{n} D_X \phi|\les |D_X D\rp{\le n}_\Omega \phi|+|\L_\Omega\rp{\le n-1} F_{\Omega X}\c \phi|+|F_{\Omega X} \c D_\Omega\rp{\le n-1} \phi|.
\end{equation}
\end{lemma}
\begin{proof}
 In view of (\ref{curv_2}),
 \begin{equation}\label{12.29.3.18}
 D_\Omega D_X \phi= D_X D_\Omega \phi+i F_{\Omega X}\phi+D_{[\Omega, X]} \phi.
\end{equation}
With the help of the symbolic formula $[\Omega, X]=X$ or $0 $, the $n=1$ case in (\ref{12.29.1.18}) can be proved.

For $n=2$, we derive by virtue of (\ref{12.29.3.18}), (\ref{curv_2}), and $[X, \Omega]=X $ or $0$  that
\begin{align*}
D_\Omega^2 D_X \phi&=D_X D^2_\Omega \phi+[D_\Omega, D_X]D_\Omega\phi+D_\Omega[D_\Omega, D_X]\phi\\
&=\sum_{l=1}^2D_X D^l_\Omega \phi+i F_{\Omega X} D_\Omega \phi+  D_\Omega(i F_{\Omega X} \phi+D_X\phi).
\end{align*}
By using the symbolic formulae $[\Omega, \Omega]=\Omega$ and $[\Omega, X]=X$ or $0$,  from the calculation
\begin{align*}
D_\Omega(F_{\Omega X} \phi)&= \p_\Omega F_{\Omega X} \phi+F_{\Omega X} D_\Omega \phi=(\L_\Omega F_{\Omega X} +F_{[\Omega, \Omega] X}+F_{\Omega [\Omega, X]})\phi+F_{\Omega X} D_\Omega \phi,
\end{align*}
we can derive for $X=\Tb$ or $ \p$
\begin{equation*}
D_\Omega( F_{\Omega X} \phi)= (\L_\Omega F_{\Omega X} +F_{\Omega X})\phi +F_{\Omega X} D_\Omega \phi.
\end{equation*}
Substituting the above identity and (\ref{12.29.3.18}) into the calculation of $D_\Omega^2 D_X \phi$ gives
\begin{equation*}
D_\Omega^2 D_X \phi= D_X D^{\le 2}_\Omega \phi+i F_{\Omega X} D_\Omega \phi+ i \L_\Omega\rp{\le 1} F_{\Omega X} \phi
\end{equation*}
as desired in (\ref{12.29.1.18}).
\end{proof}

	\begin{proof}[Proof of Proposition \ref{thm:pointdecay:derScal}]
The pointwise decay estimate for $\phi$ follows directly from the Sobolev embedding \eqref{11.10.6.18} and the bootstrap assumption (\ref{BA}). More precisely we can estimate that
\begin{equation*}
\sup_{\H_\tau} \tau_+^\frac{3}{2}|\phi|\les \|(\frac{t}{\tau})^\f12D_\B\rp{\le 2}\phi\|_{L^2(\H_\tau)}\les (\E^{T_0}_{\le 2}[\phi](\tau))^\f12\les \dn \jb{\tau}^{\delta}.
\end{equation*}
By straightforward calculations, there holds
\begin{equation}\label{1.24.1.19}
\int_{\H_\tau}\frac{t}{\tau}(\sum_{a=1}^3|D_{\bfe_a}\psi|^2+|\psi|^2)+\frac{\tau}{t}|D_{T_0}\psi|^2=2\int_{\H_\tau} \T[\psi,0](T_0, \Tb).
\end{equation}
By (\ref{11.11.1.18}), (\ref{11.13.1.18}),  the bound for $D_{T_0}\psi$ in (\ref{1.24.1.19}), and $L=T_0+N$, we obtain
\begin{equation}\label{1.24.2.19}
\int_{\H_\tau} \frac{\tau}{t} (|D_\Tb \psi|^2+|D_\Nb \psi|^2+|D_N\psi|^2)+\frac{t}{\tau}|\sl{D}\psi|^2\les \int_{\H_\tau} \T[\psi,0](T_0, \Tb).
\end{equation}
The fact that  $|\B^l (t^{-1})|\les t^{-1}$ when $t>r$ implies
\begin{equation*}
\sum_{a=1}^3|D_{\B}^{l}D_{\bfe_a}\phi|\les \sum_{a=1}^3|D_{\bfe_a}D_\B\rp{\le l}\phi|.
\end{equation*}
We then apply  Proposition \ref{lemma::ScalarDecayExterior} to $D_{\bfe_a} \phi$ and use (\ref{1.24.1.19}) to derive
\begin{align*}
	\sup_{\H_\tau} \tau_+^3 (\sum_{a=1}^3 |D_{\bfe_a}\phi|^2)&\les \int_{\H_\tau}\frac{t}{\tau}\left(\sum_{a=1}^3 |D_\B\rp{\le 2} D_{\bfe_a}\phi|^2\right)\les \int_{\H_\tau}\frac{t}{\tau}\left(\sum_{a=1}^3 |D_{\bfe_a}D_\B\rp{\le 2} \phi|^2\right)\\
&\les \E^{T_0}_{\le 2}[\phi](\tau) \les \dn^2 \jb{\tau}^{2\delta}.
\end{align*}
Consequently, also by using (\ref{11.10.7.18}), $|D_\O\phi|\approx r|\sl{D}\phi|$, and $r\le t$ in $\D^{\tau_*}$, we can obtain
\begin{equation*}
\tau |D_\Nb \phi|+t|\sl{D}\phi|+|D_\O\phi|\les \tau_+^{-\f12}\dn \jb{\tau}^{\delta}.
\end{equation*}
Hence the estimates for $D_\Omega \phi$, $\sl{D}\phi$ and $D_\Nb \phi$ have been proved.

 Next we consider the estimates for $D_X\phi$ with $X=\Tb, \p$.
By using the Sobolev embedding (\ref{11.10.6.18}) with the help of $|\B^l(t^{-1})|\les t^{-1}$ for $l\le 2$, and using (\ref{12.29.1.18}), we can bound
 \begin{align*}
	\sup_{\H_\tau} \tau^2 t |D_X \phi|^2&\les \int_{\H_\tau}\frac{\tau}{t}|D_\Omega\rp{\le 2} D_X \phi|^2 d\mu_{\H_\tau}\\
&\les \int_{\H_\tau} \frac{\tau}{t} \{|D_X D_\Omega\rp{\le 2} \phi|^2+|\L_\Omega\rp{\le 1} F_{\Omega X} \phi|^2+|F_{\Omega X} D_\Omega\rp{\le 1}\phi|^2\} d\mu_{\H_\tau}.
		\end{align*}
We next prove
\begin{equation}\label{12.29.4.18}
|F_{\Omega X}|\les 1, \qquad \tau_+^{-\frac{3}{2}} (\frac{\tau}{t})^\f12 |\L_\Omega F_{\Omega X}|\les \tau^{-\f12} |\nt{\L_\Omega F}|.
\end{equation}
By Lemma \ref{11.13.5.18}, it holds for any 2-form $G$ that
\begin{align}\label{12.29.5.18}
 |G_{\Omega \Tb}|\les \tau_+ |\nt{G}|; \qquad |G_{\p \Omega}|\les \jb{\tau}^{-1}\tau_+^2 |\nt{G}|.
\end{align}
With  $G=F$,  using the decay estimates for the Maxwell field $F$ in (\ref{10.27.4.18}) and  $\tau^2 =\tau_{-}\tau_+$,
\begin{align*}
\tau |F_{\Omega \Tb}|\les\tau \tau_+ |\nt{F}|\les 1; \qquad |F_{\p \Omega}|\les \jb{\tau}^{-1}\tau_+^2 |\nt{F}|\les \tau_+\tau^{-2}\les 1.
\end{align*}
This gives the first estimate in (\ref{12.29.4.18}).
Applying $G=\L^l _\Omega F$ with $l=0,1$ to (\ref{12.29.5.18}) gives
\begin{align*}
&\tau_+^{-\frac{3}{2}}(\frac{\tau}{t})^\f12| \L_\Omega ^l F_{\Omega \Tb}|\les \tau_+^{-1}\tau^\f12 |\nt{\L^l_\Omega F}|;\quad  \tau_+^{-\frac{3}{2}} (\frac{\tau}{t})^\f12| \L_\Omega^l F_{\Omega \p}|\les \tau^{-\f12} |\nt {\L^l_\Omega F}|,
\end{align*}
which imply the second estimate in (\ref{12.29.4.18}).
It follows by using (\ref{1.24.2.19}), the decay property of $\phi$ in (\ref{11.11.4.18}) and  using (\ref{12.29.4.18})  to control the curvature terms that
\begin{align*}
\sup_{\H_\tau} \tau^2 t |D_X \phi|^2&\les \int_{\H_\tau}\frac{\tau}{t}  (|D_X D\rp{\le 2}_{\Omega} \phi|^2+|D_\Omega\rp{\le 1}\phi|^2) +|\nt{\L\rp{\le 1}_{\Omega}F}|^2\dn^2 \jb{\tau}^{2\delta-1}\\
&\les \E^{T_0}_{\le 2}[\phi](\tau)+\tau^{-2+2\delta}\dn^2(\E^S_{\le 1}[\tF](\tau)+\hat P_{\le 1})\\
&\les \dn^2 \jb{\tau}^{2\delta},
\end{align*}
where we employed (\ref{BA}) and Lemma \ref{lemma:decay:lF} to derive the last inequality.  This gives  the estimates for $|D\phi|$ and $|D_\Tb \phi|$ in (\ref{11.11.4.18}).

Finally  due to $\Lb=T_0-N$ and the identity $
  (t+r)L=\tau(\Tb+\Nb)$  in  (\ref{11.11.1.18}), we can derive that
\begin{align*}
  |D_L\phi|&\les \tau_+^{-1}\tau (|D_{\Tb}\phi|+|D_{\Nb}\phi|)\les \tau_+^{-1}\tau \tau_+^{-1}\tau_{-}^{-\frac{1}{2}} \dn \jb{\tau}^{\delta}\les \tau_+^{-\frac{3}{2}}\dn \jb{\tau}^{\delta},\\
  |D_\Lb \phi|&\le |D\phi|\les \tau_+^{-\f12} \tau^{-1}\dn \jb{\tau}^{\delta}.
\end{align*}
By using $\tau^2=\tau_+\tau_{-}$, we can complete the proof for the proposition.

	\end{proof}
	
	\subsection{Boundedness theorem of energies in $\D^{\tau_*}$}
	
	The main goal of this subsection is to prove \Cref{them::mainThm}, which is done by improving the bootstrap assumptions (\ref{BA}) with the help of Proposition \ref{11.8.1.18}.
		\begin{theorem}\label{thm::EnergyDecay}
		Under the bootstrap assumptions \eqref{BA},  there hold for $\tau_0\le \tau\le \tau_*$ and $k=0,1,2$ that
 \begin{align}
 \label{eqn:BT:imp:scal}
			\mathcal{E}_k^{T_0}[\phi](\tau) &\les (\varepsilon_0^2 + \dn^3) \jb{\tau}^{k\delta},\\
\label{eqn:BT:imp:Max}
			\mathcal{E}_k^{S}[\tF](\tau) &\les \varepsilon_0^2 +\dn^3.
		\end{align}
	\end{theorem}
		
		We divide the proof into four steps. In Step 1-Step 2, we improve the energy estimates for $\E^{T_0}_k[\phi](\tau)$ with $0\le k\le 2$. In Step 3-Step 4, we complete the energy estimates for $\tF$.

\subsubsection{Energy estimates for the scalar field $\phi$}\label{scal_1}
Apply the energy identity \eqref{eq:EnergyID} for the fields $(D_Z^k\phi, 0)$, $Z^k\in \Pp^k$ and $k\le 2$.
Since $T_0$ is a Killing vector field, by using Proposition \ref{11.8.1.18}, we derive with $k\le 2$ that
 \begin{align}
\int_{\H_\tau} \ET[D_Z^k\phi](T_0, \Tb)& \le \int_{\H_{\tau_0}} \ET[D_Z^k\phi](T_0, \Tb) +\iint_{\mathcal{D}^\tau}|\p^\mu \ET[D_Z^k\phi]_{\mu 0}|+\int_{C_0^\tau} \ET[D_Z^k \phi](T_0, L) \nn\\
&\les \ve_0^2+\iint_{\mathcal{D}^\tau}|\p^\mu \ET[D_Z^k\phi]_{\mu 0}|,\label{11.17.1.18}
		\end{align}
where the volume element  in the each integral is omitted for simplicity.  For the error integral on the right hand side, we recall from identity \eqref{eq:div4TfG} with $(D_Z^k \phi, 0)$ that
\begin{equation}\label{2.09.1.19}
			\p^\mu \ET[D_Z^k\phi ]_{\mu \nu} =
\Re([\Box_A, D_Z^k]\phi \overline{D_\nu D_Z\phi}) + F_{\nu\mu}\Im(D_Z^k\phi\overline{D^\mu D_Z^k\phi}),\quad k\le 2.
		\end{equation}
		\paragraph{\bf Step 1.} Estimate of $\mathcal{E}_{\le 1}^{T_0}[\phi](\tau)$.

By using Lemma \ref{lem::Commutator1} for commutators, we can show that with $k=0,1$
 \begin{equation}\label{1.04.1.18}
			|\p^\mu \ET[D_Z^k\phi ]_{\mu 0} |\le   2|D_0D_Z^k\phi||Z^\nu F_{\mu\nu}D^\mu \phi + \partial^\mu(Z^\nu F_{\mu\nu})\phi| + |D_Z^k\phi||F_{\mu0} D^\mu D_Z^k\phi|.
		\end{equation}
In particular, the first term on the right hand side vanishes if $k=0$.

In view of (\ref{eq:decomposition4F}), if $k=0$, $\iint_{\D^\tau}|\p^\mu\T[\phi]_{\mu 0}|\le I_0$ with
$I_0=\iint_{\D^{\tau}}( |\tF_{\mu 0} D^\mu \phi|+|\hat F_{\mu 0} D^\mu \phi| )|\phi|.$
By applying  (\ref{11.13.11.18}) to $W=D\phi$,
\begin{align*}
I_0&\les \iint_{\D^\tau} (\frac{s_+}{s})^\f12|\phi|\cdot (\frac{s}{s_+})^\f12(s^{-1}|D\phi|+s_{-}^{-1}|\sl{D}\phi|)s_+^{-1}(s_{-}^{-\ep} +\dn)\\
&\les\int_{\tau_0}^\tau \E^{T_0}_0 [\phi](s) s^{-2} ds,
\end{align*}
where  $s_{-}$, $s$ and $s_{+}$  represent the varying parameters corresponding to $\tau_{-}$, $\tau$ and $\tau_{+}$ of a fixed point.
Substituting the estimate of $I_0$ into (\ref{11.17.1.18}) and  by  using Gronwall's inequality, we conclude
\begin{equation}\label{11.17.2.18}
\E_0^{T_0}[\phi](\tau)\les \ve_0^2.
\end{equation}

 When $k=1$ in (\ref{1.04.1.18}), we bound  $\iint_{\D^\tau}|\p^\mu \ET[D_Z\phi]_{\mu 0}|$  by the following three integrals,
\begin{align*}
			I_1 &= \iint_{\mathcal{D}^{\tau}} |D_0D_Z\phi||Z^\nu \lF_{\mu\nu}D^\mu \phi + \partial^\mu(Z^\nu \lF_{\mu\nu})\phi| + |D_Z\phi|| \lF_{\mu0} D^\mu D_Z\phi|,\\ 
			I_2 &=\iint_{\mathcal{D}^\tau}|D_0 D_Z\phi||Z^\nu \tF_{\mu\nu}D^\mu \phi + \partial^\mu(Z^\nu \tF_{\mu\nu})\phi|,\\
			I_3 &= \iint_{\mathcal{D}^\tau}|\tF_{\mu0} D^\mu D_Z\phi||D_Z\phi|.
		\end{align*}
	
		\subparagraph{\it Bounding $I_1$:}
		
	 Note that $\p^\mu \lF_{\mu\nu}=0$. We can bound
\begin{align*}
				I_1&\le \iint_{\mathcal{D}^\tau} |D_0D_Z\phi|(|Z^\nu \lF_{\mu\nu} D^\mu \phi| + |\partial^\mu Z^\nu ||\lF_{\mu\nu}||\phi|) + |D_Z\phi|| \lF_{\mu0} D^\mu D_Z\phi|.
		\end{align*}
 Applying
 the first estimates in (\ref{11.13.10.18})  and (\ref{11.13.11.18}) to $W=D  \phi$ and $W= D D_Z\phi$ respectively, also using  $|\partial^\mu Z^\nu ||\lF_{\mu\nu}|\les s_+^{-1} s_{-}^{-\ep}$ due to  Lemma \ref{lemma:decay:lF}, we derive by virtue of $s_+s_{-}=s^{2}$ for $I_1$ that
\begin{align*}
  I_1&\les \int_{\tau_0}^{\tau}\int_{\H_{s}}  s_+^{-1}s_{-}^{-\epsilon} |D_0 D_Z\phi|( |D \phi| +| D_{\Omega}\phi| + |\phi|) + s_+^{-1}s_{-}^{-\epsilon}|D_Z\phi|(s^{-1}| D D_Z\phi|+s_{-}^{-1}| \sD D_Z\phi|)\\
  &\les \int_{\tau_0}^{\tau} s^{-1-\epsilon}  (\int_{\H_{s}}\frac{s_{-}}{s}|DD_Z\phi|^2)^{\frac{1}{2}} (\int_{\H_{s}}\frac{s_+}{s}(\sum_{\Ga\in \Pp}|D_{\Gamma}\phi|^2+|\phi|^2))^{\frac{1}{2}} + s^{-1-\epsilon} \int_{\H_{s}}\frac{s_+}{s}| \sD\rp{\le 1} D_Z\phi|^2\\
  &\les \ve_0^2+\int_{\tau_0}^\tau s^{-1-\ep} \E_1^{T_0}[\phi](s) ds.
\end{align*}
For deriving the last inequality we employed (\ref{11.17.2.18}).  The last term on the right hand side will be treated by using Gronwall's inequality.
 	
		\subparagraph{\it Bounding $I_2$:}
		
By using (\ref{eq:eq4lFandNLF}) for $\tF$, we first can estimate that
		\begin{equation} \label{eqn::I2Decomp}
			\begin{split}
				I_2 
				&\les \iint_{\mathcal{D}^\tau}| \tF_{\mu Z} D^\mu\phi|| D_0D_Z\phi| + \iint_{\mathcal{D}^\tau}(|J[\phi]_Z|+|\tF|)|\phi||D_0D_Z\phi|.
			\end{split}
		\end{equation}
 Applying the second inequality in \eqref{11.13.10.18} to $W=D\phi$ implies
\begin{equation*}
  | \tF_{\mu Z} D^\mu\phi|\les  \dn (|D \phi| \tau_+^{-1}+ |D_\Omega\phi|\tau^{-2}).
\end{equation*}
By using the bootstrap assumptions \eqref{BA}, we therefore can derive that
\begin{align*}
  \iint_{\mathcal{D}^\tau}| \tF_{\mu Z} D^\mu\phi|| D_0D_Z\phi|&\les \dn \int_{\tau_0}^{\tau}  \int_{\H_s} | D_0D_Z\phi|(|D \phi| s_+^{-1}+ |D_\Omega\phi|s^{-2})\\
   &\les \dn \int_{\tau_0}^{\tau} s^{-1} (\int_{\H_{s}} \frac{s}{s_+}| D_0D_Z\phi|^2)^{\frac{1}{2}} (\E_0^{T_0}[\phi]^\f12(s) +s^{-1}\E_1^{T_0}[\phi]^\f12 (s)  )\\
   &\les \dn^3 \int_{\tau_0}^{\tau} s^{-1} \jb{s}^{\delta}ds\les  \dn^3 \jb{\tau}^{\delta}.
\end{align*}

For the second integral on the right hand side of \eqref{eqn::I2Decomp}, using the pointwise decay estimates for $\phi$ and the Maxwell field $\tF$  in
 (\ref{11.11.4.18}) and (\ref{dc:Max_ex}),  we have
\begin{align*}
  \iint_{\mathcal{D}^\tau}(|J[\phi]_Z||D_0D_Z\phi| + |\tF||D_0D_Z\phi|)|\phi| &\les \iint_{\mathcal{D}^\tau}|D_Z\phi| |\phi|^2|D_0D_Z\phi| + |\tF||\phi||D_0D_Z\phi|\\
  &\les\dn \iint_{\mathcal{D}^\tau}s_+^{-3+2\delta}  |D_Z\phi| |D_0D_Z\phi| + s^{-2} |\phi||D_0D_Z\phi|\\
  &\les \dn \int^\tau_{\tau_0}s^{-2} (\int_{\H_{s}} \frac{s_{-}}{s} |D_0D_Z\phi|^2)^{\frac{1}{2}} (\int_{\H_{s}} \frac{s_{+}}{s} |D\rp{\le 1}_Z\phi|^2)^{\frac{1}{2}}\\
  &\les \dn^3 \int^\tau_{\tau_0}s^{-2} \jb{s}^{2\delta}ds \les \dn^3.
\end{align*}
 We therefore conclude
\begin{equation*}
  I_2\les  \dn^3 \jb{\tau}^{\delta}.
\end{equation*}
		
		\subparagraph{\it Bounding $I_3$:}
		
		For $I_3$, we employ  the pointwise estimate for $\tF$ in (\ref{dc:Max_ex}) and (\ref{BA}) to derive
 \begin{align*}
   I_3&\les \iint_{\mathcal{D}^\tau}|\tF_{\mu0} D^\mu D_Z\phi||D_Z\phi|  \les \iint_{\mathcal{D}^\tau}|D_Z\phi|  |\tF||D D_Z\phi|\\
    &\les \dn \int^\tau_{\tau_0} s^{-2}(\int_{\H_{s}}\frac{s_{-}}{s}|DD_Z\phi|^2)^{\frac{1}{2}} (\int_{\H_{s}}\frac{s_{+}}{s}|D_Z\phi|^2)^{\frac{1}{2}} \\
    &\les \dn^3 \int^\tau_{\tau_0} s^{-2}\jb{s}^{2\delta}\les \dn^3.
    \end{align*}
 Combining the estimates for $I_1$, $I_2$ and $I_3$ gives
 \begin{align*}
   \mathcal{E}_{1}^{T_0}[\phi](\tau)\les \ve_0^2+ \dn^3 \jb{\tau}^{\delta}+\int_{\tau_0}^{\tau}s^{-1-\epsilon}\mathcal{E}_{1}^{T_0}[\phi](s)ds.
 \end{align*}
In view of Gronwall's inequality, this  leads to
\begin{equation*}
  \mathcal{E}_{1}^{T_0}[\phi](\tau)\les \ve_0^2+ \dn^3 \jb{\tau}^{\delta}.
\end{equation*}
Hence \eqref{eqn:BT:imp:scal} holds for $k=1$.	
\medskip
		\paragraph{\bf Step 2.} Estimate of $\mathcal{E}_{2}^{T_0}[\phi](\tau)$.
The estimate for $\mathcal{E}_2^{T_0}[\phi](\tau)$ is more involved since the error integral is much more complicated. Noting that with $k=2$ in (\ref{11.17.1.18}), there holds for $Z^2\in \Pp^2$ that
 \begin{equation}\label{1.05.1.19}
			\int_{\H_\tau} \ET[D_Z^2\phi](T_0, \Tb) \les \ve_0^2 +\iint_{\mathcal{D}^\tau}|\p^\mu \ET[D_Z^2\phi]_{\mu 0}|.
		\end{equation}
    In view of (\ref{2.09.1.19}) and  Lemma \ref{lem::Commutator1},
 \begin{align*}
			|\p^\mu \ET[D_Z^2\phi ]_{\mu 0} |&\les |D D_Z^2\phi|\big(\sum_{Z_1\sqcup Z_2=Z^2}(|Q(F, D_{Z_1}\phi, Z_2)|+|Q(\L_{Z_1} F, \phi, Z_2)|+|F_{Z_1\mu}F^{\mu}_{\ Z_2}||\phi|)\\
&+\sum_{X\in \Pp} |Q(F, \phi, X)|\big) + |F_{\mu0}D^\mu D_Z^2\phi||D_Z^2\phi|,
		\end{align*}
where the notation $Z_1\sqcup Z_2=Z^2$ stands for the two cases that $Z_1Z_2=Z^2$ and $Z_2 Z_1=Z^2$, and we used the fact that the set $\Pp$ is closed under the Lie bracket $[\cdot,\cdot]$.
   Combining the decay estimates in (\ref{10.27.4.18}) with (\ref{10.27.8.18}) implies
\begin{equation}\label{10.27.5.18}
|F_{Z_1 \mu} \tensor{F}{^\mu_Z_2}|\les {\tau_+}^2 |\nt{F}|^2\les\tau^{-2}.
\end{equation}
Also by using $|F|\les \tau^{-2}$ in (\ref{10.27.4.18}) and (\ref{11.17.2.18}) again, we can bound
\begin{align*}
  \iint_{\mathcal{D}^\tau} |D D_Z^2\phi| &|\phi|\sum_{Z_1\sqcup Z_2=Z^2}|F_{Z_1\mu}F^{\mu}_{\ Z_2}| + |F_{\mu0}||D_Z^2\phi|| D^\mu D_Z^2\phi|\\
  &\les  \int_{\tau_0}^{\tau}  s^{-2}( \int_{\H_{s}} \frac{s_{-}}{s}|DD_Z^2\phi|^2)^{\frac{1}{2}}( \int_{\H_{s}} \frac{s_+}{s} (|D_Z^2\phi|^2+|\phi|^2)  )^{\frac{1}{2}}\\
  &\les \ve_0^2+\int_{\tau_0}^{\tau}  s^{-2} \mathcal{E}_2^{T_0}[\phi](s)ds.
\end{align*}
The last term will be treated by using Gronwall's inequality.

Next we move on to estimate  $\iint_{\D^\tau}|D D_Z^2\phi||Q(F, D_{Z_1}\phi, Z_2)|$. In view of the equation (\ref{eqn::mMKG}) for the Maxwell field $F$, we can bound
\begin{align*}
  |Q(F, D_{Z_1}\phi, Z_2)|&= |Z_2^\nu F_{\mu\nu}D^\mu D_{Z_1}\phi + \partial^\mu(Z_2^\nu F_{\mu\nu})D_{Z_1}\phi|\\
  &\les |F_{\mu Z_2} D^\mu D_{Z_1}\phi| + (|F|+ |D_{Z_2}\phi| |\phi|) |D_{Z_1}\phi|.
\end{align*}
The second term is easy to bound as there is sufficient decay from $|F|+|D_{Z_2}\phi||\phi|$. We may denote $Z_1$ or $Z_2$ by $Z$ for short in the sequel.  Indeed from (\ref{10.27.4.18}) and  Proposition \ref{thm:pointdecay:derScal},  and using $0<\delta\le \frac{1-\ep}{2}$,  we can bound
\begin{equation*}
  |F|+ |D_Z\phi| |\phi|\les \tau^{-2}+ \dn^2  \tau_+^{-2+2\delta}\les \tau^{-2+2\delta}\les \tau^{-1-\ep}.
\end{equation*}
 As for $|F_{\mu Z_2}D^\mu D_{Z_1}\phi|$, we employ (\ref{eq:decomposition4F}) and (\ref{11.13.10.18})  to obtain
\begin{align*}
  |F_{\mu Z_2} D^\mu D_{Z_1}\phi|&\les|\lF_{\mu Z_2} D^\mu D_{Z_1}\phi|+|\tF_{\mu Z_2} D^\mu D_{Z_1}\phi|\\
  &\les \tau ^{-1-\epsilon} \sum_{\Ga\in \Pp}|D_{\Gamma}D_{Z_1}\phi|+ \dn \tau^{-1}|DD_{Z_1}\phi|.
\end{align*}
Thus,
\begin{equation*}
|Q(F, D_{Z_1}\phi, Z_2)|\les (\tau^{-1-\epsilon}+\dn \tau^{-1}) \sum_{\Ga\in \Pp}|D\rp{\le 1}_{\Gamma}D_{Z_1}\phi|
\end{equation*}
and  similarly there holds for any $X\in \Pp$,
\begin{equation*}
|Q(F, \phi, X)|\les (\tau^{-1-\epsilon}+\dn \tau^{-1})\sum_{\Ga\in \Pp} |D_{\Gamma}\rp{\le 1}\phi|.
\end{equation*}
Therefore,
\begin{align*}
  \iint_{\mathcal{D}^\tau}&|D D_Z^2\phi|\big(\sum_{Z_1\sqcup Z_2=Z^2} |Q(F, D_{Z_1}\phi, Z_2)|+\sum_{X\in \Pp} |Q(F, \phi, X)|\big)\\
  &\les \int_{\tau_0}^{\tau}  \big(s^{-1-\epsilon}+\dn s^{-1}\big)\int_{\H_{s}}|D D_Z^2\phi|\sum_{Z\in \{Z_1, Z_2\}}\sum_{\Ga\in \Pp} |D\rp{\le 1}_{\Gamma}D\rp{\le 1}_Z\phi|\\
  &\les \int_{\tau_0}^{\tau} (s^{-1-\epsilon}+\dn s^{-1})( \int_{\H_{s}} \frac{s_{+}}{s}(\sum_{\Ga\in \Pp}|D_{\Gamma}D\rp{\le 1}_Z\phi|^2+|D\rp{\le 1}_Z\phi|^2))^{\frac{1}{2}}\c ( \int_{\H_{s}}\frac{s_{-}}{s}|DD_Z^2\phi|^2 )^{\frac{1}{2}} ds\\
  &\les \int_{\tau_0}^{\tau} (s^{-1-\epsilon}+\dn s^{-1}) (\mathcal{E}_{2}^{T_0}[\phi](s)+\E_{\le 1}^{T_0}[\phi](s))ds\\
  &\les (\ve_0^2+\dn^3) \jb{\tau}^{2\delta}+\int_{\tau_0}^{\tau} s^{-1-\epsilon} \mathcal{E}_{2}^{T_0}[\phi](s)ds,
\end{align*}
where we employed the Cauchy-Schwarz inequality and the proven estimate for $k\le 1$ in (\ref{eqn:BT:imp:scal}) in the above calculation to treat the lower order energy for $\phi$. In the last step we also employed the bootstrap assumption \eqref{BA} to bound $s^{-1}\dn\E_2^{T_0}[\phi](s)$. Again the last term in the last line will be treated by using Gronwall's inequality.

It remains to control the other quadratic term $|Q(\L_{Z_1} F, \phi, Z_2)|$.
We first estimate that
 \begin{align}
   \iint_{\mathcal{D}^\tau}&|D D_Z^2\phi|\sum_{Z_1\sqcup Z_2=Z^2} |Q(\L_{Z_1} F, \phi, Z_2)|\nn\\
   &\les  \int_{\tau_0}^{\tau}  (\int_{\H_{s}}\frac{s_{-}}{s}|D D_Z^2\phi|^2)^{\frac{1}{2}}(\int_{\H_{s}}\frac{s_{+}}{s}\sum_{Z_1\sqcup Z_2=Z^2} |Q(\L_{Z_1} F, \phi, Z_2)|^2)^{\frac{1}{2}}ds \nn\\
   &\les \int_{\tau_0}^{\tau}  (\mathcal{E}_{2}^{T_0}[\phi](s))^{\frac{1}{2}}(\int_{\H_{s}}\frac{s_{+}}{s}\sum_{Z_1\sqcup Z_2=Z^2} |Q(\L_{Z_1} F, \phi, Z_2)|^2)^{\frac{1}{2}}ds.\label{1.05.2.19}
 \end{align}
Next we consider the second factor in the last line. By using the equation for the Maxwell field in  (\ref{eqn::mMKG}) and the formula in Lemma \ref{lem::Commutator1}, we derive
  \begin{align*}
  |Q(\L_{Z_1} F, \phi, Z_2)|&= |{Z_2}^\nu (\L_{Z_1} F)_{\mu\nu}D^\mu \phi + \partial^\mu(Z_2^\nu (\L_{Z_1} F)_{\mu\nu})\phi|\\
  &\les | (\L_{Z_1} F)_{\mu Z_2}D^\mu  \phi| +  (|\L_{Z_1} F|+|Z_2^\nu (\L_{Z_1} J[\phi])_{\nu}|)|\phi|\\
  &\les | (\L_{Z_1} F)_{\mu Z_2}D^\mu  \phi| +  (|\L_{Z_1} F|+|D_{Z_1}\phi| |D_{Z_2}\phi| + |\phi||D_Z^2 \phi|+| F_{Z_1 Z_2}||\phi|^2 )|\phi|.
\end{align*}
For the first term in the last line,   applying (\ref{11.13.6.18}) to $G_\mu=\L_{Z_1}F_{\mu Z_2}$ and $W=D\phi$ implies
\begin{equation*}
|{\L_{Z_1} F}_{\mu Z_2} D^\mu \phi| \les  \jb{\tau_+} |\nt{\L_{Z_1}F}| (\tau^{-1}|D_\Omega \phi|+\frac{\tau}{\tau_+} |D \phi|).
\end{equation*}

 We can treat the term of  $F_{Z_1Z_2}$ by using  estimates (\ref{10.27.8.18}) and (\ref{10.27.4.18}),
 \begin{equation}
 \label{eq:bd4FZZ}
 |F_{Z_1Z_2}|\les \jb{\tau_+}^2|\nt{F}| \les \frac{\tau_+}{\tau},  \mbox{ if } Z_1, Z_2 \in \Pp\cup\{S\}.
 \end{equation}

 In view of  the above two estimates, also by using  Proposition \ref{thm:pointdecay:derScal}, we can further derive
 \begin{align}
   |Q(\L_{Z_1} F, \phi, Z_2)|
  &\les \tau_+|\nt{ \L_{Z_1} F}|\big(\tau^{-1}|D_\Omega  \phi|+\frac{\tau}{\tau_+}| D\phi|\big)+ |\L_{Z_1} F| |\phi| \nn\\
   &\quad +(|D_Z\phi|   + \tau_+^{-1}|D_Z^2 \phi|+\tau^{-1} |\phi| )\dn^2 \tau_+^{-2+2\delta}\nn\\
  &\les \frac{\tau_+}{\tau}| \nt{\L_{Z_1} F} | |D_\Omega \phi |+\tau |\nt{\L_{Z_1} F}||D\phi| + |\L_{Z_1}F ||\phi|+\dn^2 \tau_+^{-2+2\delta}|D_Z\rp{\le 2}\phi|.\label{10.27.11.18}
 \end{align}
 The last term in (\ref{10.27.11.18}) is favorable due to the sufficient decay in $\tau_+$. To control the remaining terms, recall from (\ref{11.13.4.18}),  (\ref{BA}) and due to (\ref{eq:decomposition4F}) that
 \begin{equation}\label{1.25.1.19}
 \|\tau^\f12 \nt\L_Z^l F\|_{L^2(\H_\tau)}\les \dn+1\les 1, \forall\, Z^l\in \Pp^l, \, l\le 2.
 \end{equation}
 Using the above estimate,  applying (\ref{sobsclp})  to  $(\frac{t}{\tau})^\f12 D_\Omega \phi$ with $p=4$, and applying Corollary \ref{cor_12_30} to $\nt\L_{Z_1}F$, we  derive
 \begin{align*}
 \| (\frac{\tau_+}{\tau})^\frac{3}{2} |\nt{ \L_{Z_1} F}| |D_\Omega \phi|\|_{L^2(\H_\tau)}&\les \tau^{-2}\| \tau_+^\f12 \tau^\frac{3}{4} \nt{\L_{Z_1} F}\|_{L^4(\H_\tau)}\|(\frac{\tau_+}{\tau})^\f12\tau_+^\frac{1}{2}\tau^\frac{1}{4}  D_\Omega \phi\|_{L^4(\H_\tau)}\\
&\les\tau^{-2} \|\tau^\f12\nt\L\rp{\le 1}_\Omega{\L_{Z_1} F}\|_{L^2(\H_\tau)}\E^{T_0}_{\le 2}[\phi]^\f12(\tau)\les \tau^{-2}\E_{\le 2}^{T_0}[\phi]^\f12(\tau).
 \end{align*}
 Applying Corollary \ref{cor_12_30} to  $(\frac{\tau}{t})^\f12 \nt{ \L_{Z_1} F}$ and  (\ref{sobsclp}) to $(\frac{t}{\tau})^\f12 D_\p \phi$ yields
 \begin{align*}
 \|\tau_+ |\nt{\L_{Z_1} F}|| D\phi|\|_{L^2(\H_\tau)}&\les\tau^{-1}\| (\frac{\tau}{\tau_+})^\f12\tau^{\f12} (\tau_+^2{\tau})^\frac{1}{4}\nt{ \L_{Z_1} F}\|_{L^4(\H_\tau)}\| \tau_+^\f12\tau^\frac{1}{4} (\frac{\tau_+}{\tau})^\f12 D\phi\|_{L^4(\H_\tau)}\\
 &\les \tau^{-1}\E_{\le 2}^{T_0}[\phi]^\f12(\tau)  \|(\frac{\tau}{\tau_+})^\f12\tau^\f12\nt{\L\rp{\le 1}_\Omega \L_{Z_1} F}\|_{L^2(\H_\tau)}.
 \end{align*}
 We  use (\ref{10.27.8.18}) to bound $|\frac{\tau}{\tau_+}\L_{Z_1}F|\les |\nt{\L_{Z_1}F}|$. It then follows by using (\ref{11.10.6.18})  and (\ref{1.25.1.19}) that
 \begin{align*}
 \|(\frac{\tau_+}{\tau})^\f12 |\L_{Z_1} F||\phi|\|_{L^2(\H_\tau)}&\les \|\frac{\tau}{\tau_+}\L_{Z_1} F|\|_{L^2(\H_\tau)}\| (\frac{\tau_+}{\tau})^\frac{3}{2}\phi\|_{L^\infty(\H_\tau)}\\
 &\les \tau^{-2}\E_{\le 2}^{T_0}[\phi]^\f12(\tau) \|\tau^\f12 \nt{\L_{Z_1} F}\|_{L^2(\H_\tau)}\\
 &\les  \tau^{-2}\E_{\le 2}^{T_0}[\phi]^\f12(\tau).
 \end{align*}

 We thus summarize the above three calculations in view of (\ref{10.27.11.18}) and (\ref{BA})
 \begin{equation}\label{10.27.12.18}
 \|(\frac{\tau_+}{\tau})^\f12 Q(\L_{Z_1} F, \phi, Z_2)\|_{L^2(\H_\tau)}\les \E_{\le 2}^{T_0}[\phi]^\f12(\tau) ( \tau^{-1}\|(\frac{\tau}{\tau_+})^\f12 \tau^\f12 \nt{\L\rp{\le 1}_\Omega\L_{Z_1} F}\|_{L^2(\H_\tau)}+\tau^{-2+2\delta}).
 \end{equation}

In view of (\ref{11.13.4.18}), the fact that $\tau^2=\tau_+\tau_{-}$ and (\ref{BA}),
\begin{equation*}
\|(\frac{\tau}{\tau_+})^\f12 \tau^\f12 \nt{\L_\Omega\rp{\le 1}\L_{Z_1} F}\|_{L^2(\H_\tau)}\les \E^S_{\le 2} [\tF]^\f12(\tau)+\tau^{-\ep}\les \tau^{-\ep}+\dn.
\end{equation*}
Substituting this estimate  into (\ref{10.27.12.18}), by using the already proven $k=0,1$ cases in (\ref{eqn:BT:imp:scal}), $0<\delta\le \f12(1-\ep)$ and (\ref{BA}), we obtain the bound of (\ref{1.05.2.19}) for $Z_1\sqcup Z_2=Z^2$
\begin{align*}
   \iint_{\mathcal{D}^\tau}|D D_Z^2\phi||Q(\L_{Z_1} F, \phi; Z_2)|&\les \int_{\tau_0}^{\tau}  \mathcal{E}_{2}^{T_0}[\phi](s) (s^{-1-\epsilon}+s^{-1}\dn )ds+(\varepsilon^2_0+\dn^3) \jb{\tau}^{2\delta}\\
   &\les  \int_{\tau_0}^{\tau}  \mathcal{E}_{2}^{T_0}[\phi](s) s^{-1-\epsilon} ds +(\varepsilon^2_0+\dn^3) \jb{\tau}^{2\delta}.
\end{align*}
Hence we conclude for  the error integral that
\begin{align*}
  \iint_{\mathcal{D}^\tau}|\p^\mu \ET[D_Z^2\phi]_{\mu 0}|\les \int_{\tau_0}^{\tau}  \mathcal{E}_{2}^{T_0}[\phi](s) s^{-1-\epsilon} ds +(\varepsilon_0^2+\dn^3) \jb{\tau}^{2\delta}.
\end{align*}
From (\ref{1.05.1.19}), this implies that
\begin{equation*}
  \mathcal{E}_{2}^{T_0}[\phi](\tau)\les  \int_{\tau_0}^{\tau}  \mathcal{E}_{2}^{T_0}[\phi](s) s^{-1-\epsilon} ds +(\varepsilon_0^2+\dn^3) \jb{\tau}^{2\delta}.
\end{equation*}
\eqref{eqn:BT:imp:scal} for $k=2$ then follows by using Gronwall's inequality.

\subsubsection{Energy estimates for the Maxwell fields }\label{maxwell_1}
		
Note that for the scaling vector field $S=t\p_t+r\p_r$ and  any 2-form $G$,
\[
\T[G]_{\mu\nu} (\pi^S)^{\mu\nu}= \T[G]_{\mu\nu} \textbf{m}^{\mu\nu} =0.
\]		
Applying Proposition \ref{prop::EnergyEqMKGModified} with $(0, \L_Z^k \tF)$ and $X=S$, in view of Proposition \ref{11.8.1.18}, we have 
\begin{align}
			\int_{\H_\tau}\ET[\L_Z^k \tF](S,\Tb)&\le \int_{\H_{\tau_0}}\ET[\L_Z^k \tF](S,\Tb)+\int_{C_0^\tau}\ET[\L_Z^k \tF](S,L)+|\iint_{\mathcal{D}^\tau}\p^\mu(\ET[\L_Z^k\tF])_{\mu\nu}S^\nu  |\nn\\
&\les \ve_0^2+ |\iint_{\mathcal{D}^\tau}\p^\mu(\ET[\L_Z^k\tF])_{\mu\nu}S^\nu |,\quad \tau_0<\tau\le \tau_*. \label{1.05.3.19}
\end{align}
  It is important that the absolute value in the last integral is taken after the integration. This term will be treated by  an integration by parts when $k=2$. Using (\ref{eqn::mMKG}), the equation for $\L_Z^k \tF$ in Lemma \ref{lem::Commutator1} and applying \eqref{eq:div4TfG} to $(0, \L_Z^k \tF)$, we have
		\begin{equation}\label{10.27.13.18}
			\p^\mu(\ET[\L_Z^k\tF])_{\mu\nu}S^\nu=J[\L_Z^k\tF]^{\mu}S^{\nu} (\L_Z^k\tF)_{\mu\nu}= (\L_Z^k J[\phi])^{\mu}S^{\nu} (\L_Z^k\tF)_{\mu\nu}.
		\end{equation}
For any vector field $\sJ$ and 2-form $G$,  we adopt the radial decomposition by using the radial normal $\Nb$ and  the orthonormal basis $\{e_B\}_{B=1}^2$ of $S_{\tau, r}$:
$$
  \sJ^{\mu} S^{\nu}G_{\mu\nu}=\sJ^{\Nb} G_{\Nb S}+\sJ^{B} G_{BS}.
$$
 By using (\ref{11.10.7.18}), we derive that for any $1$-tensor $\sJ$
\begin{equation}
\label{eq:keyObservation}
  |\sJ^{\mu} S^{\nu}G_{\mu\nu}|\les  |\sJ(\B)| (|G_{\Tb \Nb}|+\tau_+^{-1} \tau|G_{\Tb B}|)\les |\sJ(\Omega)||\nt{G}|,
\end{equation}
where $|\sJ(\Omega)|=\sum_{Z\in \Omega} |\sJ(Z)|$.

Thus we can conclude from (\ref{10.27.13.18}) that
\begin{equation}\label{10.27.14.18}
|\p^\mu(\ET[\L_Z^k\tF])_{\mu\nu}S^\nu|\les |\L_Z^k J[\phi](\Omega)| |\nt{\L_Z^k\tF}|, \quad k\le 2.
		\end{equation}

\medskip
		\paragraph{\bf Step 3.} Estimate of $\mathcal{E}_{\le 1}^{S}[\tF](\tau)$.
Now for the case when $k=0$, from (\ref{10.27.13.18}), the decay property of $\phi$ in (\ref{11.11.4.18}) and (\ref{eq:keyObservation}), we derive by using (\ref{BA}) that
\begin{align*}
  |\iint_{\mathcal{D}^\tau}\p^\mu(\ET[ \tF])_{\mu\nu}S^\nu \ | &\les \iint_{\mathcal{D}^\tau}|J[\phi]^{\mu}S^\nu \tF_{\mu\nu}|\les \iint_{\mathcal{D}^\tau}|\phi| |D_{\Omega}\phi||\nt{\tF}| \\
 &\les  \dn \int_{\tau_0}^{\tau}    s^{-\frac{3}{2}+\delta} (\int_{\H_{s}} \frac{s_+}{s}|D_{\Omega}\phi|^2)^{\frac{1}{2}}(\int_{\H_{s}}\frac{s_{-}}{s} |\nt{\tF}|^2)^{\frac{1}{2}}\\
 &\les \dn^3 \int_{\tau_0}^{\tau}    s^{-2+2\delta} ds\les \dn^3.
\end{align*}
Substituting the estimate into (\ref{1.05.3.19}) yields
\begin{equation*}
  \mathcal{E}_{0}^{S}[\tF](\tau)\les \ve_0^2+ \dn^3.
\end{equation*}
Thus \eqref{eqn:BT:imp:Max} holds for $k=0$.

 With $k=1$ in (\ref{10.27.14.18}), it follows by using Lemma \ref{lem::Commutator1} that
\begin{equation*}
  |\p^\mu(\ET[\L_Z^k\tF])_{\mu\nu}S^\nu|\les  (\sum\limits_{l_1+l_2\leq 1}|D_Z^{l_1}\phi||D_{\Omega} D_{Z}^{l_2}\phi|+ |F_{Z\Omega}||\phi|^2)|\nt{\L_Z \tF}|.
\end{equation*}
Now from \eqref{eq:bd4FZZ} and the pointwise bounds for $\phi$ in (\ref{11.11.4.18}), we can show that
\begin{align*}
  \sum\limits_{l_1+l_2\leq 1}|D_Z^{l_1}\phi||D_{\Omega} D_{Z}^{l_2}\phi|+ |F_{Z\Omega}||\phi|^2\les \dn\tau^{-1+\delta}\tau_{+}^{-\frac{1}{2}}|\phi|+\dn \tau_+^{-\frac{3}{2}+\delta}|D_{\Omega}D_Z\phi|+|D_{Z}\phi||D_{\Omega}\phi|.
\end{align*}		
 Therefore by using  (\ref{BA}),  we can bound
 \begin{align*}
   &|\iint_{\mathcal{D}^\tau}\p^\mu(\ET[ \L_Z \tF])_{\mu\nu}S^\nu \ | \\
   &\les  \int_{\tau_0}^{\tau}   \big(\int_{\H_{s}}  \dn^2 s^{-3+2\delta}(|D_{\Omega}D_Z\phi|^2+|\phi|^2)+|D_Z\phi|^2 |D_{\Omega}\phi|^2\big)^\f12 \cdot (\int_{\H_{s}}|\nt{\L_Z\tF}|^2)^\f12\\
 &\les \dn^3 
  + \dn \int_{\tau_0}^{\tau}    s^{-\frac{1}{2}}(\int_{\H_{s}}  |D_Z\phi|^2 |D_{\Omega}\phi|^2)^{\frac{1}{2}}ds.
 \end{align*}
To treat the last line, the decay estimate  $|D_{Z}\phi| \les \dn\tau_+^{-\frac{1}{2}}\tau^{\delta}$ in Proposition \ref{thm:pointdecay:derScal}  is insufficient to give  the boundedness of the above integral uniformly. Since a similar term  will appear again when $k=2$ in (\ref{1.05.3.19}), we employ the $L^p$ estimates of Sobolev embedding in Lemma \ref{sob_1} and Proposition \ref{sob} to treat such terms in the following result.
\begin{lemma} \label{prop:bd4DphiDDphi}
		For  $k\leq 2$ and $X\in \Pp$,  there holds
\begin{equation*}
  (\int_{\H_s}  |D_X\phi|^2 |D_{Z}^k \phi|^2)^{\frac{1}{2}}\les \dn^2 \jb{s}^{-1+2\delta}, \quad\mbox{ with } Z^k\in \Pp^k,\, s\ge \tau_0.
\end{equation*}
	\end{lemma}
\begin{proof}
Note that $|\B((\frac{t}{s})^p)|+|s\ud \nabla ((\frac{t}{s})^p)|\les (\frac{t}{s})^p$ holds for $p\in {\mathbb R}$.  By using  (\ref{9.16.1}), (\ref{sobsclp}) and (\ref{BA})  we have
  \begin{align*}
    s\|  |D_X\phi| |D_{Z}^k \phi|\|_{L^2(\H_s)} &\les \|(\frac{s}{t})^{\frac{1}{2}} D_{Z}^k \phi\|_{L^3(H_s)}\|(\frac{t}{s})^{\frac{1}{2}}s D_X\phi\|_{L^6(\H_s)}\\
    &\les \|(\frac{s}{t})^\f12 s^{-1}(s \ud D)\rp{\le 1} D_{Z}^k\phi \|_{L^2(\H_s)}   \|(\frac{t}{s})^\f12 D\rp{\le 1}_{\Omega}D_X\phi\|_{L^2(\H_s)}\\
    &\les \E^{T_0}_k[ \phi]^\f12(s)\E^{T_0}_{\le 2}[ \phi]^\f12(s)\les \dn^2 \jb{s}^{2\delta}.
  \end{align*}
  \end{proof}
From this lemma, we  can derive in view of (\ref{1.05.3.19}) that
\begin{align*}
  \mathcal{E}_{1}^{S}[\tF](\tau)\les \ve_0^2+|\iint_{\mathcal{D}^\tau}\p^\mu(\ET[ \tF])_{\mu\nu}S^\nu \ | &\les \ve_0^2+\dn^3 
  + \dn \int_{\tau_0}^{\tau}    s^{-\frac{1}{2}}(\int_{\H_{s}}  |D_Z\phi|^2 |D_{\Omega}\phi|^2)^{\frac{1}{2}}ds\\
  &\les \ve_0^2+\dn^3
  + \dn^3 \int_{\tau_0}^{\tau}    s^{-\frac{1}{2}-1+2\delta}ds\les \ve_0^2+\dn^3 .
\end{align*}
Therefore  \eqref{eqn:BT:imp:Max} holds for $k=1$.
\medskip
\paragraph{\bf Step 4.} Estimate of $\mathcal{E}_{2}^{S}[\tF](\tau)$.

 Recall from (\ref{1.05.3.19}) and (\ref{10.27.13.18}) that
 \begin{equation}\label{10.28.1.18}
		\mathcal{E}_{2}^{S}[\tF](\tau)\les \ve_0^2+ \dn^3+\sum_{Z^2\in \Pp^2}|\iint_{\mathcal{D}^\tau}(\L_Z^2 J[\phi])^{\mu}S^{\nu} (\L_Z^2\tF)_{\mu\nu}|.
		\end{equation}
  In view of Lemma \ref{lem::Commutator1}, $\L_Z^2 J[\phi]$ can be expanded as several nonlinear terms. Following the treatment of (\ref{10.27.14.18}) in Step 3, which worked  well for estimating $\mathcal{E}_{1}^{S}[\tF](\tau)$, results in losing the uniform boundedness of the top order energy $\mathcal{E}_{2}^{S}[\tF](\tau)$.

We observe that the obstacle lies in controlling the ``low-high" interaction  term
$\mathscr{F}_\mu=\Im(\phi\cdot\overline{D_\mu D_Z^2\phi})$
 in $\L_Z^2 J[\phi]$ (see the formula in Lemma \ref{lem::Commutator1}). By (\ref{10.27.14.18}),
 $$ |S^\nu\mathscr{F}^\mu \L_Z^2 \tF_{\mu\nu}|\les |\mathscr{F}(\Omega)||\nt\L_Z^2 \tF|.$$
 Note that $|\mathscr{F}(\Omega)|\le |\phi||D_{\Omega}D_Z^2\phi|.$ The higher order derivative reaches $D_Z^3,\, Z^3\in \Pp^3$. Meanwhile we only expect to nicely control energies induced by two commuting vector fields for $\phi$. By a careful analysis, such treatment leads to a growth of $\tau^{2\delta}$ for the top order energy $\mathcal{E}_{2}^{S}[\tF](\tau)$.

For this term, we carry out integration by parts to pass one derivative from the highest order term  to the lower order terms. The integration by parts on $$\I=\abs{\iint_{\D^\tau} \Im(\phi\c \overline{D^\mu D_Z^2\phi})\L_Z^2 \tF_{\mu S}}$$ imposes a very high order derivative on  Maxwell field, i.e.
$
\p^\mu\L_Z^2 \tF_{\mu \nu}.
$
 Since we only expect the bound of $\|\L_Z\rp{\le 2} \tF_{\mu\nu}\|_{L^2(\H_\tau)}$, this term already exceeds the expectation by one order of derivative. Fortunately, by using the Maxwell equation in (\ref{eq:eq4lFandNLF}) and  the commutation property in Lemma \ref{lem::Commutator1}, this term is actually $-\L_Z^2J[\phi]_\mu$. The original trilinear term in the integrand of $\I$  is now reduced to quartic terms of $\phi$ and its derivatives, which have much better asymptotic behavior. Such a trilinear estimate for treating  the last integral in (\ref{10.28.1.18}) exploits the additional cancelation structure supplied by the Maxwell equation. It enables us to  achieve sharp control on the Maxwell field.

More precisely, since $\p^\mu S^{\nu}=\textbf{m}^{\mu\nu}$ is symmetric while the 2-form $(\L_Z^2 \tF)_{\mu\nu}$ is antisymmetric, we compute that
\begin{align*}
 &S^{\nu} \Im(\phi\cdot\overline{D^\mu D_Z^2\phi}) (\L_Z^2\tF)_{\mu\nu}\\
 &= S^{\nu} (\p^\mu \Im(\phi\cdot\overline{ D_Z^2\phi}) -\Im(D^\mu \phi \cdot \overline{D_Z^2\phi}))(\L_Z^2\tF)_{\mu\nu}\\
 &= \p^\mu(S^{\nu} \Im(\phi\cdot\overline{ D_Z^2\phi}) (\L_Z^2\tF)_{\mu\nu}) -S^{\nu} \Im(D^\mu \phi \cdot \overline{D_Z^2\phi})(\L_Z^2\tF)_{\mu\nu}- S^{\nu} \Im(\phi\cdot\overline{D_Z^2\phi}) \p^\mu(\L_Z^2 \tF)_{\mu\nu}\\
 &= \p^\mu(S^{\nu} \Im(\phi\cdot\overline{ D_Z^2\phi}) (\L_Z^2\tF)_{\mu\nu}) -S^{\nu} \Im(D^\mu \phi \cdot \overline{D_Z^2\phi})(\L_Z^2\tF)_{\mu\nu}+ S^{\nu} \Im(\phi\cdot\overline{D_Z^2\phi}) (\L_Z^2 J[\phi])_{\nu}.
\end{align*}
 For the second term we can apply  \eqref{eq:keyObservation}  together with Lemma \ref{prop:bd4DphiDDphi}. In view of Lemma \ref{lem::Commutator1}, the third term mainly consists of  quartic  terms of scalar field and its derivatives which provides sufficient decay to achieve the boundedness of the integral. Thus, with $\textbf{n}^\mu$ the surface normal of $\p \D^\tau$, these treatments lead to
\begin{align}
 \I &\leq |\int_{\p\mathcal{D}^\tau} \Im(\phi\cdot\overline{ D_Z^2\phi}) (\L_Z^2\tF)_{{\mathbf{n}} S}|+ \iint_{\mathcal{D}^\tau}|D_{\Omega}\phi||D_Z^2\phi| |\nt{\L_Z^2\tF}|\nn\\
  & + \iint_{\mathcal{D}^\tau} |\phi||D_Z^2\phi|  \left( \sum_{l_1+l_2\le 2}|D_Z^{l_1}\phi||D_S D_{Z}^{l_2}\phi|+ \sum_{X, Y\in \Pp}(  |(\L_X F)_{Y S}||\phi|^2 +|F_{X S}||\phi| |D\rp{\le 1}_Z \phi|)\right).\nn
  \end{align}
  The second term on the right can be bounded by  $\dn^3\int_{\tau_0}^\tau s^{-\frac{3}{2}+2\delta} ds$, by using Lemma \ref{prop:bd4DphiDDphi} and (\ref{BA}).

 Note that if $l_1=1$ in the last line, we apply Lemma \ref{prop:bd4DphiDDphi}, (\ref{11.11.4.18}) and (\ref{BA}) to bound  the corresponding quartic term  of scalar fields by
 \begin{align*}
 \iint_{\mathcal{D}^\tau} |\phi||D_Z^2\phi|  |D_Z \phi||D_S D_{Z}^{\le 1}\phi|&\les \dn \iint_{\D^\tau} s_+^{-\frac{3}{2}+\delta}|s D_\Tb D_Z\rp{\le 1}\phi|  |D_Z^2\phi|  |D_Z \phi|\\
 &\les \dn^2 \int_{\tau_0}^\tau s^{-\f12+2\delta} \||D_Z \phi||D_Z^2 \phi|\|_{L^2(\H_s)} ds\\
& \les \dn^4\int_{\tau_0}^\tau s^{-\frac{3}{2}+4\delta}.
 \end{align*}
 For other combinations of $l_1+l_2\le 2$, we rely on  (\ref{11.16.1.18}) to obtain $|D_S f|\les |s_+ D f|$ for a smooth scalar field $f$, and on (\ref{11.11.4.18})  to bound $|D_S\rp{\le 1}\phi|$, which gives
 \begin{align*}
 \iint_{\mathcal{D}^\tau} |\phi||D_Z^2\phi|\left(|\phi||D_S D_{Z}^{\le 2}\phi|+|D_Z^{\le 2}\phi||D_S \phi|\right)\les \iint_{\D^\tau} \dn^2 s_+^{-2+2\delta}|D\rp{\le 1}D_Z\rp{\le 2} \phi||D_Z^2 \phi|.
 \end{align*}
Moreover,  the integral  on the boundary $\p\mathcal{D}^{\tau}$  vanishes on  hyperboloids $\H_\tau$ and $\H_{\tau_0}$ since their surface normal is $\Tb=\tau^{-1}S$ and $\L_Z^2 \tF$ is a 2-form. Thus it is  bounded merely by the terms on the null boundary $C_0^\tau$, controlled by  using $\tau_{-}=R$, $|\phi|\les\ve_0 \tau_+^{-\frac{3}{2}}$ in Theorem \ref{ext_stb} and  (\ref{11.7.3.18})
\begin{align*}
\int_{C_0^\tau}|(\L_Z^2 \tF)_{L S}||\Im(\phi\cdot \overline{D_Z^2 \phi})|&\les \int_{C_0} \tau_{-} |\rho[\L_Z^2 \tF]||\Im(\phi\cdot \overline{D_Z^2 \phi})|\\
&\les \ve_0 \|\rho[\L_Z^2 \tF]\|_{L^2(C_0)}\|D_Z^2 \phi\|_{L^2(C_0)}\les \ve_0^3.
\end{align*}

Summarizing the above treatments  leads to
  \begin{align*}
\I  &\les\ve_0^3+ \dn^3 
+   \iint_{\mathcal{D}^\tau}|D_Z^2\phi|\dn^2 s_+^{-2+2\delta} \big ( |D D_Z\rp{\le 2}\phi|+|D_Z\rp{\le 2}\phi|+ \sum_{X\in\Pp}|\nt\L_X F |\dn s_+^{-\frac{1}{2}+\delta} \big).
\end{align*}
 For deriving the last inequality, we employed (\ref{eq:bd4FZZ}) and  the second pair of estimates in (\ref{10.27.8.18}) to treat $|F_{XS}|$ and $|(\L_X F)_{YS}|$ respectively; and we also used  (\ref{11.11.4.18}) to bound $|\phi|$.

 It follows by using (\ref{BA}) for energies of scalar fields and (\ref{1.25.1.19}) for the curvature term that, for $k\le 2$, $X\in \Pp$,
\begin{equation*}
\begin{split}
  &\iint_{\mathcal{D}^\tau}s_+^{-2+2\delta}|D_Z^2\phi|  (  |D D_Z^k\phi|+|D_Z^k\phi|+|\nt\L_X F |\dn s_+^{-\frac{1}{2}+\delta} )\\
  &\les \int_{\tau_0}^{\tau}s^{-2+2\delta} ( \int_{\H_{s}}\frac{s_+}{s} |D_Z^2\phi|^2 )^{\frac{1}{2}} \left( \int_{\H_{s}}\frac{s_{-}}{s} (|D D_Z^k\phi|^2+|D_Z^k\phi|^2+ |\nt\L_X F |^2\dn^2 s_+^{-1+2\delta}) \right)^{\frac{1}{2}}\\
  &\les \int_{\tau_0}^{\tau}s^{-2+4\delta} \dn^2 ds\les \dn^2.
  \end{split}
\end{equation*}

Combining the above estimates yields
\begin{align*}
  \I=|\iint_{\mathcal{D}^\tau}S^{\nu} \Im(\phi\cdot\overline{D^\mu D_Z^2\phi}) (\L_Z^2\tF)_{\mu\nu} \ |\les \ve_0^3+\dn^3+\dn^4\les \ve_0^3+\dn^3.
\end{align*}
 The main difficulty is thus solved. Going back to  (\ref{10.28.1.18}), we estimate the remaining terms in the formula of $\L_Z^2 J[\phi]$ in Lemma \ref{lem::Commutator1}, which are denoted by
\begin{equation*}
\sJ^+_\mu=\L_Z^2 J[\phi]_\mu-\Im(\phi\cdot\overline{D_\mu D_Z^2\phi}).
\end{equation*}

  In view of   \eqref{eq:keyObservation},  we first bound $\sJ^+(\Omega)$. From Lemma \ref{lem::Commutator1}, (\ref{eq:bd4FZZ}) and (\ref{11.11.4.18}), we can derive
\begin{align*}
  |\sJ^+(\Omega)|&\les |D_\Gamma \phi||D_{\Gamma}^2 \phi|+ |(\L_Z F)_{Z\Omega}||\phi|^2+|F_{\Ga\Omega}||\phi| |D\rp{\le 1}_Z \phi|\\
  &\les |D_\Gamma \phi||D_{\Gamma}^2 \phi|+ |(\L_Z F)_{Z\Omega}| |\phi|^2+\tau_+^{\frac{1}{2}}\tau_{-}^{-\frac{1}{2}} |\phi| |D\rp{\le 1}_Z \phi|\\
  &\les  |D_\Gamma \phi||D_{\Gamma}^2 \phi|+ \dn^2 \tau_+^{-3+2\delta}|(\L_Z F)_{Z\Omega} |+\dn|D_Z\rp{\le 1}\phi| \tau_+^{-1+\delta}\tau_{-}^{-\frac{1}{2}}.
\end{align*}
Here $|D_{\Gamma}^l\phi|$ stands for $\sum \limits_{\Gamma^l\in \Pp^l}|D^l_\Gamma\phi|$, for $l\le 2$.
By using (\ref{eq:keyObservation}) and (\ref{BA}), we  obtain
\begin{align*}
  |\iint_{\mathcal{D}^\tau}S^{\nu} \sJ^+_\mu \tensor{(\L_Z^2\tF)}{^\mu_\nu} \ |
  &\les \iint_{\mathcal{D}^\tau}|\sJ^+(\Omega)| |\nt{\L_Z^2\tF}|
 \les \dn \int_{\tau_0}^\tau s^{-\f12} \|\sJ^+(\Omega)\|_{L^2(\H_s)} ds.
 \end{align*}
  Therefore, by using  Lemma \ref{prop:bd4DphiDDphi},  Proposition \ref{thm:pointdecay:derScal}, the second pair of estimates in Lemma \ref{11.13.5.18} and (\ref{BA}), we can estimate
 \begin{align*}
  |\iint_{\mathcal{D}^\tau}S^{\nu} \sJ^+_\mu \tensor{(\L_Z^2\tF)}{^\mu_\nu} \ | &\les  \dn \int_{\tau_0}^{\tau}s^{-\frac{1}{2}} \big(s^{-\frac{3}{2}+\delta}\dn \|D\rp{\le 1}_\Gamma \phi\|_{L^2(\H_s)}+\||D _\Gamma \phi||D_\Gamma^2 \phi|\|_{L^2(\H_s)}\\
  &+\dn^2\| s_+^{-3+2\delta} (\L_\Gamma F)_{\Gamma\Omega}\|_{L^2(\H_s)}\big) \\
  &\les  \dn^3 \int_{\tau_0}^{\tau}s^{-\frac{1}{2}} \big( s^{-1+2\delta}+ \|s_+^{-1+2\delta} \nt{\L_\Gamma F}\|_{L^2(\H_s)}\big) \\
  &\les  \dn^3 \int_{\tau_0}^{\tau}s^{-\frac{1}{2}} ( s^{-1+2\delta}+ s^{-\frac{3}{2}+2\delta}  )ds\les \dn^3,
\end{align*}
where we have used (\ref{1.25.1.19}) to bound $\nt \L_\Ga F$.
Combining the above estimate with the one for $\I$ yields
\begin{align*}
 |\iint_{\mathcal{D}^\tau}S^{\nu} (\L_Z^2 J[\phi])^{\mu} (\L_Z^2\tF)_{\mu\nu} \ |&\les  |\iint_{\mathcal{D}^\tau}S^{\nu} \sJ^+_\mu \tensor{(\L_Z^2\tF)}{^\mu_\nu} \ |+ |\iint_{\mathcal{D}^\tau}S^{\nu} \Im(\phi\cdot\overline{D^\mu D_Z^2\phi}) (\L_Z^2\tF)_{\mu\nu} \ |\\
 &\les \dn^3+\ve_0^3.
\end{align*}
In view of (\ref{10.28.1.18}),  \eqref{eqn:BT:imp:Max} is proved for $k=2$.
	
	\subsection{Improving the bootstrap assumptions}
	
	Let $C\ge 1$ be the maximum of the implicit constants in \Cref{thm::EnergyDecay} and (\ref{11.7.2.18}), which is independent of $\varepsilon_0$ and $\dn$.   Let $\varepsilon_0, \dn$ verify the following conditions:
\begin{equation}\label{1.31.1.19}
		C\varepsilon_0^2 \le \frac{1}{4}\dn^2,\quad C\dn \le \frac{1}{4}.
	\end{equation}
	This improves the bootstrap assumptions (\ref{BA}) with each $\dn^2$ replaced by $\f12 \dn^2$. Hence Theorem \ref{them::mainThm} holds in the entire interior region $\D^+$ for all $\tau\ge\tau_0$ by the continuity argument.
	
		\section{Proof of Theorem \ref{ext_stb}}\label{ext}
\subsection{Set-up and bootstrap assumptions}
 Recall from  (\ref{eq:decomposition4F}) that the full Maxwell field $F$  is decomposed into the linear part $\hat{F}$, and the perturbation part $\tF$ with its chargeless part denoted by $\ck F$ (see \eqref{12.31.2.18}). This section is devoted to the proof of  Theorem \ref{ext_stb} on the solution in the exterior region $\{t-t_0\le r-R, t\ge t_0\}$ with various decay properties  on $(\ck F, \phi)$. The proof will be based on controlling the propagation of energies in the exterior region, for which, up to the second order,
 Proposition \ref{11.7.1.18} has provided the initial weighted energy on $\Sigma_{t_0}^e $ for $\ck F$, and the initial weighted energy of $\phi$ has been given by $\E_{2, \ga_0}$.

   The proof for Theorem \ref{ext_stb} has a similar part which treats the terms contributed by  $\tF$ in the nonlinear analysis as that of the main result in \cite{KWY}, including the terms contributed by the slow decaying charge. This set of calculations will be omitted to avoid repetition. We will focus on the analysis that involves the general large linear field $\hat{F}$ when proving Theorem \ref{ext_stb}.

  The control of the energy fluxes for the Maxwell field $\ck F$ is similar as in \cite{KWY}, with details given in Proposition \ref{prop:BT:imp:Max} and proved with the help of Proposition \ref{flux_ext}. The main difference lies in controlling the higher order energy fluxes for the scalar field, which is proved in Proposition \ref{11.4.3.18} with the help of the crucial Lemma \ref{com_est}. Note that when the signature for $D_Z^l\phi$, i.e. $\zeta(Z^l)$, $l\le 2$  has a smaller value, the error integral in the corresponding energy estimates is simpler and the fluxes are expected to decay better in $u_+$. We control  in Lemma \ref{com_est}  the error integral for bounding the flux of $D_Z^ l \phi$ for $l\le 2$,  and manage to show that the weak  decay terms  in the bound are either of lower  signature  than  $\zeta(Z^l)$ or of lower order derivatives than  $D_Z^l\phi$.
     We then bound the energy fluxes for $D_Z\phi$ inductively in terms of the value of $\zeta(Z)$. For the top order energy, we first bound the weighted flux on $\H_u^{-u_2}$  for $u_+^{-\zeta(Z^2)} D_Z^2 \phi$. This type of flux provides an important improvement over Lemma \ref{com_est}, which then gives the complete set of energy fluxes for the scalar fields in Theorem \ref{ext_stb}.

The decay properties for the linear Maxwell field in the exterior region will be important for the proof of Theorem \ref{ext_stb}.
Let $1<\ga_0<2$ be the fixed number  in the assumption of Theorem \ref{ext_stb}, with the given bound $\M_0>1$ in (\ref{asmp}).
 Recall that the linear Maxwell field $\hat{F}$ verifies the following equation
\begin{equation*}
\p^\mu \L_Z^k\hat F_{\nu\mu}=0,\quad \forall\, Z^k\in \Pp^k, \quad k\leq 2
\end{equation*}
by the definition in  (\ref{eq:eq4lFandNLF}) and the commutation property.
The multiplier approach in \cite{KWY} directly implies the following decay properties for $\hat F$ in the exterior region $\{r\ge t-t_0+R\}$ with $t\ge t_0$.
\begin{prop}
\label{linear_ex}
In the exterior region, the linear Maxwell field $\hat{F}$ verifies the following energy decay estimates
\begin{align}
\label{10.30.4.18}
&(u_1)_+^{\ga_0}  E[\L_Z^k \hat F](\H_{u_1}^{-u_2})+ (u_1)_+^{\ga_0} E[\L_Z^k\hat F](\Hb_{-u_2}^{u_1})
+\int_{\H_{u_1}^{-u_2}}r^{\ga_0}|\a[\L_Z^k \hat F]|^2\\
\notag
 &+\iint_{\mathcal{D}_{u_1}^{-u_2}}r^{\ga_0-1}|(\a, \rho,\sigma)[\L_Z^k\hat F]|^2\nn
+\int_{\Hb_{-u_2}^{u_1} }r^{\ga_0}(|\rho[\L_Z^k\hat F]|^2+|\si[\L_Z^k\hat F]|^2) \les (u_1)_+^{2\zeta(Z^k)}\M_0^2
\end{align}
as well as the pointwise decay estimates
\begin{equation}\label{10.31.2.18}
|\hat\a|^2+|\hat \rho|^2+|\hat \sigma|^2\les \M_0^2 r^{-2-\ga_0} u_+^{-1},\qquad |\hat \ab|^2 \les \M_0^2 r^{-2} u_+^{-\ga_0-1}
\end{equation}
for all $u_2<u_1\leq -\frac{R}{2}$, $Z^k\in \Pp^k$, $k\leq 2$. Here $\hat{\alpha}$, $\hat{\ab}$, $\hat{\rho}$, $\hat{\si}$ are the null components for $\hat{F}$.
There also hold
\begin{equation}\label{11.02.2.18}
\begin{split}
\|r^{\frac{\ga_0}{2}} \a[\L_Z\hat F], r^{\frac{\ga_0}{2}} \rho[\L_Z\hat F], r^{\frac{\ga_0}{2}}\si[ \L_Z\hat F], u_+^{\frac{\ga_0}{2}}\ab[\L_Z\hat F]\|^2_{ L^4(S_{u, v})} \les \mathcal{M}_0^2 u_+^{2\zeta(Z)-1}r^{-1}
\end{split}
\end{equation}
for all $-v\leq u\leq -\frac{R}{2}$ and $Z\in \Pp$, with the constant bounds in ``$\les$" depending merely on $\ga_0$ and $\ep$.
\end{prop}
\begin{proof}
The energy decay  and the weighted energy decay estimates in \eqref{10.30.4.18} follow by using $\p_t$ and $r^{\ga_0}L$ as multipliers. The  estimates \eqref{10.31.2.18} and \eqref{11.02.2.18}  are consequences of (\ref{10.30.4.18}) together with Sobolev embedding adapted to null hypersurfaces and spheres. For the details of the proof, we refer to \cite{KWY}.
\end{proof}

For convenience, we allow the constants  in ``$\les $" to depend on $\M_0$, $\ep$, $\ga_0$ in this section.
Recall the definition of $\ck F$ from (\ref{12.31.2.18}). Let $v_*$ be a large fixed number.  To prove Theorem \ref{ext_stb}, we make the following bootstrap assumptions on $(\ck F, \phi)$ in the exterior region,
\begin{align}
&E[D_Z^k\phi, \L_Z^k\ck{F}](\H_{u_1}^{-u_2})+E[D_Z^k\phi, \L_Z^k\ck{F}](\Hb_{-u_2}^{u_1})\leq \dn^2 (u_1)_+^{-\ga_0+2\zeta(Z^k)},\label{eq:BT:PWE:scal}\\
&\int_{\H_{u_1}^{-u_2}}r|D_L D_Z^k\phi|^2 +\int_{\Hb_{-u_2}^{u_1} }r(|\slashed{D} D_Z^k\phi|^2+|D_Z^k\phi|^2)  \leq \dn^2(u_1)_+^{1-\ga_0+2\zeta(Z^k)},
\label{eq:BT:PWE:Max}\\
&\int_{\H_{u_1}^{-u_2}}r^{\ga_0}|\a[\L_Z^k \ck{F}]|^2 +\iint_{\mathcal{D}_{u_1}^{-u_2}}r^{\ga_0-1}|(\a, \rho,\sigma)[\L_Z^k\ck{F}]|^2,\nn\\
&\qquad \qquad+\int_{\Hb_{-u_2}^{u_1} }r^{\ga_0}(|\rho[\L_Z^k\ck{F}]|^2+|\si[\L_Z^k\ck{F}]|^2) \leq \dn^2(u_1)_+^{2\zeta(Z^k)}, \label{eq:BT:energy}
\end{align}
for all $ -v_*\le u_2<u_1\leq -\frac{R}{2}$, $Z^k\in \Pp^k$, $k\leq 2$, where $\dn>0$, $\dn^2=C_1\E_{2,\ga_0}$ with $C_1\ge 4$ to be further determined. The  local well-posedness result guarantees  the above estimates hold for some $v_*>\frac{R}{2}$. Our goal is to achieve the same set of estimates with $\dn^2$ on the right hand sides all improved to be $\f12\dn^2$. Then by the continuity argument, the above estimates hold for all $u_2<u_1\le -\frac{R}{2}$.
It follows from the above bootstrap assumptions that
\begin{prop}\label{prop:decay:scal&Max}
Under the bootstrap assumptions \eqref{eq:BT:PWE:scal}-\eqref{eq:BT:energy}, there hold the decay estimates for the scalar field $\phi$ and the chargeless perturbation part $\ck{F}$ of the Maxwell field in $\D_{-\frac{R}{2}}^{v_*}$,
\begin{align}
\label{eq:decay:scal}
&r^{\frac{1}{2}+\epsilon}u_+^{\f12}|\phi|^2+r^2|\slashed{D}\phi|^2+u_+^2|D_{\Lb}\phi|^2+r^2|D_L \phi|^2 +u_+^{-2\zeta(Z)}|D_Z\phi|^2\les \dn^2 r^{-\frac{5}{2}+\ep}u_+^{\f12-\ga_0},\\
\label{eq:decay:Max:checkF}
&r^{-\ga_0}u_+^{\ga_0}|\ck{\ab}|^2+|\ck{\rho}|^2+|\ck{\a}|^2+|\ck{\si}|^2\les \dn^2 r^{-2-\ga_0}u_+^{-1},
\end{align}
where $Z\in \Pp$ and $u \le -\frac{R}{2}$.
\end{prop}
\begin{proof}
  See the proofs  in \cite[Sections 3.2 and 3.3]{KWY}. 
\end{proof}


\subsection{Energy decay estimates for the chargeless perturbation part of the Maxwell field}\label{ext_maxwell}
Now we  improve the bootstrap assumptions \eqref{eq:BT:PWE:scal}-\eqref{eq:BT:energy}.
 Based on Proposition \ref{prop:decay:scal&Max}, we first derive the (weighted) energy estimates for  $\ck{F}$, with data on $\Sigma_{t_0}^e$ controlled by Proposition \ref{data_1}.
\begin{prop}
\label{prop:BT:imp:Max}
Under the bootstrap assumptions \eqref{eq:BT:PWE:scal}-\eqref{eq:BT:energy}, $\ck{F}$ verifies the following energy decay estimates
\begin{align}
\label{eq:BT:imp:Max:energy}
&E[\L_Z^k \ck{F}](\H_{u_1}^{-u_2})+E[\L_Z^k \ck{F}](\Hb_{-u_2}^{u_1})\les (\mathcal{E}_{k, \ga_0}+\dn^4) (u_1)_+^{-\ga_0+2\zeta(Z^k)},
\end{align}
as well as the $r$-weighted energy decay estimates
\begin{align}
 & \iint_{\mathcal{D}_{u_1}^{- u_2}}r^{\ga_0-1} (|\a[\L_Z^k\ck F]|^2+|\rho[\L_Z^k\ck{F}]|^2+|\sigma[\L_Z^k\ck F]|^2)+\int_{\H_{u_1}^{-u_2}}r^{\ga_0}|\a[\L_Z^k\ck F]|^2
 \label{eq:BT:imp:Max:pwe}\\
 &+\int_{\Hb_{-u_2}^{u_1} }r^{\ga_0}(|\rho[\L_Z^k\ck{F}]|^2+|\sigma[\L_Z^k\ck F]|^2)\les (\mathcal{E}_{k, \ga_0}+\dn^4) (u_1)_+^{2\zeta(Z^k)}\nn
\end{align}
for all  $-v_*\le u_2<u_1\le-\frac{R}{2}$,  $Z^k\in \Pp^k$ and  $k\leq 2$ .
\end{prop}
Since $\L_{Z}^k {\ck{F}}$ satisfies the Maxwell equation $\p^\mu (\L_Z^k{\ck{F}})_{\nu\mu}=\L_Z^k J[\phi]_{\nu}$, the proof for the above proposition is the same as that for Proposition 4 in \cite{KWY} once we have the following estimates for the nonlinear terms.
\begin{prop}\label{flux_ext}
Under the bootstrap assumptions \eqref{eq:BT:PWE:scal}-\eqref{eq:BT:energy}, the nonlinear term $J[\phi]$ satisfies
\begin{equation*}
\begin{split}
\iint_{\mathcal{D}_{u_1}^{-u_2}} r^{\ga_0}u_+^{1+\ep}|\L_Z^k \slashed{J}|^2+r^{\ga_0+1}|\L_Z^k J_{L}|^2 +u_+^{1+\ga_0+\ep}|\L_Z^k J_{\Lb}|^2 \les \dn^4 (u_1)_+^{2\zeta(Z^k)}
\end{split}
\end{equation*}
for all $ -v_*\le u_2<u_1\le-\frac{R}{2}$,  $k\leq 2$ and $Z^k\in \Pp^k$, where $\L_Z^k \slashed{J}$, $\L_Z^k J_L$ and $\L_Z^k J_{\Lb}$ represent the angular, $L$ and $\Lb$ components of the 1-form $\L_Z^k J[\phi]$ respectively.
\end{prop}
\begin{proof}
  The proof is similar to \cite[Section 3.4]{KWY} except for treating  the terms with a factor that is  the full Maxwell field $F$, which arise from the commutation between the covariant derivatives $D$. By the decomposition $F=\hat F+\tF$,  we focus on treating the part containing $\hat F$, since  the rest of the calculations have been done in \cite{KWY}.
Note that  the case of $k=0$ does not involve any commutation. Hence Proposition \ref{flux_ext} holds as in \cite{KWY}.

  When $k=1$, recall the formula from Lemma \ref{lem::Commutator1}
  \begin{align*}
    \L_X J_{\mu}[\phi]
=\Im(D_X\phi \cdot \overline{D_\mu \phi} )+\Im(\phi\cdot \overline{D_\mu D_X \phi})- F_{X \mu}|\phi|^2.
  \end{align*}
  Estimates for the first two terms are the same (as that in \cite{KWY}) except the last term where the full Maxwell field $F$ arises from commutation.

   When $k=2$, recall that for $Z^2=XY$, there holds for any vector field or the frame in the null tetrad $V$,
  \begin{align*}
|V^\mu\L_X \L_Y J_\mu[\phi]|&\les |D_X D_Y\phi||D_V\phi|+|D_Y\phi| |D_V D_X\phi|+|F_{X V}||\phi||D_Y\phi|
+|D_X\phi||D_V D_Y \phi|\\&
+|\phi||D_V D_X D_Y \phi|+|(\L_X F)_{Y V }||\phi|^2+|F_{[X, Y]V}||\phi|^2+|F_{Y V}| |\phi||D_X\phi|.
\end{align*}
Since $\Pp$ is closed under the commutation $[\cdot,\cdot]$, if $[X,Y]\neq 0$, then $\zeta([X, Y])=\zeta(XY)$ for all $X$, $Y\in \Pp$, 
 we only need to estimate the terms  $|(\L_X F)_{Y V}||\phi|^2$ and
\begin{equation}\label{1.28.1.19}
 \sum_{\zeta(\tilde XZ^l)=\zeta(Z^2), \tilde X\in\Pp}|F_{\tilde X V}||\phi||D_Z^l\phi|,\,l\leq 1.
 \end{equation}
  In view of the decomposition $F=\hat{F}+\ck{F}+q_0 r^{-2}dt\wedge dr$ in (\ref{12.31.2.18}), the part of $\tF=\ck{F}+q_0 r^{-2}dt\wedge dr$ can be controlled in the same way as in \cite{KWY}, with a quicker treatment on the charge part due to the smallness of $|q_0|$  in (\ref{10.30.5.18}). It hence suffices to control the terms contributed by the large linear Maxwell field $\hat{F}$.

Consider the part $|\hat{F}_{\tilde X V}||\phi||D_Z^l\phi|$, $l\le 1$ in (\ref{1.28.1.19}), with $V$ a component of the null tetrad.  Firstly, according to the decay estimates in Proposition \ref{linear_ex}, we have for $X\in \Pp$
\begin{align*}
  |\hat{F}_{X L}| \les r^{\frac{-\ga_0}{2}}u_+^{-\frac{1}{2}+\zeta(X)},  \quad |\hat{F}_{X \Lb}| \les u_+^{-\frac{1+\ga_0}{2}+\zeta(X)}, \quad |\hat{F}_{X e_A}|\les r^{\frac{-\ga_0}{2}}u_+^{-\frac{1}{2}+\zeta(X)},\,\,  A=1,2.
\end{align*}
 Therefore by using the decay estimate for the scalar field $\phi$ in Proposition \ref{prop:decay:scal&Max}, \eqref{eq:BT:PWE:Max} and $\zeta(\tilde X Z^l)=\zeta(Z^2)$, we can estimate that
\begin{align}
  &\iint_{\mathcal{D}_{u_1}^{-u_2}} (r^{\ga_0}u_+^{1+\ep}|\hat{F}_{\tilde X e_A}|^2+r^{\ga_0+1}|\hat{F}_{\tilde X L}|^2 +u_+^{1+\ga_0+\ep}|\hat{F}_{\tilde X \Lb}|^2)|\phi|^2 |D_Z^l \phi|^2\label{2.13.2.19}\\
  &\les \iint_{\mathcal{D}_{u_1}^{-u_2}} u_+^{\ep+2\zeta(\tilde X)-1-\ga_0} \dn^2 r^{-2} |D_Z^l \phi|^2\nn\\
  &\les \dn^2 (u_1)_+^{2\zeta(\tilde X)+\ep-3-\ga_0}\int_{u_2}^{u_1} E[D_Z^l\phi](\H_{u}^{-u_2})du\nn\\
  &\les \dn^4 (u_1)_+^{2\zeta(\tilde X Z^l)}=\dn^4 (u_1)_+^{2\zeta(Z^2)}.\nn
\end{align}
Note if $l=0$ in (\ref{2.13.2.19}), the bound of the inequality is $\dn^4 (u_1)_+^{2\zeta(\tilde X)}$. This controls the term contributed by $\hat F_{\tilde X\mu} |\phi|^2$ in the $k=1$ case.

It now remains to understand $|(\L_X \hat{F})_{YV}| |\phi|^2$, for which using the pointwise estimate for $\phi$ in Proposition \ref{prop:decay:scal&Max} and the estimate (\ref{10.30.4.18}), we can bound
\begin{align*}
  &\iint_{\mathcal{D}_{u_1}^{-u_2}} (r^{\ga_0}u_+^{1+\ep}|(\L_X \hat{F})_{Y e_A}|^2+r^{\ga_0+1}|(\L_X \hat{F})_{Y L}|^2 +u_+^{1+\ga_0+\ep}|(\L_X \hat{F})_{Y\Lb}|^2)|\phi|^4 \\
  &\les \iint_{\mathcal{D}_{u_1}^{-u_2}} \dn^4 r^{-4} u_+^{\ep-2\ga_0+2\zeta(Y)}(r^{\ga_0}u_+ |\ab[\L_X \hat{F}]|^2+r^{\ga_0+1}|(\a, \rho, \si)[\L_X \hat{F}]|^2 )\\
  &\les \dn^4 (u_1)_+^{2\zeta(XY)}.
\end{align*}
Combining the above two estimates  leads to the estimates of Proposition \ref{flux_ext}  for $k=2$. Hence the proposition is proved.
\end{proof}

\subsection{Energy decay estimates for the scalar field}
In this subsection, we derive the following energy decay estimates for the scalar field
\begin{prop}\label{11.4.3.18}
Under the bootstrap assumptions \eqref{eq:BT:PWE:scal}-\eqref{eq:BT:energy}, the scalar field verifies the following energy decay estimates
\begin{align}
\label{eq:BT:imp:scal:energy}
&E[D_Z^k \phi](\H_{u_1}^{-u_2})+E[D_Z^k \phi](\Hb^{u_1}_{-u_2})\les (\E_{k,\ga_0}+\dn^4) (u_1)_+^{-\ga_0+2\zeta(Z^k)}, \\
\label{eq:BT:imp:scal:PWE}
&\int_{\H_{u_1}^{-u_2}}r|D_L D_Z^k\phi|^2 +\int_{\Hb_{-u_2}^{u_1} }r(|\slashed{D} D_Z^k\phi|^2+|D_Z^k\phi|^2) \les (\E_{k, \ga_0} +\dn^4)(u_1)_+^{1-\ga_0+2\zeta(Z^k)}
\end{align}
for all  $-v_*\le u_2<u_1\le-\frac{R}{2}$,  $Z^k\in \Pp^k$ and $k\le 2$.
\end{prop}
\begin{remark}
 Let $C\ge 1$ be the sum of all the implicit constants in ``$\les $" in  Proposition \ref{prop:BT:imp:Max} and Proposition \ref{11.4.3.18}. By choosing
$ C \E_{2,\ga_0}\le\frac{\dn^2}{4},\, C \dn^2 \le \frac{1}{4}$,  the same inequalities in (\ref{eq:BT:PWE:scal})-(\ref{eq:BT:energy}) hold with $\dn^2$ replaced by $\f12\dn^2$ on the right hand side. Then by using the continuity argument, we can  conclude the result of Theorem \ref{ext_stb}.
\end{remark}
The rest of this section is devoted to proving the above proposition. Compared with the proof of Proposition \ref{prop:BT:imp:Max}, the proof is more delicate due to the fact that the covariant derivative  does not commute with the Klein-Gordon operator $(\Box_A -1)$. The commutation  with  $D_Z^k$, $k\le 2$  will generate nonlinearities involving the full Maxwell field $F$.
These nonlinear error terms are controlled in the following crucial result.
\begin{lemma}\label{com_est}
Under the bootstrap assumptions \eqref{eq:BT:PWE:scal}-\eqref{eq:BT:energy}, there hold for $Z\in \Pp$
\begin{equation}\label{11.3.1.18}
\begin{split}
&\iint_{\D_{u_1}^{-u_2}} r^{1+\ep} u_+^{1+\ep} |(\Box_A -1)D_Z\phi|^2\\
 &\les (\E_{1,\ga_0}+\dn^4) (u_1)_+^{2\zeta(Z)+1-\ga_0}+\iint_{\D_{u_1}^{-u_2}}  u_+^{2\zeta(Z)}|D^{\leq 1}  \phi|^2+u_+^{2\zeta(Z)-2}(1+\zeta(Z))|D_{\Omega} \phi|^2,
\end{split}
\end{equation}
 as well as for $Z^2=Z_1 Z_2\in \Pp^2$
\begin{align}
\label{11.3.1.18:2}
&\iint_{\D_{u_1}^{-u_2}} r^{1+\ep} u_+^{1+\ep} |(\Box_A -1)D_Z^2\phi|^2\\
\notag
 &\les (\E_{2,\ga_0}+\dn^4) (u_1)_+^{2\zeta(Z^2)+1-\ga_0}+\iint_{\D_{u_1}^{-u_2}} u_+^{2\zeta(Z^2)}|D_{\Omega}^{\leq 1}D^{\leq 1}  \phi|^2+u_+^{2\zeta(Z^2)-2}|D_{\Omega}^{\leq 2} \phi|^2\\
 \notag
 &\quad +\iint_{\D_{u_1}^{-u_2}}  \sum\limits_{Z_1\sqcup Z_2=Z^2}\left( u_+^{2\zeta(Z_1)}|D^{\leq 1}  D_{Z_2}\phi|^2+u_+^{2\zeta(Z_1)-2}
 |D_{\Omega} D_{Z_2}\phi|^2\right)
\end{align}
for all $-v_*\le u_2<u_1\le-\frac{R}{2}$.
\end{lemma}

Similar to the treatment in Section \ref{ext_maxwell}, we focus on the nonlinear terms involved with $\hat F$ in the decomposition for the full Maxwell field $F$, with the remaining terms  controlled in exactly the same way as in \cite[Lemma 10 in Section 3.5]{KWY}. Indeed, by substituting the bootstrap assumptions (\ref{eq:BT:PWE:scal})-(\ref{eq:BT:energy}) and (\ref{10.30.5.18}) to \cite[Lemma 10 in Section 3.5]{KWY}, these remaining terms can be directly bounded by $(\E_{2,\ga_0}+\dn^4)(u_1)_+^{2\zeta(Z^k)+1-\ga_0}$ due to the smallness of the charge $q_0$.
\begin{proof}
According to Lemma \ref{lem::Commutator1}, we note that $(\Box_A-1)D_Z\phi=Q(F, \phi, Z)$ and $(\Box_A-1)D_{Z_1}D_{Z_2}\phi$ can be written as the sum of the same quadratic forms $Q(G, f, Z)$ and cubic terms, with the cubic term appearing only for the second order estimate (\ref{11.3.1.18:2}). The quadratic terms can be classified into two categories. Let us now focus on the first  type, which is
$$
\sum_{\zeta(Z^l X)=\zeta(Z^k), X\in \Pp, l< k}|Q(F, \phi_{, l}, X)|
$$
where
 $\phi_{, 0}=\phi$; and when $l=1$,  $\phi_{,l}$ is  $D_Z \phi$ with $Z^1=Z\in \{Z_1, Z_2\}$. Here for giving the above symbolic formula  we have used the fact that $\zeta([X, Y])=\zeta(XY)$ when $[X, Y]\neq 0$ for any $X, Y\in \Pp$. This quadratic term can be further decomposed into $Q(\hat{F}, \phi_{, l},  X)$ and $Q(F-\hat{F}, \phi_{, l}, X)$. The latter can be estimated in the same way as that in \cite[Lemma 10 in Section 3.5]{KWY}. It hence suffices to treat the former involving the large Maxwell field $\hat{F}$. In view of  (\ref{eq:Est4commu:1}), we have for any $X\in \Pp$
\begin{align*}
 |Q(\hat{F}, \phi_{,l}, X)|^2& \les r^{2\zeta(\tilde X)+2}(|\hat{\a}|^2|D \phi_{,l}|^2+|\ud{ \hat{F}}|^2|D_L \phi_{,l}|^2+|\hat{\si}|^2|\sl{D} \phi_{,l}|^2)\\
& \quad   +u_+^{2\zeta(X)+2}(|\hat{\rho}|^2|D_{\Lb}\phi_{,l}|^2 +|\hat{\ab}|^2|\sl{D} \phi_{,l}|^2)+r^{2\zeta(X)}| \hat{F}|^2|\phi_{,l}|^2,
\end{align*}
where $\ud {\hat{F}}$ represents all the components of the 2-form $\hat{F}$ except $\a[\hat{F}]$. Here we note that $\hat{F}$ verifies the linear Maxwell equation, hence the current terms in the last line of (\ref{eq:Est4commu:1}) vanish. Recall the following improved estimate for $D_L f$:
\begin{equation}\label{11.01.4.18}
r|D_L f|\les |D_{\Omega }f|+u_+|D f|,
\end{equation}
which will be used when $\zeta(X)=0$.
 Then by using the pointwise decay estimates for $\hat{F}$ in Proposition \ref{linear_ex}, in view of $1+2\ep<\ga_0<2$, we can estimate that
\begin{align*}
&\iint_{\D_{u_1}^{-u_2}} |Q(\hat{F}, \phi_{,l}, X)|^2 r^{1+\ep} u_+^{1+\ep}\\
&\les \iint_{\D_{u_1}^{-u_2}} u_+^{\ep} r^{2\zeta(X)+1+\ep}(r^{-\ga_0}|D \phi_{,l}|^2+u_+^{-\ga_0}|D_L \phi_{,l}|^2) +  r^{\ep-1} u_+^{2\zeta( X)+\ep-\ga_0}( u_+^2|\sl{D} \phi_{,l}|^2+|\phi_{,l}|^2) \\
&\les \iint_{\D_{u_1}^{-u_2}}  u_+^{2\zeta(X)}(|D  \phi_{,l}|^2+|\phi_{,l}|^2)+u_+^{2\zeta(X)-2}(1+\zeta(X))|D_{\Omega} \phi_{,l}|^2. 
\end{align*}
In particular this gives  \eqref{11.3.1.18}.  For $k=2$,  if $l=0$, then $X=\L_{Z_1} Z_2$ or $0$.
In this case, the bound on the right is as desired as in (\ref{11.3.1.18:2}).

 In view of (\ref{2.19.1.19}) in Lemma \ref{lem::Commutator1}, the second order estimate \eqref{11.3.1.18:2} involves the control of $Q(F, D_{Z_1}\phi, Z_2)$, $Q(F, D_{Z_2}\phi, Z_1)$ and $Q(F, \phi, \L_{Z_1} Z_2)$ as given above, the estimate of the second type of quadratic terms
\begin{equation}\label{2.19.2.19}
 Q(\L_{Z_1} F, \phi, Z_2)=Q(\L_{Z_1} \hat F, \phi, Z_2)+Q(\L_{Z_1} \tF, \phi, Z_2)
 \end{equation}
   and the cubic term in (\ref{2.19.1.19}). Note that $Q(\L_{Z_1} \tF, \phi, Z_2)$ can be similarly treated as in \cite[Lemma 10 in Section 3.5]{KWY}, except that the current terms in the last line of (\ref{eq:Est4commu:1}) involve the full Maxwell field as a factor. Hence, in view of the decomposition  (\ref{2.19.2.19}) and (\ref{2.19.1.19}), besides the leading term $\sum_{X\sqcup Y=Z^2}|Q(\L_Y \hat F, \phi, X)|$, we will need to bound the two types of cubic terms: (i) $\sum_{X\sqcup Y=Z^2}|F_{X\mu}F^{\mu}_{\ Y}\phi|$; (ii) the current terms in the last line of (\ref{eq:Est4commu:1}) for $\sum_{X\sqcup Y=Z^2}|Q(\L_Y \tF, \phi, X)|$.

For (i), we recall from \cite[Section 2.4]{KWY} that
\begin{equation*}
|F_{Y\mu} F_{\ X}^\mu|\les u_+^{\zeta(XY)}\left(u_+^2 |\ab|^2+r^2(|\sigma|^2+|\a|^2+|\rho|^2+|\a||\ab|)\right).
\end{equation*}
By using the pointwise decay properties for $\hat{F}$ in Proposition \ref{linear_ex} and the decay estimates for $\ck{F}$ in Proposition \ref{prop:decay:scal&Max}, we derive that
\begin{equation*}
|F_{Y\mu} F_X^\mu|^2\les u_+^{2\zeta(XY)} r^{-\ga_0} u_+^{-4+\ga_0}.
\end{equation*}
In particular we can bound
\begin{equation}\label{10.31.3.18}
\iint_{\D_{u_1}^{-u_2}} r^{1+\ep} u_+^{1+\ep} |F_{X\mu} \tensor{F}{^\mu_Y}|^2 |\phi|^2\les \iint_{\D_{u_1}^{-u_2}} u_+^{2\zeta(Z^2)-2+2\ep} |\phi|^2,
\end{equation}
where $\zeta(XY)=\zeta(Z^2)$ since $X\sqcup Y=Z^2$.
Substituting the bootstrap assumption on the energy flux of $\phi$ will not lead to any improvement as there is no additional smallness. The term is incorporated in the lower order terms in the bound of (\ref{11.3.1.18:2}).

For (ii), since $J[\L_Y F]=\L_Y J[\phi]$,  we  first treat the current terms in the third line of (\ref{eq:Est4commu:1}) with the help of  the formula for $\L_Y J_\mu$ in Lemma \ref{lem::Commutator1}.
Recall from \cite[Corollary 1 in Section 3.3]{KWY} that for  $\sI_Y[F]=r(|F_{L Y}|+|F_{AY}|)+u_+|F_{\Lb Y}|$, by using (\ref{10.31.2.18}), (\ref{10.30.5.18}) and Proposition \ref{prop:decay:scal&Max}, there holds
\begin{equation*}
\sI_Y^2[F] \les (1+ r^{2-\ga_0}u_+^{-1})u_+^{2\zeta(Y)},
\end{equation*}
which implies
\begin{align*}
\iint_{\D_{u_1}^{-u_2}}& u_+^{1+\ep+2\zeta(X)} r^{1+\ep} \sI_Y^2[F] |\phi|^4\\
&\les \iint_{\D_{u_1}^{-u_2}} \dn^2 (1+ r^{2-\ga_0}u_+^{-1} ) u_+^{2\zeta(XY)-\ga_0+1+\ep}r^{-3+1+\ep}|\phi|^2\\
&\les \dn^2 \iint_{\D_{u_1}^{-u_2}} u_+^{2\zeta(XY)}|\phi|^2.
\end{align*}
Thus, also by using Proposition \ref{prop:decay:scal&Max} we can complete the estimate by
\begin{align}
&\sum_{X\sqcup Y=Z^2}\iint_{\D_{u_1}^{-u_2}}r^{1+\ep}u_+^{1+\ep} |\phi|^2 \big(u_+^{2\zeta(X)+2}|\L_Y J_\Lb|^2+r^{2\zeta(X)+2}(|\L_Y J_L|^2+|\L_Y \slashed{J}|^2)\big)\nn\\
&\les \dn^2 \sum_{X\sqcup Y=Z^2}\iint_{\D_{u_1}^{-u_2}} r^{2\ep-\frac{5}{2}}u_+^{2\zeta(X)-2\ga_0-\f12+\ep}(|D_Y \phi|^2+ u_+ |D D_Y \phi|^2)+u_+^{2\zeta(Z^2)} |\phi|^2\nn\\
&\les\dn^2 \iint_{\D_{u_1}^{-u_2}}\sum_{X\sqcup Y=Z^2}u_+^{2\zeta(X)-2}|D\rp{\le 1} D_Y\phi|^2+u_+^{2\zeta(Z^2)} |\phi|^2\label{2.13.1.19}
\end{align}
with  the bound incorporated into the right hand side of  (\ref{11.3.1.18:2}).

Next, we consider the quadratic form $Q(\L_Y \hat F, \phi, X)$. Since the Maxwell field $\L_Y\hat{F}$ is large, substituting the pointwise decay estimates for the scalar field in Proposition \ref{prop:decay:scal&Max} directly will not give the desired result in (\ref{11.3.1.18:2}).  
As $\L_Y \hat{F}$ verifies the linear Maxwell equation, in view of  \eqref{eq:Est4commu:1} and using (\ref{11.01.4.18}) to bound  $D_L\phi$, we can derive that
\begin{align*}
 |Q(\L_Y\hat{F}, \phi, X)|^2& \les r^{2\zeta(X)+2}(|\a[\L_Y\hat F]|^2|D \phi|^2+|\L_Y  \hat{F}|^2|D_L \phi|^2+|\si[\L_Y\hat F]|^2|\sl{D} \phi|^2)\\
& \quad   +u_+^{2\zeta(X)+2}(|\rho[\L_Y\hat F]|^2|D_{\Lb}\phi|^2 +|\ab[\L_Y\hat F]|^2|\sl{D} \phi|^2)+r^{2\zeta(X)}| \L_Y\hat{F}|^2|\phi|^2\\
& \les u_+^{2\zeta(X)}(r_+^{2}|\L_Y\hat{\a}|^2+u_+^{2}|\L_Y\hat{F}|^2)|D \phi|^2+
u_+^{2\zeta(X)}|\L_Y  \hat{F}|^2|D_{\Omega}^{\leq 1} \phi|^2.
\end{align*}
Then by using the following Sobolev embedding on the 2-sphere  $S_{u,v}$
\begin{equation*}
r^\f12 \|f \|_{L^4(S_{u,v})}\les \|D\rp{\le 1}_\Omega f\|_{L^2(S_{u,v})}
\end{equation*}
as well as the $L^4$ estimates for the linear Maxwell fields in \eqref{11.02.2.18} and H\"{o}lder's inequality, we can estimate that
\begin{align*}
  &u_+^{-2\zeta(X)}\|Q(\L_Y\hat{F}, \phi, X)\|_{L^2(S_{u, v})}^2\\
  &\les  \|r_+|\a[\L_Y\hat F]|+u_+|\L_Y\hat{F}|\|_{L^4(S_{u, v})}^2\|D\phi\|_{L^4(S_{u, v})}^2 +
  \|\L_Y  \hat{F}\|^2_{L^4(S_{u, v})}\|D_{\Omega}^{\leq 1} \phi\|^2_{L^4(S_{u, v})}\\
  &\les u_+^{2\zeta(Y)-1}r^{-\ga_0}\|D_{\Omega}^{\leq 1}D\phi\|_{L^2(S_{u, v})}^2+
   u_+^{2\zeta(Y)-1-\ga_0}r^{-2}\|D_{\Omega}^{\leq 2}\phi\|_{L^4(S_{u, v})}^2.
\end{align*}
 Therefore
 \begin{align*}
   &\iint_{\D_{u_1}^{-u_2}} r^{1+\ep} u_+^{1+\ep} |Q(\L_Y \hat{F}, \phi, X)|^2\les   \iint_{\D_{u_1}^{-u_2}} u_+^{2\zeta(XY)+2\ep+1-\ga_0} (|D_{\Omega}^{\leq 1}D\phi|^2+  u_+^{-2} |D_{\Omega}^{\leq 2}\phi|^2).
 \end{align*}
 Here again we note that $1+2\ep<\ga_0<2$. Combining  the above estimate with estimates (\ref{10.31.3.18}) and (\ref{2.13.1.19}),  we conclude that estimate \eqref{11.3.1.18:2} holds.
\end{proof}

\begin{proof}[ Proof of Proposition \ref{11.4.3.18}]
Applying  the energy identity in \cite[Lemma 2(2) in Section 2.2]{KWY} to $(D_Z^k\phi, 0)$ with $k\le 2$, we have the energy estimate
\begin{equation*}\label{11.8.6.18}
\begin{split}
&E[D_Z^k \phi](\H_{u_1}^{-u_2})+E[D_Z^k\phi](\Hb^{u_1}_{-u_2})\\
&\les \E_{k, \ga_0}(u_1)_+^{-\ga_0+2\zeta(Z^k)} +\iint_{\D_{u_1}^{-u_2}}\big( |F_{0\mu} J[D^k_Z \phi]^\mu| +|(\Box_A-1)D_Z^k \phi||D_0 D_Z^k \phi|\big).
\end{split}
\end{equation*}
Now by using the pointwise decay estimates for the Maxwell field in Proposition \ref{linear_ex} and Proposition \ref{prop:decay:scal&Max}, we can estimate that
\begin{align*}
  &\iint_{\D_{u_1}^{-u_2}}\big( |F_{0\mu} J[D^k_Z \phi]^\mu| +|(\Box_A-1)D_Z^k \phi||D_0 D_Z^k \phi|\big)\\
  &\les \iint_{\D_{u_1}^{-u_2}}  (r^{-2}+r^{-1}u_+^{-\frac{1+\ga_0}{2}})|D_Z^k\phi||DD_Z^k\phi| +|(\Box_A-1)D_Z^k \phi||D D_Z^k \phi| \\
  &\les \iint_{\D_{u_1}^{-u_2}}  (r^{-3+\ep}+r^{\ep-1}u_+^{-1-\ga_0})|D_Z^k\phi|^2 +r^{1+\ep}|(\Box_A-1)D_Z^k \phi|^2 +r^{-1-\ep}|DD_Z^k\phi|^2.
\end{align*}
In particular we conclude that
\begin{align*}
  E[D_Z^k \phi](\H_{u_1}^{-u_2})+E[D_Z^k\phi]&(\Hb^{u_1}_{-u_2})\les  \E_{k, \ga_0}(u_1)_+^{-\ga_0+2\zeta(Z^k)} +\iint_{\D_{u_1}^{-u_2}}   r^{1+\ep}|(\Box_A-1)D_Z^k \phi|^2 \\ &  +\int_{u_2}^{u_1}u_+^{-1-\ep}E[D_Z^k\phi](\H_{u}^{-u_2})du+\int_{-u_1}^{-u_2}r^{-1-\ep}E[D_Z^k\phi](\Hb_{v}^{u_1})dv.
\end{align*}
By using Gronwall's inequality, we derive that
\begin{equation}
\label{1.27.2.18}
  E[D_Z^k \phi](\H_{u_1}^{-u_2})+E[D_Z^k\phi](\Hb^{u_1}_{-u_2})\les  \E_{k, \ga_0}(u_1)_+^{-\ga_0+2\zeta(Z^k)} +\iint_{\D_{u_1}^{-u_2}}   r^{1+\ep}|(\Box_A-1)D_Z^k \phi|^2 .
\end{equation}
In particular when $k=0$, the energy decay estimate \eqref{eq:BT:imp:scal:energy} holds for $k=0$ as $\Box_A\phi=\phi$.

Next we consider the weighted energy bounds in \eqref{eq:BT:imp:scal:PWE}.
Recall from \cite[Section 2.2 and 3.5]{KWY} the following weighted energy estimate
\begin{align*}
  &\int_{\H_{u_1}^{-u_2}}r|D_L D_Z^k\phi|^2 +\int_{\Hb_{-u_2}^{u_1} }r(|\slashed{D} D_Z^k\phi|^2+|D_Z^k\phi|^2)\\
  &\les \E_{k,\ga_0} (u_1)_+^{1-\ga_0+2\zeta(Z^k)}+\iint_{\D_{u_1}^{-u_2}} |D_Z^k \phi|^2+ r |F_{L \mu} J[D_Z^k\phi]^\mu|+|(\Box_A-1)D_Z^k \phi||D_L (r D_Z^k \phi)|.
\end{align*}
Note that (\ref{10.31.2.18}) and (\ref{eq:decay:Max:checkF})  imply that
\begin{equation*}
|\rho|\les r^{-2}+r^{-1-\frac{\ga_0}{2}}u_+^{-\f12}, \quad |\a|\les r^{-1-\frac{\ga_0}{2}}u_+^{-\f12}.
\end{equation*}
 In view of the above estimates and by using the Cauchy-Schwarz inequality we can show that,  with $ D^\sharp$ being either $D_L $ or $\sl{D}$,
\begin{align*}
  &\iint_{\D_{u_1}^{-u_2}} r |F_{L \mu} J[D_Z^k\phi]^\mu|+|(\Box_A-1)D_Z^k \phi||D_L (r D_Z^k \phi)| \\
  &\les \iint_{\D_{u_1}^{-u_2}}   r^{-1}|D_Z^k\phi||D_L D_Z^k\phi|+r^{-\frac{\ga_0}{2}}u_+^{-\f12}|D_Z^k\phi|| D^\sharp D_Z^k\phi|+|(\Box_A-1)D_Z^k \phi||D_L (r D_Z^k \phi)|\\
  &\les \iint_{\D_{u_1}^{-u_2}}   u_+^{-1-\ep}(|D_Z^k\phi|^2+|\slashed{D}D_Z^k\phi|^2+r|D_LD_Z^k\phi|^2)+ru_+^{1+\ep}|(\Box_A-1)D_Z^k \phi|^2.
\end{align*}
 The first two terms can be estimated by energy flux on $\H_u^{-u_2}$ with $u_2< u\le u_1$. The third term can be absorbed by using Gronwall's inequality. Therefore we conclude that
 \begin{align}
 \label{eq:BT:imp:scal:PWE:0}
   &\int_{\H_{u_1}^{-u_2}}r|D_L D_Z^k\phi|^2 +\int_{\Hb_{-u_2}^{u_1} }r(|\slashed{D} D_Z^k\phi|^2+|D_Z^k\phi|^2)\les \E_{k,\ga_0} (u_1)_+^{1-\ga_0+2\zeta(Z^k)}\nn\\
   &+(u_1)_+^{1-\ga_0}\sup_{u_2\le u\le u_1}u_+^{\ga_0}E[D_Z^k \phi](\H_u^{-u_2})+\iint_{\D_{u_1}^{-u_2}} r^{1+\ep}u_+^{1+\ep}|(\Box_A-1)D_Z^k \phi|^2.
 \end{align}
 In particular with $k=0$ in (\ref{eq:BT:imp:scal:PWE:0}),  since \eqref{eq:BT:imp:scal:energy} is already proved for $k=0$,  \eqref{eq:BT:imp:scal:PWE} holds for $k=0$ as $\Box_A\phi=\phi$. Therefore Proposition \ref{11.4.3.18} is proved for the case when $k=0$.

 Next we prove Proposition \ref{11.4.3.18} for the case when $k=1$ by using Lemma \ref{com_est}. In view of \eqref{1.27.2.18} and  \eqref{11.3.1.18}, we derive that
 \begin{align*}
  E[D \phi](\H_{u_1}^{-u_2})+E[D\phi](\Hb^{u_1}_{-u_2})\les ( \E_{1, \ga_0}+\dn^4)(u_1)_+^{-\ga_0-2} +(u_1)_+^{-1-\ep}\iint_{\D_{u_1}^{-u_2}}   u_+^{-2}|D^{\leq 1}\phi|^2,
\end{align*}
where we used the fact that $\zeta(\partial)=-1$.
 By using Gronwall's inequality and the energy decay estimate for $\phi$ in \eqref{eq:BT:imp:scal:energy} with $k=0$, we conclude that  \eqref{eq:BT:imp:scal:energy} holds for $k=1$ and $Z=\partial$.  \eqref{1.27.2.18} and  \eqref{11.3.1.18} also imply that
  \begin{align*}
  E[D_{\Omega} \phi](\H_{u_1}^{-u_2})+E[D_{\Omega}\phi](\Hb^{u_1}_{-u_2})&\les ( \E_{1, \ga_0}+\dn^4)(u_1)_+^{-\ga_0}\\
  &+(u_1)_+^{-1-\ep}\iint_{\D_{u_1}^{-u_2}}(|D^{\leq 1}\phi|^2+u_+^{-2}|D_{\Omega}\phi|^2).
\end{align*}
 Then by using  \eqref{eq:BT:imp:scal:energy} for $k=0$ and for $D\phi$ with $k=1$, together with Gronwall's inequality, we conclude that \eqref{eq:BT:imp:scal:energy} holds for $k\leq 1$, $Z\in \Pp$. As a byproduct, we also have the improved estimate for (\ref{11.3.1.18})
 \begin{align*}
   \iint_{\D_{u_1}^{-u_2}} r^{1+\ep} u_+^{1+\ep} |(\Box_A -1)D_Z\phi|^2\les (\E_{1,\ga_0}+\dn^4) (u_1)_+^{2\zeta(Z)+1-\ga_0}.
 \end{align*}
 This estimate together with \eqref{eq:BT:imp:scal:PWE:0} implies that  \eqref{eq:BT:imp:scal:PWE} holds for $k\leq 1$.


Finally we prove Proposition \ref{11.4.3.18} for $k=2$ based on Lemma \ref{com_est} and the proven estimates in \eqref{eq:BT:imp:scal:energy} and  \eqref{eq:BT:imp:scal:PWE} with $k\leq 1$. The idea is to consider the $|u|$-weighted energy fluxes along the outgoing null hypersurface. Since $u\equiv u_1$ along $\H_{u_1}^{-u_2}$, from \eqref{1.27.2.18} and  \eqref{11.3.1.18:2}, we obtain that
 \begin{align*}
  &\sum\limits_{X, Y\in \Pp} E[u_+^{-\zeta(XY)}D_X D_Y \phi](\H_{u_1}^{-u_2})\\
  &\les (\E_{2, \ga_0}+\dn^4)(u_1)_+^{-\ga_0} +\sum_{Z\in \Pp}\iint_{\D_{u_1}^{-u_2}} u_+^{-2}\big( |D_{\Omega}^2 \phi|^2 +   |u_+^{1-\zeta(Z)} D  D_{Z}\phi|^2+
 |u_+^{-\zeta(Z)}D_{\Omega} D_{Z}\phi|^2\big)\\
  &\les (\E_{2, \ga_0}+\dn^4)(u_1)_+^{-\ga_0} +\sum\limits_{X', Y'\in \Pp} \iint_{\D_{u_1}^{-u_2}} u_+^{-2} |u_+^{-\zeta(X'Y')}D_{X'} D_{Y'}\phi|^2,
 \end{align*}
 where the lower order terms have been estimated by using \eqref{eq:BT:imp:scal:energy} when $k\leq 1$.
 By using Gronwall's inequality, we  derive that
 \begin{align*}
   E[ D_X D_Y \phi](\H_{u_1}^{-u_2})\les (\E_{2, \ga_0}+\dn^4)(u_1)_+^{2\zeta(XY)-\ga_0},\quad \forall\, X, Y\in \Pp.
 \end{align*}
Since this estimate holds for any $-v_*\le u_2< u_1\le -\frac{R}{2}$, this  estimate improves  estimate (\ref{11.3.1.18:2}) to
 \begin{align*}
&\iint_{\D_{u_1}^{-u_2}} r^{1+\ep} u_+^{1+\ep} |(\Box_A -1)D_X D_Y\phi|^2
 \les (\E_{2,\ga_0}+\dn^4) (u_1)_+^{2\zeta(XY)+1-\ga_0}.
\end{align*}
By using the above estimate,  the estimates \eqref{eq:BT:imp:scal:energy} and \eqref{eq:BT:imp:scal:PWE} when $k=2$  follow from   \eqref{1.27.2.18} and \eqref{eq:BT:imp:scal:PWE:0} respectively. We finished the proof for Proposition \ref{11.4.3.18}.
\end{proof}

\appendix
\section{Commutators}
	
	In this section, we list expressions for the commutators that will be used. Calculations for these are standard and can be found in \cite{Yang2015, KWY}. For any complex scalar field $f$, closed 2-form $G$, and vector field $Z$, define the quadratic form
\begin{equation}\label{comm_1}
Q(G, f, Z)=2iZ^\nu G_{\mu\nu}D^\mu f + i\partial^\mu(Z^\nu G_{\mu\nu})f.
\end{equation}
We have the following commutator formulae.
	\begin{lemma}
\label{lem::Commutator1}
		Let $(\phi, F)$ be the solution of MKG equations, and let $X, Y, Z\in \Pp$. 
Then the following commutator formulae hold
\begin{align}
				(\Box_A-1)D_Z\phi&=[\Box_A, D_Z]\phi = Q(F, \phi, Z), \nn\\
    (\Box_A-1) D_X D_Y \phi&=Q(F, D_X\phi; Y)+Q(F, D_Y\phi, X)+Q(\L_X F,\phi, Y)\nn\\
    &+ Q(F, \phi, [X, Y])-2 F_{X\mu}F^{\mu}_{\  Y}\phi,\label{2.19.1.19}\\
				\partial^\mu(\L_Z^k \tF)_{\mu\nu} &= (\L_Z^k  J[\phi])_\nu.\nn
		\end{align}
Moreover we can expand the Lie derivative on the 1-form $J[\phi]$
\begin{equation*}
  \begin{split}
    \L_Z J_\mu[\phi] &= \Im(D_Z\phi\cdot \overline{D_\mu\phi}) + \Im(\phi\cdot\overline{D_\mu D_Z\phi}) - F_{Z\mu}|\phi|^2,\\
    \L_X\L_Y J_\mu[\phi] &= \Im(D_X D_Y\phi\cdot\overline{D_\mu \phi}) + \Im(D_Y\phi\cdot(\overline{D_\mu D_X\phi +iF_{X\mu}\phi})) + \Im(D_X\phi\cdot\overline{D_\mu D_Y\phi})\\
				& +\Im(\phi\cdot(\overline{D_\mu D_X D_Y \phi + i F_{X\mu}D_Y\phi})) - (\L_XF)_{Y\mu}|\phi|^2 - F_{[X,Y]\mu}|\phi|^2 - F_{Y\mu}D_X|\phi|^2.
  \end{split}
\end{equation*}
In particular, for any smooth vector field $X$, we have the bound with $k\le 2$
\begin{footnote}{We cancel out $Z^0$  when it appears in ``$\cdots\sqcup Z^0\cdots$", since $Z^0=\mbox{id}$. }\end{footnote}
\begin{align*}
  |X^\mu \L_Z^k J_\mu[\phi]|&\les \sum\limits_{l_1+l_2\leq k}|D_Z^{l_1}\phi||X^\mu D_\mu D_{Z}^{l_2}\phi|+|\phi|^2 |X^\mu F_{[Z^{l_1}, Z^{l_2}]\mu}|\\
  &+\sum\limits_{l_1+l_2\leq k-1, Z\sqcup Z^{l_1}\sqcup Z^{l_2}=Z^k} |X^\mu (\L_Z^{l_1}F)_{Z\mu}||\phi| |D_Z^{l_2}\phi|.
\end{align*}
\end{lemma}
\begin{lemma}\label{lem:Est4commu:1}
{\it
For any $2$-form $G=(\a, \rho, \si, \underline{\a})$    and any complex scalar field $\phi$, in the exterior region $\{t-t_0+R\leq r\}$ we have
\begin{equation}
\label{eq:Est4commu:1}
\begin{split}
|Q(G, \phi, Z)|&\les r^{\zeta(Z)+1}(|\a||D\phi|+|\ud G||D_L\phi|+|\si||\sl{D}\phi|) \\
& \quad\, +u_+^{\zeta(Z)+1}(|\rho||D_{\Lb}\phi|+|\ab||\sl{D}\phi|)\\
&\quad \,+\left(u_+^{\zeta(Z)+1}|J_{\Lb}|+r^{\zeta(Z)+1}(|J_L| +|\slashed{J}|)+r^{\zeta(Z)}|G|\right)|\phi|
\end{split}
\end{equation}
for all $Z\in \Pp$, where  $\ud G$ denotes all the components of $G$ except $\a[G]$.   Here  $(J_L, J_{\Lb}, \slashed{J})$ with $\slashed{J}=(J_{e_1}, J_{e_2})$  are  the  components of  the  null decomposition  of  the current $J_\nu=-\p^\mu G_{\mu\nu}$.
}
\end{lemma}	

\end{document}